\documentclass[11pt]{amsart}
\usepackage{amsmath,amsthm, amscd, amssymb, amsfonts,color,mathtools}
\usepackage{hhline}
\usepackage[active]{srcltx}
\usepackage{mathrsfs}
\usepackage{enumitem}
\usepackage{color}
\usepackage[all]{xy}
\usepackage[normalem]{ulem}

\newtheorem{lema}{Lemma}[section]
\newtheorem{teor}[lema]{Theorem}
\newtheorem{cor}[lema]{Corollary}
\newtheorem{prop}[lema]{Proposition}

\theoremstyle{definition}
\newtheorem{defi}[lema]{Definition}
\newtheorem{exemplo}[lema]{Example}
\newtheorem{obs}[lema]{Remark}

\newcommand{\com}{\Delta}
\newcommand{\eps}{\varepsilon}
\newcommand{\gr}{\operatorname{gr}}

\newcommand{\cS}{\mathcal{S}}
\newcommand\co{\operatorname{co}}

\newcommand\cop{\operatorname{cop}}
\newcommand\id{\operatorname{id}}

\newcommand\Img{\operatorname{Img}}
\newcommand\Res{\operatorname{Res}}

\newcommand\Lie{\operatorname{Lie}}
\newcommand{\SL}{\operatorname{SL}}
\newcommand{\SO}{\operatorname{SO}}
\newcommand{\Z}{{\mathbb Z}}
\newcommand{\N}{{\mathbb N}}
\newcommand{\C}{{\mathbb C}}
\newcommand{\Q}{{\mathbb Q}}

\newcommand{\ttn}{{\mathtt{n}}}

\newcommand{\cO}{\mathcal{O}}
\newcommand{\Oc}{{\mathcal O}}

\newcommand{\ydga}{{}^{{\Gamma}}_{{\Gamma}}\mathcal{YD}}
\newcommand{\ydh}{{}^H_H\mathcal{YD}}

\newcommand\toba{{\mathfrak B }}

\newcommand\cD{\mathcal D }
\newcommand{\lieu}{\mathfrak{u}}
\newcommand{\lieg}{\mathfrak{g}}
\newcommand{\lieb}{\mathfrak{b}}
\newcommand{\liesl}{\mathfrak{sl}}
\newcommand{\lieso}{\mathfrak{so}}
\newcommand{\lieh}{\mathfrak{h}}
\newcommand{\liel}{\mathfrak{l}}
\newcommand{\mcar}{C}
\newcommand{\subg}{T_M}

\newcommand\Hom{\operatorname{Hom}}

\newcommand\Ker{\operatorname{Ker}}

\newcommand{\ad}{\operatorname{ad}_c}

\newcommand{\ord}{\operatorname{ord}}

\def\pf{\begin{proof}}
\def\epf{\end{proof}}
\def\ot{\otimes}

\begin{document}
\title[A geometric realization of liftings]
{A geometric realization of liftings of Cartan type}

\author[ D. Bagio, G. A. Garc\'ia, O. M\'arquez]
{D. Bagio, G. A. Garc\'ia, O. M\'arquez}

\thanks{
D. Bagio was partially supported by Fundação de Amparo a 
Pesquisa e Inovação do Estado de Santa Catarina (FAPESC), Edital 21/2024.
G. A. Garc\'ia was partially supported by
the Program CAPES/PRINT (Brasil) and CONICET (Argentina).
}

\address{\noindent D. B. and O. M. : Departamento de Matem\'atica, 
	Universidade Federal de Santa Catarina,
	Santa Catarina, Brazil}
	\email{d.bagio@ufsc.br} 
	\email{oscar.marquez.sosa@ufsc.br}

\address{\noindent G. A. G. : Departamento de Matem\'atica, Facultad de Ciencias Exactas,
Universidad Nacional de La Plata. CMaLP, CONICET, (1900)
La Plata, Argentina.\\
Guangdong-Technion Israel Institute of Technology, No. 241, Daxue Road, 
Shantou, Guangdong Province, China.}
\email{ggarcia@mate.unlp.edu.ar,
gaston.garcia@gtiit.edu.cn}

\begin{abstract} 
We introduce a novel approach to compute liftings of bosonizations 
of Nichols algebras of diagonal braided vector spaces of Cartan type which replaces heavy computations with structural maps related to quantum groups.
This provides an answer to a question posed by Andruskiewitsch and Schneider in \cite{AS3},
who classified finite-dimensional complex pointed Hopf algebras over finite 
abelian groups whose order is coprime with 210.
As application and in order to give not-too-technical examples,
we recover with our method the liftings of type $A_{n}$
computed in \cite{AS2}, of type $B_2$ computed in \cite{BDR}, and of type $B_3$ computed in \cite{BGM} for Drinfeld-Jimbo type braidings. Moreover,
we present all liftings of type $B_{\theta}$ and $D_{\theta}$, for $\theta \geq 2$, giving in this way new explicit infinite families of liftings for
Drinfeld-Jimbo type braidings.
\end{abstract}

\maketitle

\section{Introduction}\label{sec:introduction}
The present work is a contribution to the classification problem of complex finite-dimensional pointed Hopf algebras. Hopf algebras are unital associative 
algebras equipped with a compatible
structure of coalgebra and a distinguished linear map called antipode.
They appear naturally in many areas 
of mathematics and theoretical physics, as they encode (classical and quantum) symmetries of geometric objects through their category of representations. 
Basic examples include group algebras, functions algebras over affine algebraic groups, 
universal enveloping algebras of Lie algebras and their quantum versions.

\subsection*{Background}
The classification problem of general complex finite-dimensional 
Hopf algebras is a hard question that has attracted the attention of several mathematicians 
with different backgrounds. 
Its difficulty may lie in the existence of few 
general structural results; progress has been made primarily for Hopf algebras of a fixed dimension (a problem posed by Kaplansky back in 1975), for Hopf algebras whose dimension 
factorizes as a product of few primes, e.g. 
$p$, $p^{2}$ or $pq$, or for Hopf algebras with additional structure, such as being semisimple as algebras, or pointed as coalgebras.
Despite these efforts, a significant amount of questions remain open; for example, the
least dimension for which the classification is still open is 24, 
and only Hopf algebras of dimension $p$ or $p^{2}$ are classified, 
see \cite{BGa} for details and references.
The study of semisimple Hopf algebras has contributed to the development 
of the theory
of fusion categories, and several results have been obtained by applying it to categories
of finite-dimensional representations of a Hopf algebra. 
The family of pointed Hopf algebras is perhaps the best understood in terms of structural results, due to 
the existence of a general technique for their study: the \emph{lifting method} developed by Andruskiewitsch and Schneider \cite{AS1}

A classical results in the theory is the celebrated theorem due
to Milnor, Moore, Cartier, Kostant and Gabriel, which states that any cocommutative
complex Hopf algebra is a semidirect product of the universal
enveloping algebra of a Lie algebra with a group algebra. Cocommutative Hopf 
algebras are particular examples of pointed Hopf algebras. These satisfy that 
all simple subcoalgebras are one-dimensional, or equivalently, that their coradical,
which is the sum of all simple subcoalgebras, 
is a group algebra. In particular, the dual of a finite-dimensional pointed Hopf algebra
is a basic algebra. Non-trivial and paradigmatic examples of pointed Hopf algebras include quantum universal enveloping algebras $U_{q}(\lieg)$, 
which are (multiparametric) 
deformations of the universal enveloping algebras $U(\lieg)$ of a semisimple 
Lie algebra $\lieg$, and finite-dimensional quotients of them known as 
small quantum groups or Frobenius-Lusztig kernels $u_{q}(\lieg)$, with $q$ a root of unity.

\subsection*{The lifting method}
The lifting method is a procedure to classify pointed Hopf algebras, or more generally,
Hopf algebras with the \emph{Chevalley Property}, that is, those whose 
coradical is a Hopf subalgebra.  Roughly speaking, 
the method proceeds as follows. Let $A$ be a complex pointed Hopf algebra 
and denote by $A_{0}$ its coradical. Since $A_{0}$ is a Hopf subalgebra, the 
coradical filtration is a Hopf algebra filtration and whence the associated graded 
coalgebra $\gr A$ is a Hopf algebra. The projection onto the zeroth component 
$\gr A \twoheadrightarrow A_{0}$ admits a Hopf algebra section, and by a result of
Majid and Radford, see \cite[Theorem 11.7.1]{Ra}, 
the Hopf algebra $\gr A$ is isomorphic to the 
\emph{biproduct} or \emph{bosonization} $R\# A_{0}$. As algebra, the bosonization
is given by the semidirect product, so this result resembles the classical one on 
cocommutative Hopf algebras. It turns out that $R$ is not a usual Hopf algebra but
a Hopf algebra in the braided category 
$_{A_{0}}^{A_{0}}\mathcal{YD}$ of Yetter-Drinfeld modules over $A_{0}$. Hence, 
assuming that the structure of the (cosemisimple) Hopf algebra given by the coradical $A_{0}$ is already known, the goal is to describe 
all possible braided Hopf algebras $R$ and all possible Hopf algebras $L$ such that
$\gr L$ is isomorphic to $R\# A_{0}$. These Hopf algebras $L$ are 
called the \emph{liftings} of $R\# A_{0}$. \smallbreak

The case where most results are known is that of pointed Hopf algebras, 
and specially, when the coradical $A_{0}$ is a group algebra $\C \Gamma$
of some abelian group $\Gamma$
and $A$ is finite-dimensional, see \cite{An3} for an up-to-date survey.
In this \emph{abelian} case, all braided finite-dimensional
Hopf algebras are Nichols algebras $\toba(V)$ over braided vector 
spaces $V \in\, _{\C\Gamma}^{\C\Gamma}\mathcal{YD}$ with diagonal braiding. Finite-dimensional
 Nichols algebras of this type were classified by Heckenberger \cite{H} and this 
 allowed Andruskiewitsch and Schneider to prove in \cite{AS3} the celebrated classification theorem on finite-dimensional complex pointed Hopf algebras over finite abelian groups
 whose order is not divisible by 210, see Theorem \ref{thm:AS}.
 This condition on the group order ensures that
 the Nichols algebras $\toba(V)$ appearing as the braided Hopf algebras $R$ in
 the bosonization are determined by finite Cartan matrices corresponding to complex
 simple Lie algebras. The theorem referred to above actually describes all possible 
 liftings $L$ of the bosonizations $\toba(V)\# \C\Gamma$ via recursive
 formulas: any such a Hopf algebra is isomorphic to a Hopf algebra 
 $u(\cD, \lambda, \mu)$ where $\cD$ is a Cartan datum related to the Cartan matrix
 and a principal realization over the group $\Gamma$, and $\lambda$, $\mu$
 are families of parameters giving the deformations, see Subsection \ref{subsec:pres-u}. 
 The main result of this paper provides a method to obtain all liftings
 without recursion and by using short exact sequences and commutative 
 diagrams involving several types of quantum groups.
 
 \smallbreak
 Let us describe the situation more technically and explicitly. 
 As we explained above, any finite-dimensional complex pointed Hopf algebra $A$
 satisfies that its associated graded Hopf algebra $\gr A$ is isomorphic to the bosonization
 $R\# A_{0}$ where the coradical $A_{0}=\C\Gamma$ is a group algebra of a finite group
 and $R$ is a Hopf algebra in the braided category 
$_{\C\Gamma}^{\C\Gamma}\mathcal{YD}$ of Yetter-Drinfeld modules. 
It turns out that $R$ is also a graded braided Hopf algebra with $R^{0}=\C$.
The braided vector space $V=R^{1}$ consisting of the elements of degree one is 
called the infinitesimal braiding, and the subalgebra $\toba(V)$ of $R$ generated by it,
the \emph{Nichols algebra} of $V$. Therefore, to classify all possible pointed Hopf 
algebras over a fixed group $\Gamma$ one has to 
\begin{enumerate}
\item[$\rm (a)$] \emph{(Nichols algebras)} determine all possible
braided vector spaces $V\in\, _{\C\Gamma}^{\C\Gamma}\mathcal{YD} $
such that $\dim \toba(V) $ is finite;
\item[$\rm (b)$] \emph{(Liftings)} find all Hopf pointed algebras $L$ such that 
$\gr L \simeq \toba(V)\# \C\Gamma$;
\item[$\rm (c)$] \emph{(Generation in degree one)} prove that any finite dimensional
braided Hopf algebra $R$ is actually isomorphic to a Nichols algebra
$\toba(V)$ for some $V\in \, _{\C\Gamma}^{\C\Gamma}\mathcal{YD} $.
\end{enumerate}

From now on, we assume that the group $\Gamma$ is finite abelian and 
its order is coprime to 210. 
In such a case, 
all braided vector spaces $V\in\, _{\C\Gamma}^{\C\Gamma}\mathcal{YD} $ are of
diagonal type  and $(a)$ follows by Heckenberger's work on arithmetic root 
systems. Moreover, one can describe 
all possible Nichols algebras explicitly by generators and relations using finite Cartan matrices.
Such braided vector spaces are said to be of \emph{Cartan type}, see Subsection
\ref{sec-datum} for more details. Furthermore, relations only involve powers of root 
vectors and quantum Serre relations, see \cite[Theorem 5.1]{AS3}.
Another main result of \emph{loc.cit.} is the proof of (c), which actually states that 
this family of Hopf algebras is generated by group-like and skew-primitive elements. 
Assuming that the Cartan matrix is indecomposable, 
to prove (b) one has to find all possible deformations of powers of root 
vectors relations and quantum Serre relations. It turns out that the latter are not deformed,
so the problem reduces to find explicit formulas for 
powers of root vectors; these are of the form 
$$
x_{\alpha}^{N}=r_{\alpha}(\mu),
$$
with $N$ a natural number (the order of a root of unity which is 
coprime to 210), positive roots $\alpha$ associated 
with the root system corresponding to the Cartan matrix and $\mu$ a family
of parameters satisfying certain conditions.
The root vectors
$x_{\alpha}$ are obtained via iterated braided
commutators of the linear generators $\{x_{i}\}_{i\in I}$ of $V$, and
$r_{\alpha}(\mu)$ are elements in the augmentation ideal of 
$\C[g_{i}^{N_{i}}\, |\, i\in I]\subset \C\Gamma$ 
that are constructed in a recursive way as
$$
r_{\alpha}(\mu) = \mu_{\alpha} (1 - g_{\alpha}^{N}) + \sum_{\beta,\gamma\neq 0, \beta+\gamma = \alpha}
t_{\beta,\gamma}^{\alpha} \mu_{\beta}r_{\gamma}(\mu),
$$
where $t_{\beta,\gamma}^{\alpha}\in \C$ are determined
by the coproduct. The family of pointed Hopf algebras obtained as liftings are 
denoted by $u(\cD,\mu)$, where $\cD$ is a so-called Cartan datum.

\smallbreak
In \emph{\cite{AS3}, the authors posed the problem of finding an explicit algorithm 
describing the elements $r_{\alpha}(\mu)$ for any indecomposable
Cartan matrix.}
Although the description may seem straightforward, it is actually quite difficult to write down the elements $r_{\alpha}(\mu)$ precisely. Up to now,
this has been carried out for these types of braidings
in \cite{AS2} for type $A_{n}$, $n\geq 1$, in \cite{BDR} for type $B_{2}$
and in \cite{BGM} for type $B_{3}$. 
All these results rely on heavy and \emph{ad-hoc} computations, which shows that alternative techniques are necessary to describe the remaining cases. \smallbreak

In another direction, Masuoka showed in \cite{Ma} that all liftings $L$ are cocycle deformations of the bosonization $\toba(V)\# \C\Gamma$.
This led to the idea of describing all liftings via cocycle deformations 
using Hopf-Galois objects \cite{AAGMV}. This approach applies
in more general settings, whenever all liftings $L$ can be obtained 
by cocycle deformations. Actually, 
it is conjectured that this is always the case, 
as confirmed by all known examples. Specifically, it
was proven for the family of complex pointed Hopf algebras 
over abelian groups whose braiding is not necessarily of Cartan type \cite{AGI}, and some families of 
finite-dimensional pointed Hopf algebras over non-abelian groups like
the dihedral ones, see for example \cite{GM}.  
In \cite{AAG} the technique using Hopf-Galois objects was applied in 
the case of 
diagonal braidings of Cartan type $A$, and in \cite{GIG} to 
Cartan type $G_{2}$. 
In the latter, the quantum Serre relations may also be deformed, and the corresponding computations become rather intricate.
Actually,
in \cite{AGI} it was shown that this method yields all possible liftings of Nichols algebras of diagonal type.
It is worth emphasizing two key features of the procedure: first, it is not necessary to determine explicitly the cocycles involved in the deformation; 
second, the outcome is independent of the choice of the group $\Gamma$
as far as one has a \emph{principal realization}.
Nevertheless, in order to find the Hopf-Galois objects corresponding to the cocycles, 
one must introduce a stratification of the defining relations of the Nichols
algebras, and then compute the associated sections, which might demand a lot 
of computations as shown in the $G_{2}$-case.

\subsection*{Main results}
In this paper, we introduce a novel approach that replaces 
heavy computations with the use of structural maps from quantum groups. 
The deforming relations are given by inclusions of finite abelian 
groups in the standard Borel subgroup of a 
simple Lie group corresponding to the Cartan matrix associated with the braiding.
The liftings are then obtained by means of exact sequences and are shown to be isomorphic
to finite-dimensional quotients of quantized function algebras, the
so-called \emph{quantum subgroups}. This gives a
\emph{geometric} interpretation of the deformation, for any complex 
finite-dimensional pointed Hopf algebra over such an abelian group is obtained 
as an extension of a small quantum group by a finite abelian subgroup of a Borel 
subgroup. \smallbreak

Let $V$ be a finite-dimensional
braided vector space of diagonal Cartan type associated with a finite, indecomposable symmetrizable 
Cartan matrix $C=(a_{ij})_{1\leq i,j\leq \theta}$ and 
denote the braiding by $c:V\ot V \to V\ot V$. 
Let $\{x_{i}\}_{1\leq i\leq \theta}$ be a basis of $V$ such that 
$c(x_{i}\ot x_{j})=q_{ij} x_{j}\ot x_{i}$ for all $1\leq i,j\leq \theta$. 
The braiding matrix $\mathbf{q}=(q_{ij})_{1\leq i,j\leq \theta}$ satisfies the Cartan condition which reads 
$q_{ij}q_{ji}=q_{ii}^{a_{ij}}$ for all $1\leq i,j\leq \theta$, 
and $q_{ii}=q$ is a root of unity of order $N$, where
$N$ is coprime with 210. 
The choice 
of a (\emph{minimal}) finite abelian group that affords a \emph{principal realization} of 
$V$ corresponds to a choice of a lattice $M$ between the root lattice $Q$
and the weight lattice $P$ associated with the root system corresponding 
to the Lie algebra $\lieg$ associated with $C$, see 
Subsection \ref{sec-datum} and Definition \ref{def-contain-group}:
Given any lattice $M$ of rank $\theta$
with $Q\subseteq M \subseteq P$, we say that
{\it an abelian group $\Gamma$  contains} $M$ if there exist
group homomorphisms $g: M\to \Gamma$, $\chi: M\to \hat{\Gamma}$ such that $g(\alpha_{i}) = g_{i}$ 
and $\chi(\alpha_{i}) = \chi_{i}$
for all $i\in I$. In particular, $\Gamma$ contains elements $g_{\lambda}=g(\lambda)$
for $\lambda \in M$, satisfying $g_{\lambda + \mu} = g_{\lambda}g_{\mu}$ and 
$g_{0} = 1$. Then any abelian group $\Gamma$ that is part of a non-trivial datum of 
finite Cartan type contains the root lattice $Q$, 
but $\Gamma$ does not necessarily contain $M$ for $M\neq Q$. 
For a finite abelian group $\Gamma$ containing a lattice $M$, we denote by 
$\cD_{M} = \cD(\Gamma, (g_i)_{i\in I},(\chi_i)_{i\in I}, (a_{ij})_{i,j\in I})$ the corresponding datum, to emphasize this dependence on 
$M$.
\subsubsection*{The Dinfeld-Jimbo case}
Let us assume first that the braiding matrix $\mathbf{q}=(q_{ij})_{1\leq i,j\leq \theta}$
is of Drinfeld-Jimbo type, that is, $q_{ij}=q^{d_{i}a_{ij}}$, for all $1\leq i,j \leq \theta$.
Given a datum $\cD$ of finite Cartan type for $\Gamma$ with indecomposable 
symmetrizable Cartan matrix, we construct a finite-dimensional 
pointed Hopf algebra $u(\cD,\mu)$,
depending on a family of parameters $\mu$, 
whose coradical is isomorphic to $\C \Gamma$. Since the Cartan matrix is indecomposable, in this case the family of parameters $\lambda$ does not appear.

 \smallbreak

Let $N$ be a positive odd integer coprime with $210$ and $q$ a root
of unity of order $N$. 
Given 
an indecomposable 
and symmetrizable finite Cartan matrix $C=(a_{ij})_{1\leq i,j\leq \theta}$
and any lattice $M$ of rank $\theta$ with $Q\subseteq M \subseteq P$, one defines
the quantized universal enveloping algebra
$U^{M}_{q}(\lieg)$, where $\lieg$ is the Lie algebra associated with $C$,
see Definition \ref{def:QUEA}.
The quantum Borel subalgebra $U^{M}_{q}(\lieb^{+})$ 
of $U^{M}_{q}(\lieg)$
is the subalgebra generated by 
the elements $K_{m}$, $E_{i}$
with $m \in M$, $i\in I$; which is actually a Hopf subalgebra.
It holds that $U^{M}_{q}(\lieb^{+})$ contains a central Hopf subalgebra
$Z_M^{\geq}$ generated by the elements 
$K_m^N$, $E_\alpha^N$, for all $m \in M$ and $\alpha \in Q^{+}$,
and the corresponding quotient Hopf algebra is isomorphic to the 
small quantum group $u_{q}^{M}(\lieb^{+})$, see Definition \ref{def:flkernels}.
In particular, $U^{M}_{q}(\lieb^{+})$ fits into a short exact sequence of Hopf algebras
$$
\xymatrix{Z^{\geq}_{M} \ar@{^{(}->}[r]^(0.4){\iota_{U}}& 
U^{M}_{q}(\lieb^{+})\ar@{->>}[r]^{\pi_{U}}& u^{M}_{q}(\lieb^{+}).}
$$ 

We prove in Propositions \ref{prop:ces-u} and \ref{prop:U-u} that,
given a family of root vector parameters
$\mu=(\mu_{\alpha})_{\alpha\in Q^{+}}$  in 
$\C$, the Hopf algebra $u(\cD,\mu)$ fits into a commutative diagram
of Hopf algebras whose rows are exact
$$
\xymatrix{
Z^{\geq}_{M} \ar@{^{(}->}[r]^(0.4){\iota_{U}} \ar@{->>}[d]_{p_{0}} & 
U^{M}_{q}(\lieb^{+})\ar@{->>}[r]^{\pi_{U}}\ar@{->>}[d]^{p}
& u^{M}_{q}(\lieb^{+}) \ar@{=}[d]^{\id}\\
\C \subg  \ar@{^{(}->}[r]^>>>>{\iota_{u}}& u(\cD_{M},\mu)\ar@{->>}[r]^{\pi_{u}}& u_{q}^{M}(\lieb^{+}),}
$$	
where $\subg$ be the subgroup of $\Gamma$ generated by $g(m)^{N}=:g_{m}^{N}$, for all $m\in M$. \smallbreak

Now, as it is well-known in the theory of quantum groups, there exists a Hopf 
algebra isomorphism $U^{M}_{q}(\lieb^{+}) \simeq \cO_{q}(B_{M}^{+})$, for all lattices 
$Q\subseteq M \subseteq P$, where $B_{M}^{+}$ is the standard Borel subgroup 
of the simple Lie group $G_{M}$ corresponding to $\lieg$ and $M$, and 
$\cO_{q}(B_{M}^{+})$ is the associated quantized function algebra, 
see Subsection \ref{subsec:Oq}. Under this isomorphism, the central Hopf subalgebra
$Z^{\geq}_{M}$ coincides with the image of the quantum Frobenius map, which
is isomorphic to the function algebra $\cO(B_{M}^{+})$, and the corresponding quotient 
Hopf algebra $\overline{\cO_{q}(B_{M}^{+})}$
is isomorphic to $\big(u_{q}^{M'}(\lieb^{+})\big)^{*}$, which is
the dual of the Frobenius-Lusztig kernel at the dual lattice $M'$,
see Proposition \ref{prop:sucOB} for more details. 
Using that 
$\big(u_{q}^{M'}(\lieb^{+})\big)^{*}\simeq 
 u_{q}^{M}(\lieb^{+})$ as
Hopf algebras, 
we obtain our first main result, Theorem \ref{teo:quantum-subgrup-DJ},
which yields a  commutative diagram of Hopf algebras 
whose the rows are exact sequences
$$
\xymatrix{\cO(B_{M}^{+}) \ar@{^{(}->}[r]^{\tilde{\iota}} \ar[d]_{\simeq} 
 \ar@/_2pc/[dd]_(0.7){k^{*}}
& 
 \cO_q(B_{M}^{+}) \ar@{->>}[r]^{\pi_{B^+}}   \ar@/_2pc/[dd]_(0.7){p_{\mu}}
 & \overline{\cO_{q}(B_{M}^{+})} \ar[d]_{\simeq} \\
 Z^{\geq}_{M} \ar@{^{(}->}[r]^(0.4){\iota_{U}} \ar@{->>}[d]_{p_{0}} & 
U^{M}_{q}(\lieb^{+})\ar@{->>}[r]^{\pi_{U}}\ar@{->>}[d]^{p}\ar[u]_{\simeq}^{\psi}
& u^{M}_{q}(\lieb^{+}) \ar@{=}[d]^{\id}\\
\C^{\widehat{\subg}}\ar@{^{(}->}[r]^>>>>>{\tilde{\iota}_u} & 
u(\cD_{M},\mu)  \ar@{->>}[r]^{{\pi}_u} & {u}_q^{M}(\mathfrak{b}^{+}).} 
$$
In particular, the Hopf algebra $u(\cD_{M},\mu)$ corresponds to a 
quantum subgroup of 
$\cO_q(B_{M}^{+})$. \smallbreak

One may describe the Hopf algebras 
$u(\cD_{M},\mu)$ by the data coming from the previous commutative diagram. Precisely,
the lifting of the power root vector relation 
 $x_{\alpha}^{N}=r_{\alpha}(\mu)$ can be computed 
 using 
the isomorphism 
$\psi: U_{q}^{M}(\lieb^{+})\to \cO_{q}(B^{+}_{M})$,
an algebraic group monomorphism $k: \widehat{\subg} \hookrightarrow B^{+}_{M}$ and the 
well-known formula that gives the expression of any positive root 
$\alpha$ in terms of the simple roots $\alpha_{i}$, with $i\in I$,
for each type of Lie algebra. Indeed,
for all $\alpha \in \Phi^{+}$, one has that
\begin{align*}
x_{\alpha}^{N}&=p(E_{\alpha}^{N}) = p_{\mu}(\psi(E_{\alpha}^{N}))
=\varphi^{-1} k^{*}(\psi(E_{\alpha}^{N})),
\end{align*}
where 
$\varphi^{-1}:\C^{\widehat{\subg}}\to \C T_{M}$ 
is a natural isomorphism of Hopf algebras.
Taking the monomorphism $k: \widehat{\subg} \hookrightarrow B^{+}_{M}$ 
to be the canonical inclusion in the diagonal matrices
gives the bosonization $\toba(V)\#\C\Gamma$. On the other hand, taking $k$
to be the canonical inclusion composed with the conjugation by a unipotent matrix gives
a lifting whose relations are the ones predicted by the classification theorem of 
Andruskiewitch and Schneider. We illustrate this with examples in type $A_{\theta}$, $B_{\theta}$
and $D_{\theta}$ in Section \ref{sec:PHa-qsubgroups}.

\smallbreak
As corollary, see Corollary \ref{cor:all-liftings-qsubgroups},
we get that any such a Hopf algebra 
$u(\cD_{M},\mu)$, and consequently any finite-dimensional complex pointed Hopf 
algebra over a group $\Gamma$ (containing $M$) 
of Drinfeld-Jimbo type and indecomposable symmetrizable finite
Cartan matrix
can be constructed as a quantum subgroup, that is,
$$
u(\cD_{M},\mu) \simeq \cO_{q}(B^{+}_{M})/(\Ker k^{*}),
$$
as Hopf algebras, for a certain monomorphism of algebraic groups $k$.\smallbreak

In the second main result of the paper, Theorem \ref{teo:qsubgroups-are-liftings},  
we prove that any inclusion of a finite abelian group in $B^{+}_{M}$ 
gives a lifting of $\toba(V)\# \C\Gamma$, for some abelian group $\Gamma$ 
that contains $M$ and some Yetter-Drinfeld module $V$ over $\C\Gamma$.

\subsubsection*{General Cartan case}
The results for diagonal braidings of general Cartan type 
are completely analogous and 
follow \emph{mutatis mutandis}. The role of the quantum universal enveloping 
algebra is now played by a multiparametric version of it, which is denoted
by $ U_{\bf q}^{M}(\lieb^{+}) $, and the role of the small quantum group by a 
bosonization $\lieu^{M}(V)^{\geq 0}:= \toba(V)\#\C\overline{\Gamma}$.
Here, $\overline{\Gamma}:= \Gamma / \subg$ where
$\subg$ is the subgroup of $\Gamma$ generated by
the elements $g_{m_i}^{N}$ such that  
$\chi_{m_i}^{N} =\varepsilon$.
Note that in the Drinfeld-Jimbo case, this subgroup 
coincides with the  subgroup $\subg$ considered previously. \smallbreak

To apply the techniques used for the Drinfeld-Jimbo case, we have to 
assume that $\chi_{m}^{N}=\eps$, for all 
$m\in M$.
Under this mild assumption, 
$\C \subg$ is central in 
$u(\cD_{M},\mu)$ 
and there exists a central exact sequence of Hopf algebras 
$$
\xymatrix{\C \subg  \ar@{^{(}->}[r]^(0.4){\iota_{u}}& 
u(\cD_{M},\mu)\ar@{->>}[r]^(0.45){\pi_{u}}& \lieu^{M}(V)^{\geq 0},}
$$
see Proposition \ref{prop:ces-u-Cartan}.
Besides, it holds that
$U_{\bf q}^{M}(\lieb^{+}) \simeq 
\widetilde{\toba}(V)\# \C \Z^{\theta}$ as Hopf algebras, 
where $\tilde{\toba}(V)$ is the distinguished pre-Nichols algebra
defined in \cite[Section 3]{An2}, and
there exists a Hopf algebra epimorphism 
$p: U_{\bf q}^{M}(\lieb^{+}) \twoheadrightarrow u(\cD_{M},\mu)$ making
the following diagram of short exact sequences exact commutative
$$
\xymatrix{
Z^{\geq}_{M} \ar@{^{(}->}[r]^(0.4){\iota_{U}} \ar@{->>}[d]_{p_{0}} & 
U_{\bf q}^{M}(\lieb^{+})\ar@{->>}[r]^{\pi_{U}}\ar@{->>}[d]^{p}
& \lieu^{M}(V)^{\geq 0} \ar@{=}[d]^{}\\
\C \subg  \ar@{^{(}->}[r]^{\iota_{u}}& u(\cD_{M},\mu)\ar@{->>}[r]^{\pi_{u}}& 
\lieu^{M}(V)^{\geq 0}.}
$$
see Proposition \ref{prop:Uq-uD} for details. Now, the key fact is that 
the multiparametric quantized universal enveloping algebra $U^{M}_{\bf q}(\lieb^{+})$
is isomorphic to a 2-cocycle deformation (by a 2-cocycle $\sigma$) of $U^{M}_{q} (\lieb^{+}) $.
This isomorphisms goes through
the Hopf algebra isomorphism 
$\psi: U_{q}^{M}(\lieb^{+}) \to \cO_{q}(B^{+}_{M})$ 
and it follows that $U^{M}_{\bf q}(\lieb^{+})\simeq 
\big(\, \cO_{q}(B^{+}_{M})\big)_{\tilde{\sigma}}$, 
where $\tilde{\sigma}$ is the $2$-cocycle
 induced by $\sigma$ 
  and $\psi$.
 The Hopf algebra $ \big(\, \cO_{q}(B^{+}_{M}) \big)_{\tilde{\sigma}}$ 
 is a multiparametric version $\cO_{\bf q}(B^{+}_{M}) $ 
 of the quantized algebra of functions on the positive Borel subgroup 
 $B^{+}_{M}$. 
 Moreover, it turns out that the  
 central Hopf subalgebra $Z^{\geq}_{M}$ is not deformed under the 
 $2$-cocycle and remains central. This leads to our third main result, 
 Theorem \ref{teo:quantum-subgroup-Cartan}, which gives the 
 commutative diagram of Hopf algebras with exact rows
$$
\xymatrix{
\cO(B^{+}_{M}) \ar@{^{(}->}[r] \ar[d]^{\simeq} 
 \ar@/_2pc/[dd]_(0.7){k^{*}}
&   
\cO_{\bf q}(B^{+}_{M}) \ar@{->>}[r] \ar[d]^{\simeq} \ar@/_2pc/[dd]_(0.7){p_{\mu}} & 
\overline{\cO_{\bf q}(B^{+}_{M})} 
\ar[d]^{\simeq}\\
Z^{\geq}_{M} \ar@{^{(}->}[r]^(0.4){\iota_{U}} \ar@{->>}[d]_{p_{0}} & 
U_{\bf q}^{M}(\lieb^{+})\ar@{->>}[r]^{\pi_{U}}\ar@{->>}[d]^{p}\ar[u]^{\psi}_{\simeq}
& \lieu^{M}(V)^{\geq 0} \ar@{=}[d]^{}\\
\C^{\widehat{\subg}}\ar@{^{(}->}[r] & u(\cD_{M},\mu)  \ar@{->>}[r] & 
\lieu^{M}(V)^{\geq 0}} 
$$
Then we prove the analogous results of
Theorem \ref{teo:qsubgroups-are-liftings} and Corollary 
\ref{cor:all-liftings-qsubgroups} in this case in Theorem 
\ref{teo:qsubgroups-are-liftings-Cartan}. 
As in the case of braidings of Drinfeld-Jimbo type, these 
theorems give a geometric interpretation of the liftings.
Again, the lifting of a power root vector
is given by a composition of Hopf algebra maps.
This allows one to interpret the pointed Hopf algebras
$u(\cD_{M},\mu)$ as quantum subgroups of the 
multiparametric quantum groups  $\cO_{\bf q}(B^{+}_{M})$, with 
${\bf q} = (q_{ij})_{i,j\in I}$.

\subsection*{Examples recovering previous results and new ones}
We end the paper with a section devoted to illustrate our results with two examples
in the Drinfeld-Jimbo case.
We recover in Theorem \ref{lifting-case-an} 
the liftings of type $A_{\theta}$
computed in \cite{AS2}, of type $B_2$ computed in \cite{BDR}, and of type $B_3$ computed in \cite{BGM} by choosing a particular 
monomorphism of algebraic groups $k: \widehat{\subg}\hookrightarrow B^{+}_{M}$,
which is given by a canonical inclusion composed with a conjugation by 
a certain unipotent matrix.
Furthermore, to show the power of this technique  
we compute all liftings of type $B_{\theta}$ and $D_{\theta}$, for all $\theta \geq 2$. 
As a byproduct, we give explicit formulas for the dual quantum 
Frobenius map in type $B_{\theta}$ and $D_{\theta}$ that seem to be 
missing in the literature, see Theorem \ref{inclusion-case-bn}.

\subsection*{Organization of the paper}
Let us outline the organization of the manuscript. 
In Section \ref{sec:prelim} we set up notation and recall some facts on 
Hopf algebras, Nichols algebras, Yetter-Drinfeld modules, quantized universal 
enveloping algebras, small quantum groups 
and quantized function algebras that will be used throughout the paper. We also 
recall, in Theorem \ref{thm:AS},
the classification theorem of complex finite-dimensional pointed Hopf 
algebras due to Andruskiewitsch and Schneider.
In Section \ref{sec:ces} we prove our main results,
showing how the family of pointed Hopf algebras $u(\cD,\mu)$
can be described as quantum subgroups.
Subsection \ref{subsec:DJ} is devoted to the case of diagonal braidings of 
Drinfeld-Jimbo type, whereas Subsection \ref{subsec:BraidingGenCartan} 
deals with braidings of general Cartan type. Finally, in Section \ref{sec:PHa-qsubgroups}, we illustrate our approach with two \emph{not-to-thechnical} examples of types
$A_{\theta}$, $B_{\theta}$ and $D_{\theta}$, recovering known results and obtaining new explicit liftings.

\section*{Acknowledments}
The authors want to thank Fabio Gavarini
for fruitful discussions on isomorphisms on quantum groups.

\section{Preliminaries}\label{sec:prelim}

\subsection{Notation and conventions}
 We work over the field of complex numbers $\C$.
 We denote by $\N$ the set of positive integers 
 and by $\N_{0} = \N \cup \{0\}$.
 
Let $A$ be an algebra and $\iota: B\hookrightarrow A$ a subalgebra of $A$.
Given a two-sided $J$ of $B$, we write $(J)$ for the two-sided 
ideal of $A$ generated by $\iota(J)$.
 
Let $z$ be an indeterminate and write $ \Z[z,z^{-1}] $  for the ring of Laurent polynomials 
with integral coefficients. For every  $ n \in \N_{0}$ we set
  $$ 
  \displaylines{
   {(0)}_z   =   1 ,  \quad\quad
    {(n)}_z   =   \frac{z^n -1}{z-1}   =   1 + z + \cdots + z^{n-1}   
     \qquad \qquad ( \in  \Z[z])  \cr 
   {(n)}_z! =   {(0)}_z {(1)}_z \cdots {(n)}_z   ,  \qquad
    {\binom{n}{k}}_{z}  =   \frac{{(n)}_z!}{\;{(k)}_z! {(n-k)}_z! }  \qquad  (\in \Z[z]) \cr
   {[0]}_z  =  1 ,\quad  
    {[n]}_z  =  \frac{z^n -z^{-n}}{z-z^{-1}}  =  z^{-(n-1)} + \cdots + z^{n-1}  
      \quad  ( \in \Z[z,z^{-1}] )  \cr
   {[n]}_z!  =   {[0]}_z {[1]}_z\cdots {[n]}_z ,  \qquad
    {\bigg[\, {n \atop k} \,\bigg]}_z  =   \frac{{[n]}_z!}{\;{[k]}_z! {[n-k]}_z!}  \qquad  
    ( \in \, \Z[z,z^{-1}])  } 
    $$
 \vskip3pt
%
%
\noindent
 With these definitions, one has the following identities: 
 \[
   \; \displaystyle{{(n)}_{z^2} = z^{n-1} {[n]}_z} ,  \qquad 
   \displaystyle{{(n)}_{z^2}! = z^{\frac{n(n-1)}{2}} {[n]}_z} ,  \qquad
   \displaystyle{{\binom{n}{k}}_{\!\!z^{2}} \!\! = z^{k(n-k)} 
   {\Big[{n \atop k} \,\Big]}_z} .
 \]
 Thinking of Laurent polynomials as functions on  
   $ \C^{\times}=\C-\{0\} $,  for any  $q \in \C^\times$ 
  we read every symbol above as the corresponding
  evaluation $z=q$ of the element in $ \C$.

We recall that the finite dual $A^{\circ}$ of an algebra $A$ is the 
vector subspace of $A^{\ast}$ whose elements 
are the functionals $f:A\to \mathbb{C}$ such that $\ker f$ 
contains an ideal of finite codimension.

\subsection{Hopf algebras}\label{subsec:Hopf-algebras}
 For the theory of 
 Hopf algebras we refer to \cite{Mo} and \cite{Ra}.
 The comultiplication is denoted  
 $ \com $, the counit  $\eps$ and the antipode
 $ S$. We use the Heyneman-Sweedler 
 notation for $\com$, i.e.  
 $\com(x) = x_{(1)} \otimes x_{(2)}$.
 Given a Hopf algebra $H$, the set of group-like elements is 
 denoted by  $ G(H)$, and by $H^+ := \Ker(\eps)$ the augmentation ideal.  
 For $g, h \in G(H)$,  the set of  $ (g,h) $--primitive elements is defined by
 $P_{g,h}(H):= \{x \in H \,\vert\, \com(x) = x \ot g + h \ot x\} $
and $P_{1,1}(H)$ is simply denoted by $P_{1}(H)$.
We write  $ H^{\text{op}} \, $,  resp.\  $ H^{\text{cop}} \, $,
for the Hopf algebra given by taking in  $ H $  the opposite product, resp.\ coproduct.
Let $ \pi : H \longrightarrow A$ be a Hopf algebra map.    
Then  $ H $  is a left and right  $ A $--comodule via
$(\pi \otimes \id) \com : H \longrightarrow A \otimes H$
and  $ (\id \otimes \pi) \com : H \longrightarrow H \otimes A$, 
respectively.
The space of left and right  
\textit{coinvariants\/}  is given, respectively, by
$$
   {}^{\co \pi}H =  \{h \in H \,|\, (\pi \otimes \id) \com(h) = 1 \otimes h\},\quad
   H^{\co \pi}   =  \{ h \in H \,|\, (\id \otimes \pi) \com(h) = h \otimes 1\}.
$$

A  Hopf algebra $H$ is {\it pointed} if the coradical $H_0$ of $H$, that is, 
the sum of all simple subcoalgebras, 
is the group algebra $\Bbbk G(H)$. 
Let $ H $  and  $ K $ be Hopf algebras with bijective antipode. A  linear  map  
$(\ ,\ ): H \otimes_\C K \longrightarrow \C \, $  is called a
\emph{Hopf pairing} between  $ H $  and  $ K$  if satisfies: for all  $h \in H$,  $k \in K$,
\begin{align*}
\big( h \, , \, k_1 \, k_2 \big) &=\big( h_{(1)} \, , \, k_1 \big)\big( h_{(2)} \, , \, k_2 \big) , & \big( h_1 \, h_2 \, , \, k \big)&= \,\big( h_1 \, , \, k_{(1)} \big)	\big( h_2 \, , \, k_{(2)} \big), &  	\big( h \, , 1 \big) &=\epsilon(h),\\[.2em]
\big( S(h) \, , k \big)&=\big( h \, , S(k) \big),    & \big( S^{- 1}(h) \, , k \big)&=\big( h \, , S^{-1}(k) \big), & \big( 1 \, , k \big)&=\epsilon(k).
\end{align*}
\smallbreak
Following \cite{AD}, a sequence of Hopf algebras maps 
$$
\xymatrix{B \ar@{^{(}->}[r]^{\iota}& A\ar@{->>}[r]^{\pi}& H}
$$
is called  {\it exact\/}  if  $ \iota $  is injective,
$ \pi $  is surjective, $\pi\circ \iota=\varepsilon$,  $ \Ker(\pi) = AB^+ \, $  and  
$B = {}^{\co \pi} A$. We say that the sequence is \textit{central} if the 
image of $B$ is central in $A$.

Let $ H$  be a Hopf algebra.
 A  normalized Hopf 2-cocycle  (see  \cite[Sec.\ 7.1]{Mo}) 
  is a map  $ \sigma $  in  $ \Hom_{\Bbbk}(H \otimes H, \C) $  which is convolution 
  invertible and satisfies for all $a, b, c \in H$ that
  $$ 
   \sigma(b_{(1)},c_{(1)}) \, \sigma(a,b_{(2)}c_{(2)}) \, = \, \sigma(a_{(1)},b_{(1)}) \, 
   \sigma(a_{(2)}b_{(2)},c),\quad     \, \sigma (a,1) = \eps(a) = \sigma(1,a).
   $$
We use a  $ 2 $--cocycle  $ \sigma $ to define a new Hopf algebra $ H_\sigma$ 
in the following way. As coalgebra, $H_\sigma=H$ and  the multiplication on $H_\sigma$ is  
 $  m_{\sigma} = \sigma * m * \sigma^{-1} : H \otimes H \to H$, that is,
  $$ 
   m_{\sigma}(a,b)  \, = \,  a \cdot_{\sigma} b  \, = \,
  \sigma(a_{(1)},b_{(1)}) \, a_{(2)} \, b_{(2)} \, \sigma^{-1}(a_{(3)},b_{(3)}),  \quad \text{for all} \;\, a, b \in H.  
  $$
The antipode map $S_{\sigma}:H_{\sigma}\to H_\sigma$ is given by 
$$
S_{\sigma}(a)=\sigma(a_{(1)},S(a_{(2)}))S(a_{(3)})
\sigma^{-1}(S(a_{(4)}),a_{(5)}),\quad\text{ for all }a\in H.
$$
The Hopf algebra $H_{\sigma}$  is  
called a $2$-cocycle deformation of $H$.

\subsection{Nichols algebra and liftings}\label{subsec:Nichols-algebras} 
A braided vector space $V$ is a pair $(V, c)$ where $V$ is a vector space 
and $c \in GL(V^{\ot 2})$ is a solution to the braid equation
\begin{align*}
	(c \ot \id) (\id \ot c) (c \ot \id) &= (\id \ot c) (c \ot \id)(\id \ot c).
\end{align*}

Let $ H$  be a Hopf algebra. A (left) Yetter-Drinfeld module over $H$ is a left 
$H$-module $(V,\cdot)$ and a left $H$-comodule $(V,\delta)$ 
with $\delta: V\to H\ot V$, $\delta(x) = x_{(-1)} \otimes x_{(0)}$ for all $x\in V$, such that
$$ 
\delta(h\cdot x) = h_{(1)}x_{(-1)}S(h_{(3)})\ot h_{(2)}\cdot x_{(0)},\qquad 
\text{ for all }h\in H, x\in V.
$$
Yetter-Drinfeld modules, together with morphisms of left $H$-modules and left $H$-comodules, form a braided rigid tensor category denoted by $\ydh$. Indeed, any $V\in \ydh$ 
is a braided vector space with braiding 
$c(x \otimes y) = x_{(-1)} \cdot y\otimes x_{(0)}$ for all
$x,y\in V$. For a group algebra $H=\C \Gamma$, we simply write
$\ydga$ for the category of Yetter-Drinfeld modules.

A Hopf algebra in a braided category is called a braided Hopf algebra for short. 
In any braided Hopf algebra $R$ with multiplication $m$ and braiding 
$c: R\ot R \to R\ot R$, the braided commutator of elements $x, y \in R$ 
is given by
$[x, y]_{c} = m(x\otimes y) -m\circ c(x \ot y)$.
If $x \in P_{1}(R)$ is a primitive element, then
\begin{equation}\label{eq:braided-adjoint-action}
	(\ad x )(y) = [x, y]_{c}
\end{equation}
denotes the \emph{braided adjoint action} of $x$ on $R$.

We are interested in the following class of braided vector spaces. A braided space
$(V, c)$ is of \emph{diagonal type} if there exist a basis
$(x_i)_{1\leq i\leq n}$ of $V$ and a matrix $ (q_{ij})\in \C^{n\times n}$
such that $q_{ij} \in \C^\times$ and 
\[
c(x_i\otimes x_j)=q_{ij}x_j\otimes x_i,\quad\text{ for all }\, 1\leq i,j\leq n.
\]

\smallbreak
Diagonal braided vector spaces can be realized as 
Yetter-Drinfeld modules over group algebras of abelian groups 
as follows. 
Let $\Gamma$ be an abelian group and let $\widehat \Gamma$ 
be the group of  characters.
Yetter-Drinfeld modules over $\C \Gamma$ 
are $\Gamma$-graded $\Gamma$-modules
$V = \oplus_{g\in \Gamma} V_g$ such that
 $h \cdot V_g=V_g$ for all $g,h\in \Gamma$. For example,
for $g\in \Gamma$ and $\chi \in \widehat\Gamma$, 
the one-dimensional 
vector space  
$\C_g^{\chi}$ is an object in $\ydga$ with one homogeneous component of
degree $g$ and  
the action given by the evaluation on 
$\chi$. 
A \emph{realization} of a braided vector space $V$ 
of diagonal type with matrix $(q_{ij})\in  \C^{n\times n}$
as an object in $\ydga$ 
is then given by a collection $(g_{1}, \chi_{1}), \dots, (g_{n}, \chi_{n})$ such that
$g_{i}\in \Gamma$, $\chi_{j}\in \widehat{\Gamma}$ and
$q_{ij} =\chi_j (g_i)$, for all $1\leq i, j \leq n$. The $\Gamma$-grading is
given by $\delta(x_{i})=g_{i}\ot x_{i}$ and the $\Gamma$-action
by $g\cdot x_{i} = \chi_{i}(g)\, x_{i}$ for all $1\leq i\leq n$.

\smallbreak  
Let $V \in \ydga$. The {\it Nichols algebra} of $V$ is the unique (up to isomorphism)
 graded connected Hopf algebra $\toba(V) = \oplus_{n\ge 0} \toba^n(V)$ in $\ydga$  
 such that $V \simeq \toba^1(V) = P_1(\toba(V))$ generates $\toba(V)$
 as algebra; we suggest \cite{a-leyva} for details on Nichols algebra. 
 Given $\toba(V)\in \ydga$, 
 the bosonization $H:=\toba(V)\#\C \Gamma$
is an ordinary pointed Hopf algebra with coradical $\C\Gamma$.
 If $\toba(V)$ is finite-dimensional and $\Gamma$ is finite then 
 $H$ is also finite-dimensional.
 
 \vspace{.1cm}
 
 Assume that $A$ is finite-dimensional pointed Hopf algebra with $G(A)=\Gamma$ abelian. 
 Then the graded algebra $\text{gr}\, A$ with respect to the coradical 
 filtration is again a pointed Hopf algebra and 
 it turns out that 
 $\text{gr}\, A \simeq B\# \C \Gamma$, where $B=\toba(V)$ is 
 the Nichols algebra of $V:=P_1(B)$. 
The Hopf algebra $A$ is called a {\it lifting of $\toba(V)\#\C\Gamma$}
and it is a $2$-cocycle deformation of $\text{gr}\, A$, see for example 
\cite{AGI}.

 \subsection{Lie algebras, root systems and Cartan matrices}
 \label{subsec:Lie-algebras}
 For the theory of Lie algebras 
 we follow \cite{Hu} and \cite{Bou}.
 Hereafter we fix  $ \theta \in \N $  and  $ \, I := \{1,\dots,\theta\} \, $.
 Let  $\mcar = {\big(a_{ij} \big)}_{i, j \in I} $  
 be a symmetrizable Cartan matrix of finite type; that is, there exists a
 unique diagonal matrix  $D = (d_i \, \delta_{ij} )_{i, j \in I} $  with positive integers entries such that  $ \, D \mcar \, $  is symmetric.  
 
 \smallbreak
 Let  $ \lieg $  be the finite-dimensional simple Lie
 algebra over  $ \C $  associated with  $ \mcar \, $. As usual,  
 $ \Phi $ denotes the (finite) root system of  $ \lieg $,  $ \Phi^+ $ a set of positive roots,
  $\Pi = \{\alpha_i \,\vert\, i\in I\} \, $  is the set of simple roots and
  $ Q = \bigoplus_{i \in I} \Z \alpha_i $  the root lattice.
  We denote by  $ P =\bigoplus_{i \in I} \Z\, \omega_i$
the weight lattice with basis  
$ \{ \omega_i \}_{i \in I}$  
dual to  $  \{\alpha_j\}_{j \in I} $ with respect to the pairing 
$ \, \omega_i(\alpha_j) = \delta_{ij}d_i$  for all  $  i, j \in I $, and
by $P^{+}= \bigoplus_{i \in I} \Z_{\geq 0}\, \omega_i$ 
the set of dominant integral weights 
$P^{+}$.
The lattice $ Q $  can be identified
with a sublattice of  $ P$;  
in particular, we have that
$  \alpha_i = \sum_{j \in I} a_{ji} \, \omega_j $,  
for all  $ i \in I$.
Also, every weight $\omega_{i}$,  $i\in I$, 
can be written in terms of the 
roots as $  \omega_i = \sum_{j \in I} a'_{ij} \, \alpha_j $,
where $\mcar^{-1}=(a'_{ij})_{i,j\in I}$ is the inverse matrix of $\mcar$;
here $a \mcar^{-1} \in \Z^{\theta\times \theta}$ with $a=|\det \mcar |= [P:Q]$.
For any lattice $Q\subseteq M \subseteq P$ of rank $\theta$ 
with basis $\{m_{1},\ldots, m_{\theta}\}$
we have that 
$  \alpha_i = \sum_{j \in I} a_{ji}^{M} \, m_j $
and $  m_j = \sum_{k \in I} \bar{a}_{kj}^{M} \, w_k $
for all  $ i,j \in I$. 
Here the matrices
$C_{M}= (a_{ji}^{M})_{i,j}, \bar{C}_{M} = (\bar{a}_{kj}^{M})_{kj}\in \Z^{\theta\times \theta} $
satisfy that $C = \bar{C}_{M}C_{M}$ with $[M:Q]=|\det C_{M}|$ and $[P:M]=|\det \bar{C}_{M}|$.
Since $a=[P:M][M:Q]$, we have that $|\det C_{M}|$ divides $a$. 
We set $M^{+}=M\cap P^{+}$.

\smallbreak
In this setup, we have a natural  non-degenerate $ \Z $--bilinear pairing  
$ ( \,\ ,\,\ ): P \times Q\to  \Z$ given by the evaluation (of weights onto roots), 
$ \, (\omega_i , \alpha_ j) = \omega_{i}(\alpha_{j}) = \delta_{ij}d_i$,  for all  $i, j \in I$. It 
satisfies that $(\alpha_{i},\alpha_{j}) = d_{i}a_{ij}$, for all $i,j\in I$.

This  $ \Z $--bilinear pairing can be extended to a  $ \Z $--bilinear pairing 
$ ( \,\ ,\,\ ): P \times P\to  \Q$ with values in $\Z[a^{-1}]$. 
We say that two lattices $Q\subseteq M$, $M'\subseteq P$ are \emph{dual 
to each other} if $M'=\{\beta\in P\, |\, (\alpha,\beta)\in \Z,\,\forall \alpha\in M\}$ and 
$M=\{\beta\in P\, |\, (\beta, \alpha')\in \Z,\,\forall\alpha'\in M'\}$ (both conditions being equivalent).
Moreover, given a lattice  $Q\subseteq M \subseteq P$ there exists a unique 
lattice $Q\subseteq M' \subseteq P$ dual to $M$ such that the 
 $ \Z $--bilinear pairing $ ( \,\ ,\,\ ): P \times P\to  \Q$ restricts to a 
 non-degenerate $ \Z $--bilinear pairing  
$ ( \,\ ,\,\ ): M \times M'\to  \Z$. 
For example, 
$Q$ and $P$ are dual to each other.


\subsection{Quantized universal enveloping 
algebras and small quantum groups}\label{subsec:QUEA-small}

We follow \cite{BG} for notation and convention on quantum groups.
Let $q\in \C$ be a primitive root of unity of odd 
order $N >1$  coprime with $3$ 
if the Cartan matrix $\mcar$ is of type $G_{2}$.  
In what follows, for all $i\in I$, we fix $q_{i} = q^{d_{i}}$, where $d_i$ is the $(i,i)$ 
entry of the diagonal matrix $D$ considered in the previous subsection.

\subsubsection{Quantized universal enveloping 
algebras}
We start recalling the definition of quantized universal enveloping algebra.
\begin{defi}\label{def:QUEA}
Let $M$ be a lattice between $Q$ and $P$ of rank $\theta$.
The \emph{quantized universal enveloping algebra} 
$U^{M}_{q}(\lieg)$ 
is the algebra generated by the elements
$\{K_{m}\}_{m \in M}$, 
$E_{1},\ldots, E_{\theta}$ and 
$F_{1},\ldots, F_{\theta}$, satisfying
the following relations: for all $m,m_{1}, m_{2} \in M$ and 
$i,j\in I$,

\begin{align*}
K_{0} = 1, &\qquad K_{m_{1}}K_{m_{2}} = K_{m_{1}+m_{2}}, \\
K_{m}E_{j}K_{-m}  = q^{(m,\alpha_{j})}E_{j}, &
\qquad K_{m}F_{j}K_{-m}
=   q^{-(m, \alpha_{j})}F_{j},\\
E_{i}F_{j} - F_{j}E_{i} & =
\delta_{ij}\frac{K_{\alpha_{i}}-K_{\alpha_{i}}^{-1}}{q_{i} -
q_{i}^{-1}},\\
\sum_{l=0}^{1-a_{ij}} (-1)^{l}&
{\left[ \begin{smallmatrix} 1-a_{ij}\\
l\end{smallmatrix} \right]_{q_{i}}} E_{i}^{1-a_{ij}-l}E_{j}E_{i}^{l}
= 0 \qquad (i\neq j),\\
\sum_{l=0}^{1-a_{ij}} (-1)^{l}& {\left[ \begin{smallmatrix} 1-a_{ij}\\
l \end{smallmatrix}\right]_{q_{i}}} F_{i}^{1-a_{ij}-l}F_{j}F_{i}^{l}
= 0 \qquad  (i\neq j).
\end{align*}
\end{defi}
Note that
$K_{-m} = K_{m}^{-1}$ for all 
$m \in M$.
It holds that $ U_q^{M}(\lieg) $  
 is a Hopf algebra with coproduct, 
 counit and antipode 
   given as follows: for all  $ \, i \in I \, $ and $m\in M$, 
\begin{align*}
  \com(E_i)  &  =   E_i \otimes 1 + K_{\alpha_{i}} \otimes E_i  ,  &
 \eps(E_i)   &  =  0,  &  S(E_i)  &  =  -K_{\alpha_{i}}^{-1} E_i  \\
  \com(F_i)  &  =   F_i \otimes K_{-\alpha_{i}} + 1 \otimes F_i,  &
 \eps(F_i)   &  =  0 ,  &  S(F_i) &  = - F_i K_{\alpha_{i}}  \\
  \com(K_{\pm m})  &  =  K_{\pm m} \otimes K_{\pm m}  ,  &
 \eps (K_{\pm m})  &  =  1,  &  S(K_{\pm m})  &  = K_{\mp m}.
\end{align*}

\begin{exemplo}\label{exemplo:M=Q-M=P}
If $M=Q$, we write $U^{Q}_{q}(\lieg)=U_{q}(\lieg)$ and $K_{i} = K_{\alpha_{i}}$
for all $i\in I$. In particular, the commuting relations 
between the $K_{i}$'s and  the $E_{j}$'s , $F_{j}$'s read 
$$
K_{i}E_{j}K_{i}^{-1} = q^{(\alpha_{i},\alpha_{j})}E_{j} =q^{d_{i}a_{ij}}E_{j},\qquad
K_{i}F_{j}K_{i}^{-1} = q^{-(\alpha_{i},\alpha_{j})}F_{j} = q^{-d_{i}a_{ij}}F_{j}
\quad\text{ for }i,j\in I.
$$
If $M=P$, we write $K_{\omega_{i}}$ for all $i\in I$ for the generators corresponding to
the fundamental weights
and the commuting relations 
read
$$
K_{\omega_{i}}E_{j}K_{\omega_{i}}^{-1} = q^{(\omega_{i},\alpha_{j})}E_{j} = q^{d_i\delta_{ij}}E_{j},\qquad
K_{\omega_{i}}F_{j}K_{\omega_{i}}^{-1} = q^{-(\omega_{i},\alpha_{j})}F_{j} = q^{-d_i\delta_{ij}}F_{j}
\quad\text{for }i,j\in I.
$$
\end{exemplo}

The quantum Borel subalgebras 
$U^{M}_{q}(\lieb^{+})$ and  $U^{M}_{q}(\lieb^{-})$
of $U^{M}_{q}(\lieg)$
are the subalgebras generated by 
the elements $K_{m}$, $E_{i}$
with $m \in M$, $i\in I$,
and  $K_{m}$, $F_{i}$
with $m \in M$, $i\in I$, repectively.
Actually, both subalgebras admit a similar presentation
with genetarors and relations. 
Notice that $U^{M}_{q}(\lieb^{+})$ and $U^{M}_{q}(\lieb^{-})$ 
are both Hopf subalgebras of $U^{M}_{q}(\lieg)$. \smallbreak

Let $Q\subseteq M, M'\subseteq P$ be dual lattices. 
Then, there exists a non-degenerate Hopf pairing  
$(\ ,\ ):\, U^{M}_{q}(\lieb^{-})^{\cop} \ot U_{q}^{M'}(\lieb^{+}) \to \C$
 given by: 
\begin{equation}\label{eq:pairingb+b-}
	(K_{m},K_{m'})  = q^{-(m,m')}, 
	\quad (K_{m},E_{j}) = 0,\quad
	(F_{i},E_{j}) = \frac{\delta_{ij}}{q_{i}^{-1}-q_{i}},\quad
	(F_{i}, K_{m}) = 0,
\end{equation}
for all $i,j\in I$ and $m\in M$, $m'\in M'$.

\smallbreak
The braid group $B_{W}$ associated with the Weyl group 
corresponding to $\lieg$ acts
on $U^{M}_{q}(\lieg)$, see \cite[I.6.7, III.6.2]{BG}. 
Using this action one construct root vectors
$E_{\beta_{j}}$, $F_{\beta_{j}}$ for $\beta_{j}\in Q^{+}$ 
and we get a PBW basis of $U^{M}_{q}(\lieg)$ given 
by the monomials $M_{r,m,k} = {\bf F}^{r}K_{m}{\bf E}^{k}$ with 
$r,k \in \N_{0}^{t}$, $t$ the amount of positive roots, 
and $m \in M$;
here ${\bf F}$ and ${\bf E}$ are given monomials in a fixed ordered  
product of  root vectors
$F_{\beta_{j}}$, $E_{\beta_{j}}$, respectively.
Moreover, $E_{\beta_{j}} \in U^{M}_{q}(\lieb^{+})$ and 
$F_{\beta_{j}}\in U^{M}_{q}(\lieb^{-})$ for all $\beta_{j}\in Q^{+}$.

\subsubsection{Small quantum groups}\label{subsec:smallqgr}

Let $Z_M$ be subalgebra of 
$U^{M}_{q}(\lieg)$ generated by the elements 
$K_m^N$, $E_\alpha^N$, $F_\alpha^N$ for all $m \in M$ and $\alpha \in Q^{+}$. In turns out that $Z_M$ 
is a central Hopf subalgebra of $U^{M}_{q}(\lieg)$; 
see for instance,  \cite[Theorem III.6.2]{BG}.  
Also, $Z_{M}$ is a polynomial ring in $\dim \lieg$ generators, 
with $\theta$ invertible generators and 
$U^{M}_{q}(\lieg)$ is a free $Z_{M}$-module of rank 
$N^{\dim \lieg}$.

\begin{defi}\label{def:flkernels}
The \textit{small quantum group} or 
\textit{Frobenius-Luzstig kernel} $u_{q}^{M}(\lieg)$ associated with the lattice $M$ is defined as the quotient
Hopf algebra
$$
u_{q}^{M}(\lieg) =U^{M}_{q}(\lieg)/\big(U^{M}_{q}(\lieg)Z_{M}^+\big).
$$ 
In the case $M=Q$, we simply write $u_{q}^{Q}(\lieg)= u_{q}(\lieg)$.
\end{defi}

The small quantum groups $u_{q}^{M}(\lieb^{\pm})$
associated with the Borel subalgebras $\lieb^{\pm}$ are
the subalgebras of $u_{q}^{M}(\lieg)$ generated by 
the elements $\overline{K_{m}} = k_{m}$, $\overline{E_{i}}=e_{i}$
 and 
$\overline{K_{m}} = k_{m}$, $\overline{F_{i}}=f_{i}$ 
for all $m\in M$ and $i\in I$, respectively. 
Note that $\overline{E_{\beta}}=e_{\beta}\in u_{q}^{M}(\lieb^{+})$ and 
$\overline{F_{\beta}}=f_{\beta}\in u_{q}^{M}(\lieb^{-})$ for all $\beta \in Q^{+}$.
Alternatively, one may define  $u_{q}^{M}(\lieb^{\pm})$ 
as quotients of the Borel quantum groups 
$U^{M}_{q}(\lieb^{\pm})$ by
$$
u_{q}^{M}(\lieb^{+}) = U^{M}_{q}(\lieb^{+}) / 
\big(U^{M}_{q}(\lieb^{+})(Z_{M}^{\geq})^+\big), \qquad
u_{q}^{M}(\lieb^{-}) = U^{M}_{q}(\lieb^{-}) / 
\big(U^{M}_{q}(\lieb^{-})(Z_{M}^{\leq})^+\big),
$$
where $Z^{\geq}_{M}$ (resp. $Z^{\leq}_{M}$) is the
subalgebra of 
$U^{M}_{q}(\lieb^{+})$ (resp. $U^{M}_{q}(\lieb^{-})$)  generated by the elements 
$K_m^N$, $E_\alpha^N$ (resp. $K_m^N$, $F_\alpha^N$)
 for all $m \in M$ and $\alpha \in Q^{+}$.
We denote the generators of $u_{q}(\lieb^{+})$  by  $\overline{K_m} = h_{m}$ and 
$\overline{K_{\alpha_{i}}} = h_{i}$, for all $m \in M$ and $\alpha_i \in \Pi$.

%
%

\begin{obs}\label{rmk:ses-U-u}
Let $Q\subseteq M \subseteq P$ be a lattice. As $Z^{\geq}_{M}$
is a central Hopf subalgebra of $U^{M}_{q}(\lieb^{+})$ and   
$U^{M}_{q}(\lieb^{+})$ is a free $Z^{\geq}_{M}$-module of rank 
$N^{\dim \lieb^{+}}$, it follows from  \cite[Prop. 3.4.3]{Mo} that 
\begin{equation}\label{eq:sec-U-u} 
\xymatrix{Z^{\geq}_{M} \ar@{^{(}->}[r]^(0.4){\iota_{U}}& 
U^{M}_{q}(\lieb^{+})\ar@{->>}[r]^{\pi_{U}}& u^{M}_{q}(\lieb^{+}).}
\end{equation}
is a central exact sequence of Hopf algebras.
\end{obs}

\subsection{
Quantized function algebras at roots of unity}\label{subsec:Oq}

In this subsection we also follow the approach of \cite{BG}, using 
some key results from \cite{L2} and \cite{DL}.
Fix a connected, complex, simple affine 
algebraic group $G$, and let $\lieg$ be its
Lie algebra. Let $M$ be the lattice 
corresponding to the character group of a maximal torus of $G$,
with $Q \subseteq M \subseteq P$.
We write $G=G_{M}$ to stress this relation; in particular
$G_{P}=G_{sc}$ denotes the simply connected group and 
$G_{Q}=G_{ad}$
the adjoint group. 
Hereafter, we recall the definition of the quantized function algebra 
of $G_{M}$
and some results that will be used in the sequel. 
We also give as an example an explicit presentation for the 
case $G_{sc}=\SL_{\theta+1}(\C)$.

\smallbreak
Let $z$ be a formal variable and set $z_{i}=z^{d_{i}}$, for all $i\in I$.
Let $\mathcal{U}_{z}(\lieg)$ be the Lusztig integer form of 
$U_{z}(\lieg):=U_{z}^{Q}(\lieg)$
over the ring $R=\Q[z,z^{-1}]$, see \cite[III.7.2]{BG}.
Let $\mathcal{C}(\lieg,P)$ be the strictly full subcategory of 
$\mathcal{U}_{z}(\lieg)$-mod
whose objects are highest weight modules $V(\lambda)$ with $\lambda \in P$
such that $V(\lambda)$ is a free
$R$-module of finite rank and the operators
\[K_{\alpha_{i}},\qquad \qquad \left(\begin{smallmatrix}K_{\alpha_{i}};\,\, 0\\
	t\end{smallmatrix}\right) = 
\prod_{s=1}^{t}\left(\frac{K_{\alpha_{i}}
	z_{i}^{-s+1}-1}{z_{i}^{s}-1}\right)  \]
are diagonalizable with eigenvalues $z^{(\alpha_{i},\lambda)}$ and 
$\left(\begin{smallmatrix}k\\
t\end{smallmatrix}\right)_{z_{i}}$ 
respectively, for some $k \in
\N$ and for all $i\in I$; these are called also \emph{modules of type 1}.

Following \cite[Section 4.1]{DL}, we define $R_{z}(G_{sc})$
as the $R$-submodule of $\Hom_{R}(\mathcal{U}_{z}(\lieg),R)$ spanned by
the matrix coefficients $c^{V}_{i,j}$ of representations $V$ from
$\mathcal{C}(\lieg,P)$. These satisfy  
$$
c^{V}_{i,j}(x) = v^{i}(x\cdot v_{j}) 
,
$$
where
$\{v_{i}\}_{i\in I}$ is an $R$-basis of $V$, $\{v^{j}\}_{j\in I}$ 
is the dual basis of the
dual module and $x \in \mathcal{U}_{z}(\lieg)$. 
Since the subcategory $\mathcal{C}(\lieg,P)$ is
a tensor one, $R_{z}(G_{sc})$ is a Hopf algebra. In particular, the evaluation 
gives a non-degenerate Hopf pairing 
$\langle \,\ ,\,\ \rangle: R_{z}(G_{sc})\otimes  \mathcal{U}_{z}(\lieg) \to R$ 
and consequently $R_{z}(G_{sc})$ can be consider as a Hopf subalgebra
of $\mathcal{U}_{z}(\lieg)^{\circ}$.

For a lattice $Q\subseteq M \subseteq P$ of rank $\theta$, 
set $\mathcal{C}(\lieg,M)$ to 
be the subclass of $\mathcal{C}(\lieg,P)$
consisting of those modules which are direct sums of copies of 
highest weights modules $V(\lambda)$ for $\lambda \in M^{+}$.
Let $G_{M}$ be the connected simple affine algebraic group corresponding to 
the lattice $M$. 
We define the Hopf algebra $R_{z}(G_{M})$ as the $R$-submodule of 
$\mathcal{U}_{z}(\lieg)^{\circ}$ spanned by
the matrix coefficients $c^{V}_{i,j}$ of representations $V$ from
$\mathcal{C}(\lieg,M)$. Let $M'$ be the dual lattice of $M$ and 
$\mathcal{U}_{z}^{M'}(\lieg)$ be the Lusztig integer form of 
$U_{z}^{M'}(\lieg)$.
Then there exists a non-degenerate Hopf pairing 
$\langle \,\ ,\,\ \rangle: R_{z}(G_{M})\otimes  \mathcal{U}_{z}^{M'}(\lieg) \to R$,
see \cite[Section 4]{Gav2} for more details.

\smallbreak
Let $q\in \C$ be a primitive root of unity of odd 
order $N >1$  coprime with $3$ 
if the Cartan matrix $\mcar$ is of type $G_{2}$,
and let $p_{N}(z)$ 
be the minimal polynomial of $q$ over $\Q$. 
For dual lattices $Q\subseteq M$, $M' \subseteq P$, 
consider the quotient algebras 
 $ \mathcal{U}_{q}^{M'}(\lieg)_{\Q(q)}:= \mathcal{U}_{z}^{M'}(\lieg) / (p_{N}(z) \mathcal{U}_{z}^{M'}(\lieg))$
and $\Oc_{q}(G_{M})_{\Q(q)}:=R_{z}(G_{M}) / (p_{N}(z)R_{z}(G_{M}))$.
The Lusztig integral form of $\lieg$ over $M'$ at $q$ is 
$ \mathcal{U}_{q}^{M'}(\lieg):= \mathcal{U}_{q}^{M'}(\lieg)_{\Q(q)}\otimes_{\Q(q)} \C$ and the
{\it quantized function algebra of $G_{M}$ at $q$} is the Hopf algebra 
$\Oc_{q}(G_{M}):=\Oc_{q}(G_{M})_{\Q(q)}\otimes_{\Q(q)} \C$. 

In particular, there exists a non-degenerate Hopf pairing 
$\langle \,\ ,\,\ \rangle: \Oc_{q}(G_{M})\otimes  U_{q}^{M'}(\lieg) \to \C$
given by the evaluation.

\smallbreak
The next result summarizes some properties of $\Oc_{q}(G_{M})$, its 
proof follows \emph{mutatis mutandis} from \cite[Section III.7]{BG}, 
c.f. \cite[\S 4.4]{Gav2};
nevertheless we give a sketch for completeness.

\begin{teor}\label{OcentralenOe}
With the notation above, we have
\begin{enumerate}
\item[$(a)$] 
$\Oc_{q}(G_{M})$ contains a central Hopf subalgebra isomorphic to the
coordinate algebra $\Oc(G_{M})$ of $G_{M}$.
\item[$(b)$]  $\Oc_{q}(G_{M})$ is a projective
$\Oc(G_{M})$-module of rank $N^{\dim \lieg}$. 
\item[$(c)$]  The Hopf algebra quotient 
$\overline{\Oc_{q}(G_{M})}= \Oc_{q}(G_{M}) / [\Oc(G_{M})^{+}\Oc_{q}(G_{M})]$
has dimension $N^{\dim \lieg}$ and it 
is isomorphic to $\big(u_{q}^{M'}(\lieg)\big)^{*}$, the dual of
the Frobenius-Lusztig kernel of $\lieg$ at $q$.
\item[$(d)$] The quantized coordinate algebra $\Oc_{q}(G_{M})$ 
fits into the central exact sequence
\begin{equation}\label{extOOquq*}
\Oc(G_{M}) \hookrightarrow\Oc_{q}(G_{M}) \twoheadrightarrow 
\big(u_{q}^{M'}(\lieg)\big)^{*}.
\end{equation}
\end{enumerate}
\end{teor}

\begin{proof} (Sketch)
The quantum Frobenius map is a Hopf algebra map 
$Fr: \mathcal{U}_{q}^{M'}(\lieg)_{\Q(q)} \twoheadrightarrow U(\lieg)_{\Q}$
defined on integer forms and for any lattice $Q\subseteq M' \subseteq P$,
see \cite[\S 3.4]{Gav2}. By extending scalars over  $\C$
one gets a Hopf algebra epimorphism 
$Fr: {U}_{q}^{M'}(\lieg) \twoheadrightarrow U(\lieg)$.
Since $\Oc(G_{M})$ is a Hopf subalgebra of $U(\lieg)^{\circ}$, 
dualizing this epimorphism one gets a monomorphism  
$Fr^{\circ}: \Oc(G_{M}) \to U_{q}^{M'}(\lieg)^{\circ}$. Now, one has to prove
that $\operatorname{Im} (Fr^{\circ}) \subseteq \Oc_{q}(G_{M})$ and that 
$\operatorname{Im} (Fr^{\circ})$ is central in $U_{q}^{M'}(\lieg)^{\circ}$. Both facts follow
from the same proof as in  \cite[Theorem III.7.2]{BG}. For example, for the former one has
to prove that each finite-dimensional highest weight $U(\lieg)$-module 
$V(\lambda)$ with $\lambda \in (M')^{+}$ occurs as the specialization 
of $z$ to $q$ of a $\mathcal{U}_{z}^{M'}(\lieg)$-module:
one takes
 the irreducible $U_{z}^{M'}(\lieg)$-module $V(N\lambda)$ with 
 highest weight vector $v\in V(N\lambda)$ and consider the 
 $\mathcal{U}_{z}^{M'}(\lieg)$-module given by $\mathcal{U}_{z}^{M'}(\lieg)v$.
Item (c) follows from \cite[III.7.10]{BG} whereas (d) is a consequence of \cite[ Prop. 3.4.3]{Mo}.
\end{proof}

Now we recall some results from \cite[Sec. 2.1.3]{AG} about Lie 
subalgebras of $\lieg$ 
and the relations with integer forms of quantized 
enveloping algebras and quantized function algebras. These 
were given for the dual pair of lattices $P$ and $Q$, but the results
hold for any dual pair $M$ and $M'$, again \emph{mutatis mutandis}.

\smallbreak
Let $Q\subseteq M$, $M'\subseteq P$ be a dual pair of lattices or rank $\theta$.
For a Lie subalgebra $\liel \subseteq \lieg$ containing the Cartan 
subalgebra $\lieh$, the canonical monomorphism
$\mathcal{U}_{q}^{M'}(\liel)_{\Q(q)} \hookrightarrow \mathcal{U}_{q}^{M'}(\lieg)_{\Q(q)}$ 
of integral forms induces by duality 
an epimorphism of Hopf algebras 
$\Res: \big(\mathcal{U}^{M'}_{q}(\lieg)_{\Q(q)}\big)^{\circ}\twoheadrightarrow 
 \big(\mathcal{U}^{M'}_{q}(\liel)_{\Q(q)}\big)^{\circ}$.
As $\Oc_{q}(G_{M})_{\Q(q)}\subseteq  
\big(\mathcal{U}^{M'}_{q}(\lieg)_{\Q(q)}\big)^{\circ}$, we define
$$
 \Oc_{q}(L_{M})_{\Q(q)} = \Res(\Oc_{q}(G_{M})_{\Q(q)}).
$$
Moreover, as $\Oc(G_{M})_{\Q(q)} \subseteq \Oc_{q}(G_{M})_{\Q(q)}$, 
it follows that $\Res (\Oc(G_{M})_{\Q(q)})$ is a
central Hopf subalgebra of $\Oc_{q}(L_{M})_{\Q(q)} $ 
and whence there exists an 
algebraic subgroup $L_{M}$ of $G_{M}$ such that 
$\Res (\Oc(G_{M})_{\Q(q)} ) =\Oc(L_{M})_{\Q(q)} $. Furthermore, following 
\cite[Prop. 2.7]{AG} one can prove that $L_{M}$ is connected and the
corresponding Lie subalgebra $\Lie(L_{M})$
of $\lieg$ is $\liel$.

\smallbreak

Since $\Oc(L_{M})_{\Q(q)} $ is a central Hopf subalgebra of 
$\Oc_{q}(L_{M})_{\Q(q)}$, the
quotient
$$
\overline{\Oc_{q}(L_{M})}_{\Q(q)}:=  \Oc_{q}(L_{M})_{\Q(q)}/
[\Oc(L_{M})_{\Q(q)}^{+}\Oc_{q}(L_{M})_{\Q(q)}]
$$ 
is a Hopf algebra, which is actually also finite-dimensional. 

\smallbreak
From now on we fix the subalgebra $\liel \subseteq \lieg$ to be a 
Borel subalgebra $\liel=\lieb^{+}$ corresponding to our choice of the root system.
By the discussion above, there is a Borel subgroup $B^{+}_{M}$ of $G_{M}$
such that $\Lie(B^{+}_{M})=\lieb^{+}$.

\begin{prop}\label{prop:sucOB} 
Set $\Oc(B^{+}_{M}) = \Oc(B^{+}_{M})_{\Q(q)}\otimes_{\Q(q)} \C $ and 
$\Oc_{q}(B^{+}_{M}) = \Oc_{q}(B^{+}_{M})_{\Q(q)}\otimes_{\Q(q)} \C$.
\begin{enumerate}
 \item [\rm (a)] The following sequence of Hopf algebras is exact 
\[
\Oc(B^{+}_{M}) \hookrightarrow{} \Oc_{q}(B^{+}_{M}) \twoheadrightarrow
\overline{\Oc_{q}(B^{+}_{M})}.
\] 

 \item [\rm (b)] There exists an epimorphism of  Hopf algebras 
 $\overline{\Res}: \big(u_{q}^{M'}(\lieg)\big)^{*} \to
\overline{\Oc_{q}(B^{+}_{M})}$ such that the following diagram is
commutative
\[ \xymatrix{\Oc(G_{M}) \ar@{^{(}->}[r]^{\iota} \ar@{->>}[d]_{\Res} &
\Oc_{q}(G_{M})\ar@{->>}[r]^{\pi} \ar@{->>}[d]^{\Res} & \big(u_{q}^{M'}(\lieg)\big)^{*} \ar@{->>}[d]^{\overline{\Res}}\\
\Oc(B^{+}_{M}) \ar@{^{(}->}[r]^{\iota_{B^{+}}}  & \Oc_{q}(B^{+}_{M}) \ar@{->>}[r]^{\pi_{B^{+}}}&
\overline{\Oc_{q}(B^{+}_{M})} .}\]

\item [\rm (c)] $\overline{\Oc_{q}(B^{+}_{M})}\simeq 
\big(u_{q}^{M'}(\lieb^{+})\big)^{*}\simeq 
 u_{q}^{M}(\lieb^{+})$ as
Hopf algebras. In particular, $ \Oc_{q}(B^{+}_{M})$ fits into the central exact sequence
\[
\xymatrix{
\Oc(B^{+}_{M}) \ar@{^{(}->}[r]^{\iota_{B^{+}}}  & \Oc_{q}(B^{+}_{M}) \ar@{->>}[r]^{\pi_{B^{+}}}&
u_{q}^{M}(\lieb^{+}) .}
\]
\end{enumerate}
\end{prop}

\begin{proof}
Items (a) and (b) follows from  \cite[Prop. 2.8 (a), (b)]{AG} 
\emph{mutatis mutandis}
for $\liel = \lieb^{+}$ and $L=B^{+}_{M}$, eventually applying the
exact funtor $\_\,\,\otimes_{\Q(q)} \mathbb{C}$. 

(c) Following \cite[\S III.7.3]{BG} one may prove that $\Oc_{q}(B_{M}^{+})$
is a finitely generated projective module over $\Oc(B_{M}^{+})$ of rank 
$N^{\dim \lieb^{+}}$, which implies that 
$\dim \overline{\Oc_{q}(B^{+}_{M})} = N^{\dim \lieb^{+}}$.
On the other hand, from the presentation we have that 
the subalgebra $u_{q}^{M'}(\lieb^{+})$
of $u_{q}^{M'}(\lieg)$ also has  dimension  $N^{\dim \lieb^{+}}$. Hence, to prove that they are 
isomorphic, it is enough to prove the existence of a Hopf algebra epimorphism
between them. This follows from a similar argument used in  
\cite[Prop. 2.8 (c)]{AG}. 

Finally, $\big(u_{q}^{M'}(\lieb^{+})\big)^{*}\simeq 
u_{q}^{M}(\lieb^{+})$ follows from the isomorphisms 
 $\big(u_{q}^{M'}(\lieb^{+})\big)^{*}\simeq 
u_{q}^{M}(\lieb^{-})^{\cop}$ and 
 $\big(u_{q}^{M}(\lieb^{-})\big)^{\cop}\simeq 
u_{q}^{M}(\lieb^{+})$.
The former is a consequence of
the existence of a non-degenerate Hopf pairing
$(\ ,\ ):\, U^{M}_{q}(\lieb^{-})^{\cop} \ot U_{q}^{M'}(\lieb^{+}) \to \C$ 
in \eqref{eq:pairingb+b-}, the latter
can be directly constructed from the presentation.	
\end{proof}

\subsubsection{
Quantized function algebras for the $A_{\theta}$-case.}\label{subsec:OqSLn}
Suppose that 
$G=G_{sc}=\SL_{\theta+1}(\C)$ and
$\lieg = \liesl_{\theta+1}$. In this case, the category $\mathcal{C}(\lieg,P)$
is generated by the standard representation.
This allows us to present $\Oc_{q}(\SL_{\theta+1})$ by generators and 
relations as a quotient of a bialgebra by the two-sided ideal
generated by the central element $\det_q -1$.
Here, $\det_q$ is the so-called \emph{quantum determinant}.

\smallbreak
The bialgebra $\cO_q(\text{M}_{\theta+1})$ is generated by 
the elements $z_{ij}$,  $1\leq i,j,\leq \theta+1$, satisfying the following \emph{$q$-matrix} relations
\begin{align}
\label{rel-gln-1}&z_{is}z_{js}=qz_{js}z_{is},& &z_{si}z_{sj}=qz_{sj}z_{si},& &i<j,&\\
\label{rel-gln-2}&z_{it}z_{js}=z_{js}z_{it},& &z_{is}z_{jt}-z_{jt}z_{is}=(q-q^{-1})z_{it}z_{js},& &i<j,\,\,s< t.&
\end{align} 
The comultiplication and the counit 
in $\cO_q(\text{M}_{\theta+1})$ are given by
\[
\Delta(z_{ij})=\sum_{t=1}^{\theta+1}z_{it}\otimes z_{tj},\qquad 
\eps(z_{ij})=\delta_{ij},\quad 
\text{for all }\, 1\leq i,j,\leq \theta+1.
\]
The quantum determinant $\det_{q}$ is the central group-like element given by 
$$
{\det}_q = \sum_{\sigma\in \mathbb{S}_{\theta+1}}
(-q)^{\ell(\sigma)}z_{1\sigma(1)}\cdots z_{\theta+1\sigma(\theta+1)}
\,\,\in \,\, \cO_q(\text{M}_{\theta+1}),
$$ 
where $\ell(\sigma)$ denotes the length of $\sigma$. The elements $z_{ij}= c^{V}_{ij}$ 
represent the matrix coefficients of the standard representation $V= \C^{\theta+1}$
of  $\,\,\mathcal{U}_{q}^{P}(\liesl_{\theta+1})$.
Since the quantum determinant is a central group-like element, 
one may consider the 
Hopf algebra $\cO_q(\text{M}_{\theta+1})/(\text{det}_q-1)$, which is isomorphic to
$\cO_q(\text{SL}_{\theta+1})$, 
the quantized function algebra of $\text{SL}_{\theta+1}(\C)$ at $q$. 

\bigbreak

Let $B^{+}$ the subgroup of $\SL_{\theta+1}(\C)$ of upper triangular matrices. 
We denoted by $\cO_q(B^{+})$ the Hopf algebra quotient
$\cO_q(\SL_{\theta+1})/\mathcal{I}$ of $\cO_q(\SL_{\theta+1})$ by the 
Hopf ideal $\mathcal{I}$ generated by 
the elements $\{z_{ij}\,:\,i> j\}$. 
In particular, $\cO_q(B^{+})$ is generated by $z_{ij}$ with $i\leq j$, 
satisfying the relations \eqref{rel-gln-1}, \eqref{rel-gln-2} and 
$z_{11}z_{22}\cdots z_{\theta+1,\theta+1}=1$.

\subsection{Finite-dimensional pointed Hopf algebras of Cartan type} \label{sec-datum}
In this subsection we recall the classification of finite-dimensional pointed complex 
Hopf algebras over an abelian group and whose infinitesimal braiding is a
braided vector space of diagonal Cartan type. 
They consist of multiparameter variations of small quantum groups $u_{q}^{M}(\lieb^{+})$.
We follow \cite{AS2} for the presentation.

Although there are certain restrictions on the finite abelian groups involved,
these are not completely 
determined and still on \emph{minimal assumptions}
there is some freeness which allows one to 
consider different lattices $M$ between the lattices $Q$ and $P$.

\subsubsection{Data of finite Cartan type}\label{subsec:data-Cartan-type}
Let $\Gamma$ be a finite abelian group.
A datum $\cD$ of \textit{finite Cartan type} for $\Gamma$ is a tuple
\[\cD=\cD(\Gamma,(g_i)_{i\in I},(\chi_i)_{i\in I}, (a_{ij})_{i,j\in I})\]
consisting of elements $g_i\in \Gamma$, $\chi_i\in \hat{\Gamma}$ 
and a Cartan matrix $(a_{ij})_{i,j\in I}$ of finite type satisfying the \emph{Cartan relation}: 
\begin{align}\label{Cartan-type-trenza}
q_{ij}q_{ji}=q_{ii}^{a_{ij}}, \qquad\text{ for all }i,j\in I,
\end{align}
where $q_{ij}=\chi_{j}(g_i)$ for all $i,j\in I$, 
and we assume that $q_{ii}\neq 1$, for all $i\in I$. 

\bigbreak
Since our focus is on finite-dimensional Hopf algebras,
we further assume that the order of $q_{ii}$ is odd, for all $i\in I$, 
and the order of $q_{ii}$ is prime to $3$, 
for all $i$ in a connected component of 
type $G_{2}$. From  
(\ref{Cartan-type-trenza}) one gets that 
$q_{ii}^{a_{ij}} = q_{jj}^{a_{ji}}$ for all $i,j\in I$. Then, 
 for each connected component $J$
of the Dynkin diagram 
there exists $j_{0} \in I$  such that $q_{ii} = q_{j_0 j_0 }^{e _i} $
for some  $e_i \in \N$,  for all  $i \in J $. 
In particular, the order $N_{i}$ 
of $q_{ii}$ is constant in each connected
component $J$; whence we set $N_{J} = N_{i}$,
for all $i \in J$. We also set 
$q_{j_{0}} = q_{j_{0}j_{0}}^{1/2}$ and  
$q_i := q_{j_0}^{\,e _i} $  for all  $ i \in J$.
In the case where the Dynkin diagram is connected,
we write $q:= q_{j_{0}}$ and denote the order $N_{J}$ simply by $N$. 

\smallbreak
Given a positive root 
$\alpha= \sum_{i=1}^{\theta} n_{i}\alpha_{i} \in Q^{+}$, 
$n_{i}\in \N_{0}$, we write 
\[
g_{\alpha}= \prod_{i=1}^{\theta} g_{i}^{n_{i}},\qquad\qquad \chi_{\alpha}=
 \prod_{i=1}^{\theta} \chi_{i}^{n_{i}}. 
\]
In other words, we have group homomorphisms $g: Q\to \Gamma$ and 
$\chi:Q\to \hat{\Gamma}$ given by 
$g(\alpha_{i}) = g_{i}$ and $\chi(\alpha_{i}) = \chi_{i}$, for all $i\in I$, respectively, and 
we write $g_{\lambda} = g(\lambda)$ and $\chi(\lambda) = \chi_{\lambda}$ for any 
$\lambda \in Q$.

\smallbreak
\begin{defi}\label{def-contain-group}
Given any lattice $M$ of rank $\theta$
with $Q\subseteq M \subseteq P$, we say that
{\it an abelian group $\Gamma$  contains} $M$ if there exist
group homomorphisms $g: M\to \Gamma$, $\chi: M\to \hat{\Gamma}$ such that $g(\alpha_{i}) = g_{i}$ 
and $\chi(\alpha_{i}) = \chi_{i}$
for all $i\in I$. In particular, $\Gamma$ contains elements $g_{\lambda}=g(\lambda)$, 
for $\lambda \in M$, satifying $g_{\lambda + \mu} = g_{\lambda}g_{\mu}$ and 
$g_{0} = 1$.	
\end{defi}

 Note that any abelian group $\Gamma$ that is part of a non-trivial datum of 
 finite Cartan type contains the root lattice $Q$, 
but not necessarily $\Gamma$ contains $M$ for $M\neq Q$.
For a finite abelian group $\Gamma$ containing a lattice $M$, we denote by $\cD_{M} = 
\cD(\Gamma, (g_i)_{i\in I},(\chi_i)_{i\in I}, (a_{ij})_{i,j\in I})$
to stress this property. 

As claimed in \cite{AS3},
the explicit classification of all data of finite Cartan type for a given
finite abelian group $\Gamma$ is a computational problem; at least it is bounded.

\smallbreak
Given a datum $\cD$ of finite Cartan type for $\Gamma$ with indecomposable 
Cartan matrix, we will recall in the next 
subsection how to construct a finite-dimensional 
pointed Hopf algebra $u(\cD,\mu)$,
depending on a family of parameters $\mu$, 
whose coradical is isomorphic to $\C \Gamma$. 

\smallbreak
A family of \emph{root vector parameters}
is a family of elements $\mu=(\mu_{\alpha})_{\alpha\in Q^{+}}$  in 
$\C$ such that for all 
$\alpha \in Q^{+}$, 
\begin{equation}\label{eq:cond-lifting-param}
\text{ if }\quad g_{\alpha}^{N}=1 \quad\text{ or }\quad\chi_{\alpha}^{N}\neq \varepsilon
\qquad\text{ then }\quad\mu_{\alpha}=0. 
\end{equation}

\vskip11pt

\subsubsection{Finite-dimensional pointed Hopf algebras $u(\cD,\mu)$}\label{subsec:pres-u}
In this subsection we follow \cite{AS3}.
Given a datum of finite Cartan type $\cD$ and a family
of parameters  $\mu$ satifying \eqref{eq:cond-lifting-param}, 
one defines the Hopf algebra 
$u(\cD,\mu)$ as the algebra generated 
by the elements $g$, $x_i$ with $g\in \Gamma$ and $1\leq i\leq \theta$, 
satisfying the following relations:

\begin{align}
gx_ig^{-1}&=\chi_i(g)x_i,& & \text{for all }\, i \text{ and } g\in\Gamma,& 
\nonumber \\
(\ad x_i)^{1-a_{ij}}(x_j)&=0,& &\text{for all }\, i\neq j,&
\nonumber \\
\label{eq-HA-4}	x_{\alpha}^{N}&=r_{\alpha}(\mu),& &\text{for all }\, \alpha\in Q^{+}.&
\end{align}
Here, $(\ad x_i)^{1-a_{ij}}(x_j)$
denotes the composition of the
braided adjoint action defined in \eqref{eq:braided-adjoint-action}, 
$x_{\alpha}$ are root vectors 
which are obtained via iterated braided
commutators of the elements $\{x_{i}\}_{i\in I}$, and
$r_{\alpha}(\mu)$ are elements in the augmentation ideal of 
$\C[g_{i}^{N_{i}}\, |\, i\in I]$ 
that are constructed in a recursive way as
\begin{equation}\label{eqn:relation-recursion}
r_{\alpha}(\mu) = \mu_{\alpha} (1 - g_{\alpha}^{N}) + \sum_{\beta,\gamma\neq 0, \beta+\gamma = \alpha}
t_{\beta,\gamma}^{\alpha} \mu_{\beta}r_{\gamma}(\mu),
\end{equation}
where $t_{\beta,\gamma}^{\alpha}\in \C$ are certain scalars determined
by the comultiplication. Actually, one may prove  by induction 
on the length of $\alpha$ that the elements  
$r_{\alpha}(\mu)$ lie in the augmentation ideal of 
$\C[g_{\alpha}^{N}\, |\, \chi_{\alpha}^{N}\neq \eps]$. 
The coalgebra structure is given by setting 
$g\in G(u(\cD,\mu))$ for all $g\in \Gamma$ and $x_{i} \in P_{1,g_{i}}(u(\cD,\mu))$, i.e.
$$
\com(x_i)=x_i\otimes 1+g_i\otimes x_i,\qquad \eps(x_{i})=0,\qquad
\com(g)=g\otimes g,\qquad \eps(g)=1.
$$

\begin{obs}\label{rmk:general-construc} Let $\cD$ be a datum of finite Cartan type such 
that the Cartan matrix (in the datum) is not necessarily indecomposable. 
In this case, there is a Hopf algebra $u(\cD,\lambda,\mu)$ which depends 
on the Cartan 
datum and on two families of parameters $\lambda$ and $\mu$. 
The family of scalars $\lambda$ corresponds to a relation that is absent when the 
Cartan matrix is indecomposable, which is precisely the case we are considering. 
For this reason, we simply write $u(\cD,\mu)$. 
	
\end{obs}	

The family of Hopf algebras  $u(\cD,\lambda,\mu)$ is crucial for 
the classification 
problem of finite-dimensional complex pointed Hopf algebras
because of the following key result, which is one of the quite few structural results
in Hopf algebra theory. 

\begin{teor}\cite{AS3}\label{thm:AS}
Following the notation above, we have:
\begin{enumerate}
\item[\rm (a)] $u(\cD,\lambda,\mu)$ is a pointed 
Hopf algebra of dimension $\dim u(\cD,\lambda,\mu)=N^{|\Phi^+|} \ord \Gamma $ 
with $u(\cD,\lambda,\mu)_{0} = \C \Gamma$.
\item[\rm (b)] Let $A$ be a finite-dimensional pointed Hopf algebra  with abelian
group $\Gamma= G(A)$ and infinitesimal braiding matrix 
$(q_{ij})_{1\leq i,j\leq \theta}$. Assume that
\begin{enumerate}
	\item [$\circ$] $N=\ord q_{ii}$ is odd, for all $1\leq i\leq \theta$,
	\item [$\circ$] $N=\ord q_{ii}$ is prime to $3$ if $q_{il}q_{li}\in \{q^{-3}_{ii},q^{-3}_{ll}\}$ for some $l$,
	\item [$\circ$] $N=\ord q_{ii}>7$, for all $1\leq i\leq \theta$.
\end{enumerate}
Then $A \simeq u(\cD,\lambda, \mu)$ for some $\cD,\lambda, \mu$.
\end{enumerate}
\end{teor}

It turns out that the 
associated graded Hopf algebra $\gr u(\cD,\mu)$ of  $u(\cD,\mu)$ is isomorphic to 
a bosonization $\toba(V)\#\C \Gamma$ for some $V\in\ydga$. In other words, 
$u(\cD,\mu)$ is a lifting of $\toba(V)\#\C \Gamma$.
The Nichols algebra $\toba(V)$ is completely determined by the Cartan datum; as an 
algebra it is 
generated by the elements $\{x_{i}\}_{i\in I}$ satisfying the relations 
$(\ad x_i)^{1-a_{ij}}(x_j)=0$ for all $i\neq j$, and 
$x_{\alpha}^{N}=r_{\alpha}(\mu)$ for all $\alpha\in Q^{+}$.

\begin{obs}\label{rmk:relation-recursion}
The main obstruction to give an explicit presentation of
the families of pointed Hopf algebras $u(\cD,\mu)$ is the recursion 
formula \eqref{eqn:relation-recursion}: 
$x_{\alpha}^{N}=r_{\alpha}(\mu)$ for all $\alpha\in Q^{+}$. 
It contains implicitly the lifting data
for the corresponding Nichols algebra.
Despite the algorithmic nature of the explicit presentation 
of the elements 
$r_{\alpha}(\mu)$, few families are completely known. 
Among them one finds
\cite[Sec. 6]{AS2} and \cite{AAG} for type $A_{\theta}$, \cite{BDR} for type $B_{2}$, 
\cite{BGM} for type $B_3$ and
\cite{GIG} for type $G_{2}$. The main goal of this paper is to 
answer a question on \cite{AS3} and find an
explicit algorithm describing the elements $r_{\alpha}(\mu)$ for any 
connected Dynkin diagram by using quantum subgroups.  
As an example, we use our method to reobtain the explicit formula for type 
$A_{\theta}$, $B_{\theta}$ and $D_{\theta}$ 
in Section \ref{sec:PHa-qsubgroups}.
\end{obs}

\subsubsection{Drinfeld-Jimbo braiding}
Let $q\in \C^{\times}$ be a primitive root of unity of odd order 
$N$, and coprime with $3$ if the Cartan matrix $\mcar$ has 
a component of type $G_{2}$.  
A diagonal braiding of Cartan type such that 
\begin{align*}
q_{ij}=q^{(\alpha_i,\alpha_j)}=q^{d_ia_{ij}}=q^{d_ja_{ji}},\quad\text{for all } i,j\in I,
\end{align*}
 is said to be of \textit{Drinfeld-Jimbo type}.

\section{Pointed Hopf algebras as quantum subgroups}\label{sec:ces}

From now on, $\mcar=(a_{ij})_{i,j\in I}$ denotes an indecomposable 
symmetrizable Cartan matrix of finite type,  
$q\in \C^{\times}$ is a primitive root of unity of odd order $N$, which 
we assume bigger than $3$ if $A$ is of type $G_{2}$
and $q_{ii}$ is a power of $q$ with $\ord q_{ii} = N$, for all $i\in I$.

\smallbreak
Let $M$ be a lattice of rank $\theta$ 
with $Q\subseteq M \subseteq P$ and $\Z$-basis 
$\{m_{1},\ldots, m_{\theta}\}$, $\ell_{1},\ldots, 
\ell_{\theta}\in \N$
 and $\Gamma = \Z^{\theta}/(\ell_{1}N\Z\times \cdots \times \ell_{\theta} N\Z)$. 
Denote by 
$g:M\to \Gamma$ the canonical group epimorphism 
sending the $\Z$-basis of $M$ to the canonical generators of $\Gamma$ and write
$g_{m} = g(m)$, for all $m\in M$. Observe that
$N | \ord (g_{m_{i}})$ for all $i\in I$. Also, any element of 
$\Gamma$ can be written as $g_{m}$ for some $m\in M$.
In particular, $\Gamma$ contains $M$.

\smallbreak
\subsection{Braidings of Drinfeld-Jimbo type}\label{subsec:DJ}
In this subsection we assume that 
the braiding ${\bf q}=(q_{ij})_{i,j\in I}$ is of Drinfeld-Jimbo type.
Set $q_i=q^{d_i}$; then $q_{ii} = q_{i}^{2}$, for all $i\in I$.

\subsubsection*{The Cartan datum for any lattice $M$}\,
\smallbreak
	
For $j\in I$, let $\chi_{j}: \Gamma \to \C^{\times}$ be
the character defined by $\chi_{j}(g_{m_{i}}) = q^{(m_{i},\alpha_{j})} $
for all $i\in I$.
Note that $\chi_{j}(g_{\alpha_{i}}) = q^{(\alpha_{i},\alpha_{j})} = q^{d_{i}a_{ij}}$ for all $i,j\in I$. Thus, taking
 $g_{i} = g_{\alpha_{i}}$ for all $i\in I$, we have that
$\cD_{M} = \cD(\Gamma, (g_{i})_{i\in I}, (\chi_{i})_{i\in I}, (a_{ij})_{i,j\in I})$ is a
datum of finite Cartan type. 

\subsubsection*{Pointed Hopf algebras for any lattice $M$}\,

\smallbreak
Fix a family $\mu = (\mu_{\alpha})_{\alpha \in \Phi^{+}}$
of root vector parameters. The algebra  $u(\cD_{M},\mu)$ is given explicitly by:
\begin{align}
	\label{eq:delulmu1}
	g_{m}x_jg^{-1}_{m}&=q^{(m,\alpha_{j})} x_j,& & \text{for all }\,
	m\in M\text{ and } g\in\Gamma,& \\
	\label{eq:delulmu2}	\ad(x_i)^{1-a_{ij}}(x_j)&=0,& &\text{for all }\, i\neq j,&\\
	\label{eq:delulmu3}	x_{\alpha}^{N}&=r_{\alpha}(\mu),& &\text{for all }\, \alpha\in \Phi^{+},&
\end{align}
where  $r_{\alpha}(\mu)$ is given by \eqref{eqn:relation-recursion}.

\subsubsection*{The algebra  $u(\cD_{M},\mu)$ as a central extension}\,

\smallbreak	
The following proposition asserts that one may construct the Hopf algebras 
$u(\cD_{M},\mu)$ as central extensions of small quantum groups by
finite-dimensional group algebras. 
In order to state it, 
we recall from \S \ref{subsec:smallqgr} 
that for the generators of $u_{q}^{M}(\lieb^{+})$ we write
$\overline{K}_{m}: = k_{m}$ with $m \in M$, $\overline{K}_{\alpha_{i}} = k_{i}$, and
$\overline{E}_{i}: = e_{i}$ for all $i \in I$.  

\begin{prop}\label{prop:ces-u}
Let $M$, $\Gamma$, $\cD_{M}$ and
$\mu = (\mu_{\alpha})_{\alpha \in Q^{+}}$ be as above. 
Let $\subg$ be the subgroup of 
$\Gamma$ generated by $g_{m_{i}}^{N}$ for all $i\in I$.
\begin{enumerate}
\item[\rm (a)] There exists a Hopf algebra
epimorphism $\pi_u: u(\cD_{M},\mu) 
\twoheadrightarrow u_{q}^{M}(\lieb^{+})$ 
defined by $\pi_u(g_{m_{i}}) = k_{m_{i}}$ and 
$\pi_u(x_{i})= e_{i}$, for all $i\in I$. 
\vspace{.1cm}
\item[\rm (b)] $u(\cD_{M},\mu)^{\co\pi_u} =\C \subg$.
\vspace{.1cm}
\item[\rm (c)] There exists a central exact sequence of Hopf algebras 
\begin{equation}\label{eq:ces-u}
\xymatrix{\C \subg  \ar@{^{(}->}[r]^>>>>{\iota_{u}}& 
u(\cD_{M},\mu)\ar@{->>}[r]^{\pi_{u}}& u_{q}^{M}(\lieb^{+}).}
\end{equation}
In particular, $\dim u(\cD_{M},\mu) = |\subg|\dim u_{q}^{M}(\lieb^{+})$.\end{enumerate}
\end{prop}

\begin{proof}
$(a)$ 
Observe that $k_{m}^{N}= \overline{K}_{m}^{N} =  1$ and  
$\pi_{u}(g_{m}) = k_{m}$ for all $m\in M$; in particular,
$\pi_{u}(g_{i}) = k_{i}$ for all $i\in I$.
Let us see that $\pi_{u}$ is a 
well-defined algebra map. For $g_{m}\in \Gamma$ 
with $m\in M$, we have 
$\pi_{u}(g_{m}x_{j}g_{m}^{-1})= k_{m}e_{j}k_{m}^{-1} =
q^{(m,\alpha_{j})} e_{j} = \pi_{u}(q^{(m,\alpha_{j})} x_{j})$ 
and \eqref{eq:delulmu1} is satisfied.
As $\pi_{u}(x_{i}) = e_{i}$, it follows that $\pi_{u}$ 
sends relation \eqref{eq:delulmu2} to the quantum Serre relation 
in $u_{q}^{M}(\lieb^{+})$, hence \eqref{eq:delulmu2} is also satisfied. 
Finally, using the recursively formula \eqref{eqn:relation-recursion}
we obtain that  
$\pi_{u} (r_{\alpha}(\mu))=0=\pi_{u}(x_{\alpha})^{N}$
for all $\alpha \in Q^{+}$, since $r_{\alpha}(\mu)$ are elements in the augmentation ideal 
	of $\C[g_{\alpha_{i}}^{N}\, |\, i\in I]$ and
	 $\pi_{u}(g_{\alpha_{i}}^{N}) = \eps(g_{\alpha_{i}}^{N})=1$, for all $i\in I$.
	
	We prove $(b)$ and $(c)$ together. Since $\pi_{u}(g_{m_{i}}^{N}) = 1$ for all $i\in I$, 
	we have that $\C \subg\subseteq 
	u(\cD_{M},\mu)^{\co\pi}$. Let us see the reverse inclusion. 
	As $q$ is an $N$-th root of unity, $\C \subg$ is a 
	central Hopf subalgebra of $u(\cD_{M},\mu)$. 
	Moreover, 
	from the presentation of both $u(\cD_{M},\mu)$ and 
	$u_{q}^{M}(\lieb^{+})$ and since the algebras $u(\cD_{M},\mu) / u(\cD_{M},\mu)\C \subg^{+}$ 
	and $u_{q}^{M}(\lieb^{+})$ have the same dimension, 
	it follows that $u(\cD_{M},\mu) / u(\cD_{M},\mu)\C \subg^{+} \simeq u_{q}^{M}(\lieb^{+})$.
	Since $u(\cD_{M},\mu)$ is free over $\C \subg$, it is faithfully flat and by \cite[Prop. 3.4.3]{Mo}
	we have that $u(\cD_{M},\mu)^{\co\pi} = \C \subg$ and $u(\cD_{M},\mu)$ fits into the central 
	exact sequence 
$\xymatrix{
\C \subg  \ar@{^{(}->}[r]^(0.4){\iota_{u}}& 
u(\cD_{M},\mu)\ar@{->>}[r]^{\pi_{u}}& u_{q}^{M}(\lieb^{+}).
}$ 
\end{proof}

The next proposition gives a relation between the exact 
sequences \eqref{eq:sec-U-u} and \eqref{eq:ces-u}.
In particular, the algebras $u(\cD_{M},\mu)$ 
fit into a commutative diagram with exact rows.

\begin{prop}\label{prop:U-u}
Let $M$, $\Gamma$, $T$, $\cD_{M}$ and
$\mu = (\mu_{\alpha})_{\alpha \in Q^{+}}$ be as above.
\begin{enumerate}
		\item[\rm (a)] There exists a Hopf algebra
		epimorphism $p: U^{M}_{q}(\lieb^{+})  \twoheadrightarrow u(\cD_{M},\mu)$ 
		defined by $p(K_{m_{i}}) = g_{m_{i}}$ and $p(E_{i})= x_{i}$, for all $i\in I$.
		\item[\rm (b)] Write $p_{0}$ for the restriction 
		of $p$ to the subalgebra $Z^{\geq}_{M}$. Then the 
		image of $p_{0}$ is $\C \subg$ and the following is a commutative diagram with exact rows
\begin{align}\label{eq:diag-U-u-teo}
\begin{aligned}
\xymatrix{
Z^{\geq}_{M} \ar@{^{(}->}[r]^(0.4){\iota_{U}} \ar@{->>}[d]_{p_{0}} & 
U^{M}_{q}(\lieb^{+})\ar@{->>}[r]^{\pi_{U}}\ar@{->>}[d]^{p}
& u^{M}_{q}(\lieb^{+}) \ar@{=}[d]^{\id}\\
\C \subg  \ar@{^{(}->}[r]^>>>>{\iota_{u}}& 
u(\cD_{M},\mu)\ar@{->>}[r]^{\pi_{u}}& u_{q}^{M}(\lieb^{+}).}
\end{aligned} 
\end{align}
\end{enumerate}
\end{prop}

\begin{proof}
It is straightforward to check (a). For the item (b), just note that $p_0(K_{m}^N)=g_m^N$ 
and $p_0(E_{{\alpha}}^N)=x_{\alpha}^N
=r_{\alpha}(\mu)\in\C[g_{\alpha_{i}}^{N}\, |\, i\in I]\subset \C \subg$.
\end{proof}

\begin{obs}\label{rmk:zeta+}

By \cite[\S 1.3]{Gav},
we have that $U^{M}_{q}(\lieb^{+}) 
\simeq \cO_{q}(B_{M}^{+})$ for all lattices 
$Q\subseteq M \subseteq P$.
For instance, when $G=SL_{\theta +1}(\C)$ (i.e. $M=P$), 
we have from  \cite[\S 2.4]{Gav} that the previous isomorphism 
is given explicitly by
$$
\psi: U_{q}^{P}(\lieb^{+}) \to \cO_{q}(B_P^{+}),
\qquad K_{\omega_{i}}^{\pm 1}\mapsto (z_{1,1}\ldots z_{i,i})^{\pm 1},
\quad E_{j} \mapsto - (q-q^{-1})^{-1}z_{j,j+1}z^{-1}_{j+1,j+1}
$$
for all $1\leq i,j\leq \theta$. 
The inverse $\psi^{-1}: \cO_{q}(B_P^{+})\to U_{q}^{P}(\lieb^{+})$ is given by 
\[
\psi^{-1}(z_{i,i})=K_{\omega_i}K_{\omega_{i-1}}^{-1},\quad \psi^{-1}(z_{i,i+1})=
-(q-q^{-1})^{-1}E_iK_{\omega_i}^{-1}K_{\omega_{i+1}},
\]
where by convention we assume that $K_{\omega_0}=K_{\omega_{\theta+1}}=1$.
\end{obs}

\subsubsection*{The algebra  $u(\cD_{M},\mu)$ as a quantum subgroup}\,

\smallbreak
Using the previous remark and Propositions \ref{prop:sucOB} and  
\ref{prop:U-u}, we prove our first main result.

\begin{teor}\label{teo:quantum-subgrup-DJ} Let $M$, $\subg$ and 
$\mu = (\mu_{\alpha})_{\alpha \in Q^{+}}$ be as above.
There exists a commutative diagram of Hopf algebras 
whose the rows are exact sequences
	
\begin{align}\label{diag:theorem-liftings}
  \begin{aligned}
	\xymatrix{\cO(B_{M}^{+}) \ar@{^{(}->}[r]^{\tilde{\iota}} \ar@{->>}[d]_{\tilde{p}_0} & 
	  \cO_q(B_{M}^{+}) \ar@{->>}[r]^{\pi_{B^+}} \ar@{->>}[d]^{\tilde{p}} & \overline{\cO_{q}(B_{M}^{+})} \ar[d]^{\phi}\\
		\C^{\widehat{\subg}}\ar@{^{(}->}[r]^>>>>>{\tilde{\iota}_u} & 
		u(\cD_{M},\mu)  \ar@{->>}[r]^{{\pi}_u} & {u}_q^{M}(\mathfrak{b}^{+}).} 
  \end{aligned} 
\end{align}In particular, the Hopf algebra $u(\cD_{M},\mu)$ is a quantum subgroup of 
	$\cO_q(B_{M}^{+})$. 
\end{teor}

\begin{proof}
Denote by $\psi:U_{q}^{M}(\lieb^{+})\to \cO_{q}(B^{+}_{M})$
the Hopf algebra isomorphism mentioned in Remark \ref{rmk:zeta+}. 
It holds that $\psi^{-1}(\cO(B^{+}_{M}))=Z^{\geq}_{M}$. Indeed,
from the proof of Proposition \ref{OcentralenOe}, it follows that
the Hopf algebra $\cO(B^{+}_{M})$  
is the image of the dual of the quantum Frobenius map 
 $\text{Fr}: {U}^{M'}_{q}(\lieb^{+}) 
\twoheadrightarrow U(\lieb^{+})$.
Hence, it consists 
of the elements in $\Img \text{Fr}^{t}$ that annihilate the kernel 
(with respect to the Hopf pairing between ${U}^{M'}_{q}(\lieb^{+})$ and
$\cO_{q}(B^{+}_{M})$),
the latter being generated by the elements $K_{m'}-1$ and $E_{j}$
for 
$m\in M'$ and $j\in I$. 
On the other hand, we have the non-degenerate pairing 
$(\ ,\ ):\, U^{M}_{q}(\lieb^{-})^{\cop} \ot U_{q}^{M'}(\lieb^{+}) \to \C$
from \eqref{eq:pairingb+b-}  given by: 	%
	$$
	(K_{m},K_{m'})  = q^{-(m,m')}, 
	\qquad (K_{m},E_{j}) = 0,\qquad
	(F_{i},E_{j}) = -\frac{\delta_{ij}}{q_{i}-q_{i}^{-1}},\qquad
	(F_{i}, K_{m}) = 0,
	$$
	%
for all $i,j\in I$, $m\in M$ and $m'\in M'$.
The subalgebra of $U_{q}^{M}(\lieb^{-})^{\cop}$ that annihilates the
kernel of the Frobenius map through this pairing is the one 
generated by the elements 
$K_{m}^{\pm N}$, $F_\alpha^N$ for all $m\in M$
and $\alpha \in Q^{+}$. Through the isomorphism 
$U_{q}^{M}(\lieb^{-})^{\cop}\simeq U_{q}^{M}(\lieb^{+})$, 
the latter correspond to
the elements $K_{m}^{\mp N}$, $E_\alpha^N$.
This is exactly 
the subalgebra $Z^{\geq}_{M}$ of 
$U_{q}^{M}(\lieb^{+})$, see \S \ref{subsec:smallqgr}. 
Hence, we have a commutative diagram whose vertical
rows are isomorphisms
$$
\xymatrix{\cO(B_{M}^{+}) \ar@{^{(}->}[r]^{\tilde{\iota}} \ar@{->}[d]_{\psi^{-1}} & 
	  \cO_q(B_{M}^{+}) \ar@{->>}[r]^{\pi_{B^+}} \ar@{->}[d]^{\psi^{-1}} & \overline{\cO_{q}(B_{M}^{+})} \ar[d]^{\phi}\\
		Z_{M}^{\geq }\ar@{^{(}->}[r]^>>>>>{{\iota}_U} & 
		U^{M}_{q}(\lieb^{+}) \ar@{->>}[r]^{{\pi}_U} & {u}_q^{M}(\mathfrak{b}^{+}).} 
$$

\smallbreak
Now, to finish the proof, we use the diagram in Proposition \ref{prop:U-u}
and show that the new diagram is commutative.
As $\subg$ is an abelian finite group, 
there exists a Hopf algebra isomorphism 
$\varphi: \C\subg \to \C^{\widehat{\subg}}$.
Consider $\tilde{\iota}=\psi\circ\iota_U\circ\psi^{-1}$, 
$\tilde{p}=p\circ\psi^{-1}$, $\tilde{p}_0=\varphi\circ p_0\circ\psi^{-1}$
and $\tilde{\iota}_u=\iota_{u}\circ\varphi^{-1}$. 
Then, from \eqref{eq:diag-U-u-teo} it follows that
\[
\tilde{\iota}_u\circ\tilde{p}_0=\iota_u\circ p_0\circ\psi^{-1}
=p\circ\iota_{U}\circ\psi^{-1}=\tilde{p}\circ \tilde{\iota}.
\]
Let $\phi:\overline{\cO_{q}(B^{+}_{M})}\to \mathfrak{u}_q^{M}(\mathfrak{b}^{+})$ 
the Hopf algebra isomorphism given in Proposition \ref{prop:sucOB} (c). 
From \eqref{eq:diag-U-u-teo} follows that 
\[\pi_u\circ \tilde{p}=\pi_u\circ p\circ\psi^{-1}=\pi_U\circ \psi^{-1}.\]
Since $\pi_U\circ \psi^{-1}=\phi\circ \pi_{B^{+}}$ the result is proved.
\end{proof}

\begin{obs}\label{rmk:liftings-s.e.s}
	By the previous theorem, one may describe the Hopf algebras 
	$u(\cD_{M},\mu)$ by the data coming from the short exact 
	sequence 
	$\xymatrix{ \C^{\widehat{\subg}}\ar@{^{(}->}[r]^>>>>>{\tilde{i}_u} &
	 u(\cD_{M},\mu)  \ar@{->>}[r]^{{\pi}_u} & 
	 \mathfrak{u}_q^{M}(\mathfrak{b}^{+}).}$ 
	Indeed, relations \eqref{eq:delulmu1} and
	\eqref{eq:delulmu2} are satisfied by both the generators $g_m,\,x_i$ of
	$u(\cD_{M},\mu)$ and the generators $k_{m},\,e_i$ of 
	$\mathfrak{u}_q^{M}(\mathfrak{b}^{+})$. Then, the relation
	\eqref{eq:delulmu3} actually encodes the information
	of the previous short exact sequence. 
	\end{obs}

\begin{cor}\label{cor:all-liftings-qsubgroups}
	Let $\Gamma = \big(\Z/\ell_{1}N\Z\big)\times\cdots \times  \big(\Z/\ell_{\theta}N\Z\big)$ be a finite abelian group
	containing $M$
	with generators $g_{m_{1}},\cdots, g_{m_{\theta}}$ and 
	$\ord g_{m_{i}} = \ell_{i}N$,
	$\mu= (\mu_{\alpha})_{\alpha\in Q^{+}}$ a family of root vector 
	parameters, $\cD_{M}$ 
	the finite Cartan datum for $\Gamma$ as defined in 
	$\S$\ref{subsec:data-Cartan-type}  
	and $u(\cD_{M},\mu)$ the corresponding pointed Hopf algebra. 
	Then there are a finite abelian group $\subg$ and a monomorphism 
	of algebraic groups  $k: \widehat{\subg} \hookrightarrow B^{+}_{M}$  
	such that, as Hopf algebras,
	$$
	u(\cD_{M},\mu) \simeq \cO_{q}(B^{+}_{M})/(\Ker k^{*}).
	$$
\end{cor}

\begin{proof}
	By the previous theorem, we know that the pointed
	Hopf algebra $u(D_{M},\mu)$ fits into the commutative diagramm 
	\eqref{diag:theorem-liftings}  and we have a projection 
	$\tilde{p}_0:\cO(B^{+}_{M}) \twoheadrightarrow \C^{\widehat{\subg}}$, 
	where $T_{M}$ is a finite abelian group. 
	Consider the corresponding monomorphism of algebraic groups   
	$k: \widehat{\subg}\hookrightarrow{B^{+}_{M}}$ 
	and the Hopf algebra quotient $\cO_{q}(B^{+}_{M})/(\Ker \tilde{p}_0)$.
	Since the latter is given by a pushout,
	there exists a unique 
	Hopf algebra map $\rho: \cO_{q}(B^{+}_{M})/(\Ker \tilde{p}_0) 
	\to u(D_{M},\mu)$
	such that the following diagram commutes:
	$$
	\xymatrix{\cO(B^{+}_{M})
		\ar@{^{(}->}[r]^{\tilde{\iota}} \ar@{->>}[d]_{\tilde{p}_0} &
		\cO_{q}(B^{+}_{M}) \ar@{->>}[d]^{p_{J}} \ar@/^/[ddr]^{p_{\mu}}\\
		\C^{\widehat{\subg}}\ar@{^{(}->}[r] \ar@/_/[drr]^{\tilde{\iota}_u}& \cO_{q}(B^{+}_{M})/(\Ker \tilde{p}_0)
		\ar@{-->}[dr]|-{\exists ! \rho}\\
		& & u(\cD_{M},\mu)}
	$$
	Since the maps $p_{J}$ and $p_{\mu}$ are surjective, $\rho$ is also surjective.
	Moreover, as 
	\begin{align*}
		\dim u(\cD_{M},\mu) &= 
		|\subg| \dim u_{q}^{M}(\lieb^{+}) 
		= \dim \cO_{q}(B^{+}_{M})/(\Ker \tilde{p}_0),
	\end{align*}
	it follows that $\rho$ is indeed an isomorphism.
\end{proof}

Now we prove our second 
main result which is a converse of the previous corollary.
That is, any inclusion of a finite abelian group in $B^{+}_{M}$ 
gives a lifting of $\toba(V)\# \C\Gamma$, for some abelian group $\Gamma$ 
and some Yetter-Drinfeld module $V$ over $\C\Gamma$.
Keep the notation of \S \ref{subsec:Oq}.

\begin{teor}\label{teo:qsubgroups-are-liftings}
	Let $\subg$
	be a finite abelian group 
	and $k: \widehat{\subg}\hookrightarrow B^{+}_{M}$ 
	be a monomorphism of algebraic groups.
	Then there exist a finite abelian group $\Gamma$  and 
	a diagonal braided vector space  $V$ of Drindeld-Jimbo type 
	such that the Hopf algebra 
	$A=\cO_{q}(B^{+}_{M})/\big(\tilde{\iota}(\Ker k^{*})\big)$ 
	is a lifting of $\toba(V)\#\C \Gamma $. 
	Moreover $\Gamma$ contains $M$, it is isomorphic to 
	$\big(\Z/\ell_{1}N\Z\big)\times\cdots \times  
	\big(\Z/\ell_{\theta}N\Z\big)$
	for some $\ell_{1},\ldots, \ell_{\theta} \in \N$, and the Cartan matrix
	associated with $V$ corresponds to that of $B^{+}_{M}$.
\end{teor}

\begin{proof}
	The monomorphism $k: \widehat{\subg}\hookrightarrow B^{+}_{M}$
	corresponds to 
	an epimorphism of Hopf algebras 
	$k^{*}: \cO(B^{+}_{M})\twoheadrightarrow \C^{\widehat{\subg}} \simeq \C \subg $
	with kernel $J = \Ker k^{*}$. 
	Using \cite[Prop. 2.10]{AG}, one sees that 
	$(J) = \cO_{q}(B^{+}_{M})J$ is a 
	Hopf ideal of $\cO_{q}(B^{+}_{M})$ and the Hopf algebra 
	$\cO_{q}(B^{+}_{M})/(J)$ is
	a pushout that fits into the following commutative 
	diagram with exact rows
	\[\xymatrix{
		\cO(B^{+}_{M}) \ar@{^{(}->}[r]^{\iota} \ar@{->>}[d]_{k^{*}} &  
		\cO_q(B^{+}_{M}) \ar@{->>}[r]^{\pi} \ar@{->>}[d]^{p_{J}} & \overline{\cO_{q}(B^{+}_{M})} 
		\ar@{=}[d]\\
		\C^{\widehat{\subg}} \ar@{^{(}->}[r] & \cO_{q}(B^{+}_{M})/(J)  \ar@{->>}[r] & 
		\overline{\cO_{q}(B^{+}_{M})},} 
	\]
here $p_{J}$ is the canonical projection. 
In particular, $\cO_{q}(B^{+}_{M})/(J)$ is finite-dimensional. 
Also, since $\cO_{q}(B^{+}_{M})\simeq U_{q}^{M}(\lieb^{+})$ 
and the latter is a 
pointed Hopf algebra with coradical isomorphic to $\C \Z^{\theta}$,
$H:=\cO_{q}(B^{+}_{M})/(J)$ is a pointed Hopf algebra whose coradical is 
$H_0=\C\Gamma$, where $\Gamma = p_{J}(\Z^{\theta}) $ contains $M$. 
Moreover, since the previous diagram is commutative and the group of
grouplike elements of 
$\overline{\cO_{q}(B^{+}_{M})} \simeq u_{q}^{M}(\lieb^{+})$
is isomorphic to $(\Z/N\Z\big)^{\theta}$, the abelian group $\Gamma$
fits into the short exact sequence of groups
\[
1\to \subg \hookrightarrow \Gamma \twoheadrightarrow (\Z/N\Z\big)^{\theta}
\to 1.
\]
Denote by $g_{m_{i}}$ the image by $p_J$ of the generators 
$K_{m_{i}} \in U_{q}^{M}(\lieb^{+})$
for all $i\in I$.
Then $\Gamma$ is generated by 
$g_{m_{1}},\ldots, g_{m_{\theta}}$ and $\subg$ is generated by
$g_{m_{1}}^{N},\ldots, g_{m_{\theta}}^{N}$. 
If $\ell_{i} = \ord g_{m_{i}}^{N} $, 
then $\Gamma \simeq \big(\Z/\ell_{1}N\Z\big)\times\cdots \times  \big(\Z/\ell_{\theta}N\Z\big)$
and $\subg\simeq  \big(\Z/\ell_{1}\Z\big)\times\cdots \times  \big(\Z/\ell_{\theta}\Z\big)$.

\smallbreak
Besides, as $H$ is pointed, 
the associated graded Hopf algebra with 
respect to the coradical filtration $\gr H$ satisfies that 
$\gr H \simeq R \# \C\Gamma$. Moreover, since $\Gamma$ is a finite abelian
group, $R$ is a Nichols algebra $\toba(W)$ 
of diagonal type with $W=R^{1}=P_{1}(R)$,
see \cite{An}. 

\smallbreak

On the other hand, $\overline{\cO_{q}(B^{+})_{M}} \simeq  
u_{q}^{M}(\lieb^{+}) $
and the latter is also a pointed Hopf algebra with 
$\gr u_{q}^{M}(\lieb^{+}) \simeq \toba(V) \# \C \big(\Z/N\Z\big)^{\theta}$.  Here
$V$ is the diagonal braided vector space of Cartan type with associated Cartan 
matrix $\mcar$ as in $\S$\,\ref{sec-datum}. Since 
$\dim H = |\Gamma| \dim \toba(W) =  |\subg| \dim u_{q}^{M}(\lieb^{+}) = 
|\subg|\frac{|\Gamma|}{|\subg|}\dim \toba(V)$,
it follows that $\dim \toba(W)=\dim \toba(V)$. 	
Denote by $\bar{\pi}: H\twoheadrightarrow u_{q}^{M}(\lieb^{+})$ the Hopf algebra epimorphism,
and by $(H_{n})_{n\geq 0}$ and 
$(u_{n})_{n\geq 0}$ the terms of the coradical filtration
of $H$ and $u_{q}^{M}(\lieb^{+})$, respectively.
Using that $H$ and $u_{q}^{M}(\lieb^{+})$
are pointed, by \cite[Cor. 4.2.2]{Ra} we have that 
$\bar{\pi}(H_{n})\subseteq u_{n}$ for all $n\geq 0$. In particular, $\bar{\pi}$ induces a Hopf algebra map
$\hat{\pi}: \gr H \to \gr u_{q}^{M}(\lieb^{+})$
which is surjective because $\bar{\pi}$ is so; see \cite[Cor. 4.2.2]{Ra}. 
This implies that there is a 
surjective braided Hopf algebra map $\toba(W) \to \toba(V)$, and consequently 
$\toba(W)\simeq \toba(V)$. 
Hence, $\gr H \simeq \toba(V) \# \C\Gamma$ and 
$H=\cO_{q}(B^{+}_{M})/(J)$ is a lifting of $\toba(V)$
\end{proof}


As a direct consequence of the theorems above, it follows that all Hopf
algebras $u(\cD_{M}, \mu)$ can be described as quantum subgroups 
$\cO_{q}(B^{+}_{M})/(\Ker k^{*})$ depending on the 
monomorphism $k: \widehat{\subg} \to B^{+}$ defined using the 
family of root vector parameters $\mu=(\mu_{\alpha})_{\alpha\in Q^{+}}$.

\begin{obs}\label{rmk:computing-liftings}
	Following the commutative diagrams \eqref{eq:diag-U-u-teo} and \eqref{diag:theorem-liftings}, 
	we find a way to compute 
	the lifting of the power root vector relation \eqref{eq-HA-4}. 
	This is based in the isomorphism 
	$\psi: U_{q}^{M}(\lieb^{+})\to \cO_{q}(B^{+}_{M})$ (see the proof of 
	Theorem \eqref{teo:quantum-subgrup-DJ}),
	the monomorphism $k: \widehat{\subg} \hookrightarrow B^{+}_{M}$ and the 
	well-known formula that gives the expression of any positive root 
	$\alpha$ in terms of the simple roots $\alpha_{i}$, with $i\in I$,
	for each type of Lie algebra. Precisely, by Proposition \ref{prop:U-u}, 
	for all $\alpha \in \Phi^{+}$, we have that
\begin{align*}
	x_{\alpha}^{N}&=p(E_{\alpha}^{N}) = p_{\mu}(\psi(E_{\alpha}^{N}))
	=\varphi^{-1} k^{*}(\psi(E_{\alpha}^{N})),
\end{align*}
	because $E_{\alpha}^{N} \in Z^{\geq}_{M}$. Here, 
	$\varphi^{-1}k^{*}(\psi(E_{\alpha}^{N})) \in \C T$ and 
	$\varphi^{-1}:\C^{\widehat{\subg}}\to \C T$ is the isomorphism considered in 
	Theorem \ref{teo:quantum-subgrup-DJ}. 
\end{obs}



\subsection{Braidings of general Cartan type}\label{subsec:BraidingGenCartan}

In this subsection we assume that 
the braiding 
$\mathbf{q}=(q_{ij})_{i,j\in I}$ is of Cartan type for certain Cartan matrix
$C=(a_{ij})_{i,j\in I}$. 
Recall that we write $a= |\det(\mcar)| =[P:Q]$. For 
any lattice $Q\subseteq M\subseteq P$ of rank $\theta$
with basis $\{m_{1},\ldots, m_{\theta}\}$ we have that 
$  \alpha_i = \sum_{j \in I} a_{ji}^{M} \, m_j $
for all  $ i,j \in I$, where 
$C_{M}= (a_{ji}^{M})_{i,j} \in \Z^{\theta\times \theta} $
and $a_{M}:=|\det C_{M}|=[M:Q]$.

For each $(i,j)\in I\times I$, we choose 
$q_{ij}^{1/a^{2}}\in\C$ an $a^2$-th root of $q_{ij}$, that is, 
$\big(q_{ij}^{1/a^{2}}\big)^{a^2}=q_{ij}$. 
Given $m\in (1/a^2)\Z$, we set $q_{ij}^m:=\big(q_{ij}^{1/a^2}\big)^{a^2m}$;
in particular, we have that $q_{ij}^{1/a}=\big(q_{ij}^{1/a^{2}}\big)^{a}$
and $q_{ij}^{1/a_{M}}=\big(q_{ij}^{1/a}\big)^{[P:M]}$.
We assume further that
\[
q_{ij}^{1/a}q_{ji}^{1/a} = q_{ii}^{a_{ij}/a},  \qquad\text{ for all } i,j\in I.
\]

\bigbreak
\subsubsection{The multiparameter for any lattice $M$}
Write $C_{M}^{-1}=(\tilde{a}_{ij})_{i,j\in I}$ for the inverse matrix of $C_{M}$.
Then $a_{M}C_{M}^{-1}\in \Z^{\theta\times \theta}$ and we define:
\begin{equation}\label{eq:cond-Cartan-weights}
	q_{ij}^{M} =  \prod_{k\in I} q_{kj}^{\tilde{a}_{ki}}.
\end{equation}
Observe that $q_{kj}^{\tilde{a}_{ki}}$ is well-defined because
$\tilde{a}_{ki}\in (1/a_{M})\Z \subseteq (1/a^2)\Z$.
Recall also from \S \ref{subsec:Lie-algebras} that we may write 
$  m_j = \sum_{k \in I} \bar{a}_{kj}^{M} \, w_k $
for all  $ j \in I$; in particular 

\begin{align}
 \bar{a}_{kj}^{M} = \frac{(\alpha_{k},m_{j})}{d_{k}} \qquad
 \text{ for all } k,j \in I.	
\end{align}

The matrices
$C_{M}= (a_{ji}^{M})_{i,j\in I}$ and $ \bar{C}_{M} = (\bar{a}_{kj}^{M})_{k,j\in I}\in \Z^{\theta\times \theta} $
satisfy that $C = \bar{C}_{M}C_{M}$.

\begin{obs}\label{obs-canonical-multi}
	In case $q_{ij} = \check{q}_{ij} = q^{d_{i}a_{ij}}$ 
	is the \emph{canonical} multiparameter, one has that
	$q_{ij}^{M} = q^{(\alpha_{j},m_{i})}$. 
	Indeed, since $q_{kj}= q_{jk}$ for all $j,k\in I$, we have that
	\begin{align*}   
	q_{ij}^{M} & = 
	\prod_{k\in I} q_{kj}^{\tilde{a}_{ki}} =
	\prod_{k\in I} q_{jk}^{\tilde{a}_{ki}} = 
	\prod_{k\in I} q^{d_{j}a_{jk}\tilde{a}_{ki}} = 
	q^{d_{j}\sum_{k\in I}a_{jk}a'_{ki} }=
	q^{d_{j}(C\, C_{M}^{-1})_{ji}} = q^{d_{j}(\bar{C}_{M})_{ji}}\\
	& = q^{d_{j}\frac{(\alpha_{j},m_{i})}{d_{j}}} =
	q^{(\alpha_{j},m_{i})}.	 
	\end{align*}
For example, for $M=P$ we have $q_{ij}^{P} = 
q^{\omega_{i}(\alpha_{j})}=q^{d_{i}\delta_{ij}}$ while for $M=Q$ 
we have $q_{ij}^{Q} = 
q^{(\alpha_{j},\alpha_{i})}=q^{d_{j}a_{ji}}=q^{d_{i}a_{ij}}$.
\end{obs}

\smallbreak
\subsubsection{The Cartan datum for any lattice $M$}

\smallbreak	
Let $\Gamma = \big(\Z/\ell_{1}N\Z\big)\times\cdots \times
\big(\Z/\ell_{\theta}N\Z\big)$ be a finite abelian group
with generators $g_{m_{i}}\in \Z/\ell_{i}N\Z$, for all $i\in I$. 
Consider the group epimorphism $g:M\to \Gamma$ defined by 
$g(m_i)=g_{m_i}$ for all $i\in I$. 
For $j\in I$, define $g_{j}\in\Gamma$ and $\chi_{j}\in \widehat\Gamma$ by 
\begin{align}\label{eq-datum}
\chi_{j}(g_{m_{i}}) = q_{ij}^{M}, & &g_{j} = g(\alpha_{j}) = g_{\alpha_{j}}.	
\end{align}
Note that $\Gamma$ is a group containing $M$ 
(see Definition \ref{def-contain-group}) and
$\chi_{j}(g_{i}) = q_{ij}$ for all $i,j \in I$. Indeed,
\begin{align*}
	\chi_{j}(g_{i}) & = \chi_{j}\big(g(\alpha_{i})\big) = \chi_{j}\big(g\big(\sum_{k\in I} a^{M}_{ki}m_{k}\big)\big)
	= \chi_{j}\big(\prod_{k\in I} g_{m_{k}}^{a^{M}_{ki}}\big) 
	=  \prod_{k\in I} \chi_{j}(g_{m_{k}})^{a^{M}_{ki}}\\
	&=\prod_{k\in I} (q_{kj}^{M})^{a^{M}_{ki}}=
	\prod_{k,\ell \in I} q_{\ell j}^{\tilde{a}_{\ell k}a^{M}_{ki}}
	 = \prod_{\ell \in I} q_{\ell j}^{\sum_{k\in I}
	 \tilde{a}_{\ell k}a^{M}_{ki}} = \prod_{\ell \in I} q_{\ell j}^{\delta_{\ell\, i}} = q_{ij}.
\end{align*}
\bigbreak
Consider now the tuple
\begin{align}\label{eq-data-cartan-P}
	\cD_{M}=\cD(\Gamma,(g_i)_{i\in I},(\chi_i)_{i\in I}, (a_{ij})_{i,j\in I}),
\end{align}
where $g_i$ and $\chi_i$ are given by \eqref{eq-datum}. 
By the previous calculation, $\cD_{M}$ is a finite Cartan datum 
for $\Gamma$.
For the sequel  in this subsection, we define the elements
$\chi_{k}^{\tilde{a}_{ki}}\in \widehat{\Gamma}$  by: 
$$
\chi_{k}^{\tilde{a}_{ki}}(g_{\omega_{j}})=
(q_{jk}^M)^{\tilde{a}_{ki}},\qquad \text{ for all }i,j,k\in I.
$$

\begin{lema}\label{lema:equationsqijP} 
Let $\chi_{m_{i}}: = \prod_{k}\chi_{k}^{\tilde{a}_{ki}}\in 
\widehat{\Gamma}$. 
The following formula holds:	
\begin{eqnarray}\label{eq:Npowerchar}
\chi_{j}(g_{m_{i}})\chi_{m_{i}}(g_{j}) & = 
q_{jj}^{\frac{(\alpha_{j},m_{i})}{d_{j}}} \qquad i,j\in I.
\end{eqnarray}
\end{lema}
	
\begin{proof}
Let $i,j\in I$. Then
\begin{align*}
\chi_{j}(g_{m_{i}})\chi_{m_{i}}(g_{j}) &= 
q_{ij}^{M}\prod_{k}\chi_{k}^{\tilde{a}_{ki}}(g_{j})=
\left(\prod_{k\in I} q_{kj}^{\tilde{a}_{ki}}\right)
\left(\prod_{k\in I}q_{jk}^{ \tilde{a}_{ki}}\right)
=  \prod_{k\in I} (q_{kj}q_{jk})^{\tilde{a}_{ki}} \\
&=  \prod_{k\in I} (q_{jj}^{a_{jk}})^{\tilde{a}_{ki}} = 
q_{jj}^{\sum_{k\in I}a_{jk} \tilde{a}_{ki}} =   q_{jj}^{(\bar{C}_{M})_{ji}}
= q_{jj}^{\bar{a}^{M}_{ji}}
 = q_{jj}^{\frac{(\alpha_{j},m_{i})}{d_{j}}}.
		\end{align*}
\end{proof}

\subsubsection{The algebra  $u(\cD_{M},\mu)$ as a central extension}
Consider the pointed Hopf algebra $u(\cD_{M},\mu)$ associated with the finite
Cartan datum $\cD_{M}$ given in \eqref{eq-data-cartan-P}. 
Let $\subg$ be the subgroup of $\Gamma$ given by: 
\begin{align}\label{sub-t}
\subg = \langle g_{m_i}^{N} \in \Gamma\, |\, 
\chi_{m_i}^{N} =\varepsilon \rangle.
\end{align}
Note that in the Drinfeld-Jimbo case, the previous subgroup of $\Gamma$ actually coincides with the 
subgroup $\subg$ of $\Gamma$ generated by $g_{m_i}^N$ which was considered in Proposition \ref{prop:ces-u}. 
\smallbreak

Let  $\overline{\Gamma}:= \Gamma / \subg$ and denote by 
$\bar{g}_{m_{i}}$ the class of $g_{m_i}$. 
Let $V$ the braided vector space with basis $\{x_i\}_{i\in I}$ 
and braiding $c(x_i\otimes x_j)=q_{ij}x_j\otimes x_i$ , and  
consider the pointed Hopf algebra  
\begin{align}\label{ha-u+}
	\lieu^{M}(V)^{\geq 0}:= \toba(V)\#\C\overline{\Gamma}
\end{align}
obtained by the bosonization of 
the Nichols algebra $\toba(V)$ with the group algebra 
$\C\overline{\Gamma}$. 

\begin{prop}\label{prop:ces-u-Cartan}
Let $\mu = (\mu_{\alpha})_{\alpha \in Q^{+}}$ be a family of 
root vector parameters. The group algebra $\C \subg$ is central in 
$u(\cD_{M},\mu)$ 
and there exists an exact sequence of Hopf algebras 
\begin{equation}\label{eq:ces-u-Cartan}
\xymatrix{\C \subg  \ar@{^{(}->}[r]^(0.4){\iota_{u}}& 
u(\cD_{M},\mu)\ar@{->>}[r]^(0.45){\pi_{u}}& \lieu^{M}(V)^{\geq 0}}.
\end{equation}

\end{prop}

\pf Let $g_{m}^{N} \in T_{M}$ 
with $m = \sum_{i=1}^{\theta} n_{i} m_{i} \in M^{+}$. 
Since the braiding is of Cartan type, we have
\begin{align*}
g_{m}^{N}\, x_{j} \, g_{m}^{-N}&= 
\chi_{j}(g_{m})^{N} x_{j}=\chi_{j}\big(g(m)\big)^{N} x_{j} = 
\chi_{j}\big(g(\sum_{i=1}^{\theta} n_{i} m_{i})\big)^{N} x_{j} =
\chi_{j}\big(\prod_{i=1}^{\theta}g(m_{i})^{n_{i}}\big)^{N} x_{j}  \\
& = \prod_{i=1}^{\theta}\chi_{j}(g_{m_{i}})^{N\, n_{i}}\, x_{j} 
\overset{(*)}{=} \prod_{i=1}^{\theta}\chi_{m_{i}}(g_{j})^{-N\, n_{i}}\, 
\big(q_{jj}^{\frac{(\alpha_{j},m_{i})}{d_{j}}}\big)^{N\, n_{i}}\, x_{j}
 = \big(\prod_{i=1}^{\theta}\chi^{n_{i}}_{m_{i}}(g_{j})\big)^{-N} x_{j}\\
		&= \big(\chi_{m}^{N}(g_{j})\big)^{-1}x_{j}=x_{j},
	\end{align*}
for all $j\in I$; where the equality $(*)$ follows from \eqref{eq:Npowerchar}.
This implies that $\subg$ is central in $u(\cD_{M},\mu) $ and consequently 
one has a central exact sequence of Hopf algebras
$$
\C \subg\hookrightarrow u(\cD_{M},\mu) \twoheadrightarrow 
u(\cD_{M},\mu) /\C \subg^{+} u(\cD_{M},\mu). 
$$
To finish the proof we show that 
$ \lieu^{M}(V)^{\geq 0}$ is isomorphic to the Hopf algebra quotient 
$u(\cD_{M},\mu) /\C \subg^{+} u(\cD_{M},\mu)$. 
Let $p_{\subg}: u(\cD_{M},\mu) \twoheadrightarrow 
u(\cD_{M},\mu) /\C \subg^{+} u(\cD_{M},\mu)$ 
the canonical epimorphism.
Since $x_{\alpha}^{N} = r_{\alpha}(\mu) \in \C\subg^{+}$, it follows that 
$p_{\subg}(x_{\alpha}^{N}) = 0$ for all $\alpha \in Q^{+} $.
Then there exists an epimorphism of Hopf algebras 
$\lieu^{M}(V)^{\geq 0} \to u(\cD_{M},\mu) /\C \subg^{+} u(\cD_{M},\mu)$, 
which is an isomorphism because both algebras have the same dimension 
$\frac{|\Gamma|}{|\subg|} \dim \toba(V) $.
	\epf

In order to obtain a similar result as 
Theorem \ref{teo:quantum-subgrup-DJ} in the general case, we
 need to consider multiparametric versions of the quantized 
enveloping algebra associated with the Cartan matrix $\mcar$ and having a \emph{torus} containing the lattice $M$, 
and their corresponding Borel subalgebras.

\begin{defi}  \cite[Def. 7]{HPR}. \label{def:multiqgroup_ang}
The {\it multiparameter quantized enveloping algebra
$ U_{\bf q}^{M}(\lieg) $} is  the unital associative  $\C $--algebra 
generated by elements  
$ \, E_i \, , \, F_i \, , \, K_{m_{i}}^{\pm 1} \, , \, L^{\pm 1}_{m_{i}} \, $, 
$ \, i \in I \, $, satisfying the following relations: for all $i,j\in I$,

\begin{minipage}[t]{0.5\textwidth}
\begin{align*}
K_{m_{i}}^{\pm 1} L^{\pm 1}_{m_{j}}  & = 
L^{\pm 1}_{m_{j}} K^{\pm 1}_{m_{i}}\\[.3em]
K_{m_{i}}^{\pm 1} K^{\pm 1}_{m_{j}}  &=  
K^{\pm 1}_{m_{j}} K^{\pm 1}_{m_{i}} \\[.3em]
K_{m_{i}} \, E_j \, K_{m_{i}}^{-1}  &= q_{ij}^{M} \, E_j\\[.3em]
K_{m_{i}} \, F_j \, K_{m_{i}}^{-1}  &=   (q_{ij}^{M})^{-1} \, F_j
\end{align*}
\end{minipage}
\begin{minipage}[t]{0.4\textwidth}
\begin{align*}
K_{m_{i}}^{\pm 1} K^{\mp 1}_{m_{i}}  & = 1   =  
 L^{\pm 1}_{m_{i}} L^{\mp 1}_{m_{i}},  \\[.3em]
L_{m_{i}}^{\pm 1} L_{m_{j}}^{\pm 1}  & = \,  
L_{m_{j}}^{\pm 1} L_{m_{i}}^{\pm 1},  \\[.3em]
L_{m_{i}} \, E_j \, L_{m_{i}}^{-1}  &= \,  
(q_{ji}^{M})^{-1} \, E_j , \\[.3em]
L_{m_{i}} \, F_j \, L_{m_{i}}^{-1}   &= \,  q_{ji}^{M} \, F_j, 
\end{align*}
\end{minipage}
\begin{align*}
[E_i , F_j]  &= \,  \delta_{i,j} \, q_{ii} \, \frac{\, K_{\alpha_{i}} - L_{\alpha_{i}} \,}{\, q_{ii} - 1 \,}, \\
\sum_{k=0}^{1-a_{ij}} (-1)^k \,
{\bigg( {1-a_{ij} \atop k} \bigg)}_{\!\!q_{ii}} &q_{ii}^{{k \choose 2}} \, q_{ij}^k \,
E_i ^{\,1-a_{ij}-k} E_j \, E_i^{\,k}  \; = \;  0,   \;\; \qquad \, i \neq j, \,  \\
\sum_{k=0}^{1-a_{ij}} (-1)^k \,
{\bigg( {1-a_{ij} \atop k} \bigg)}_{\!\!q_{ii}}& q_{ii}^{{k \choose 2}} \, q_{ij}^k \,
F_i^{\, k} F_j \, F_i^{\,1-a_{ij}-k}  \; = \;  0 ,  \;\; \qquad \, i \neq j. \,  \\
 \end{align*}
	For $m = \sum_{i \in I} n_i \, m_i \, \in \, M $
we set $K_m := \prod_{i \in I} K_{m_{i}}^{n_i} $ and 
$L_m := \prod_{i \in I} L_{m_{i}}^{n_i} $.
In particular, $K_{\alpha_{i}} = \prod_{j}K_{m_{j}}^{a^{M}_{ji}}$ 
and $L_{\alpha_{i}} = \prod_{j}L_{m_{j}}^{a^{M}_{ji}}$
for any simple root $\alpha_{i}$, with $i\in I$. It turns out that $ U_{\bf q}^{M}(\lieg) $  is also 
a Hopf algebra whose structure is 
determined by: for all  $ \, i, j \in I$, 
\begin{align*}
\com(E_i) \,  &  = \,  E_i \otimes 1 + K_{\alpha_{i}} \otimes E_i  ,  &
\varepsilon(E_i) \,  &  = \,  0  ,  &  \cS(E_i) \,  &  = \,  -K_{\alpha_{i}}^{-1} E_i,  \\
\com(F_i) \,  &  = \,  F_i \otimes L_{\alpha_{i}} + 1 \otimes F_i  ,  &
\varepsilon(F_i) \,  &  = \,  0   ,  &  \cS(F_i) \,  &  = \,  - F_i{L_{\alpha_{i}}}^{-1},  \\
\com\big(K_{m_{i}}^{\pm 1}\big)  &  
= \,  K_{m_{i}}^{\pm 1} \otimes K_{m_{i}}^{\pm 1}   ,  &
\varepsilon\big(K_{m_{i}}^{\pm 1}\big)  &  = \,  1   ,  & 
 \cS\big(K_{m_{i}}^{\pm 1}\big)  &  = \,  K_{m_{i}}^{\mp 1},  \\
\com\big(L_{m_{i}}^{\pm 1}\big)  &  = \,  
L_{m_{i}}^{\pm 1} \otimes L_{m_{i}}^{\pm 1}  ,  &
\varepsilon\big(L_{m_{i}}^{\pm 1}\big)  &  = \,  1 ,  &  
\cS\big(L_{m_{i}}^{\pm 1}\big)  &  = \,  L_{m_{i}}^{\mp 1}.
\end{align*}
The {\it  positive multiparameter quantum Borel subalgebra 
$U_{\bf q}^{M}(\lieb^{+})$ } is the  $ \C $-subalgebra  of   
$ U_{\bf q}^{M}(\lieg) $  generated 
by  $K_{{m_{i}}}^{\pm 1}$ and $E_{i}$, for all $i\in I$.
Clearly, it is a Hopf subalgebra of $ U_{\bf q}^{M} (\lieg)$.
\end{defi}

\begin{obs}  \label{link_QEq-QE & symm-case}
The one-parameter quantum group  $ U_q^{M}(\lieg) $ can be obtained
as a quotient of a multiparameter quantized enveloping algebra  
$U_{\check{\bf q}}^{M} (\lieg)$ by taking
the canonical multiparameter
$\check{\bf q} := \big(\check{q}_{ij}\big)_{i, j \in I}$ with 
$\check{q}_{ij} = q^{d_i a_{ij}}$,
and taking the quotient by the two-sided
ideal generated by the elements $L_{m_{i}} - K_{m_{i}}^{-1}$,
for all $i\in I$.   
In particular, $U_{\check{\bf q}}^{M} (\lieb^{+}) = U_q^{M} (\lieb^{+})$.                                             
\end{obs}

\begin{teor} \label{thm:sigma_2-cocy}
There exists a normalized  Hopf $ 2 $-cocycle  
$ \sigma $ of $U_{\check{\bf q}} ^{M}(\lieg)$ such that
\[ 
U_{\bf q}^{M}(\lieg)  \simeq \big(\, U_{\check{\bf q}} ^{M}(\lieg) \big)_\sigma\,\text{ and }\, U_{\bf q}^{M}(\lieb^{+})  \simeq \big(\, U_{q} ^{M}(\lieb^{+}) \big)_\sigma   
\]
as Hopf algebras.
\end{teor}
\pf By  \cite[Thm. 28]{HPR}, there exists a normalized Hopf $2$-cocycle  
$\sigma $ of $U_{\check{\bf q}}^{Q}(\lieg)$ such that
$U_{\bf q}^{Q}(\lieg)  \simeq \big(\, U_{\check{\bf q}} ^{Q}(\lieg) \big)_\sigma$. The proof for $U_{\bf q}^{M}(\lieg)$ (over the 
lattice $M$) follows \emph{mutatis mutandis} from the discussion in 
\cite[\S 4.2.2]{GG1}. The isomorphism 
$U_{\bf q}^{M}(\lieb^{+})  \simeq \big(\, U_{q} ^{M}(\lieb^{+}) \big)_\sigma $
is obtained by restriction. \epf

\subsubsection{The algebra  $u(\cD_{M},\mu)$ as a quantum subgroup}From now on we assume that 
\begin{equation}\label{assumptionChi}
\chi_{m}^{N}=\eps\qquad \text{ for all } \,
m\in M.
\end{equation}

\noindent
Under this assumption, the subgroup $\subg$ 
coincides with the subgroup of $\Gamma$ generated by the elements 
$g_{m_{i}}^{N}$, with $i \in I$.

\begin{prop}\label{prop:Uq-uD} The following assertions hold:

\begin{enumerate}
\item [\rm (a)] $U_{\bf q}^{M}(\lieb^{+}) \simeq 
\widetilde{\toba}(V)\# \C \Z^{\theta}$ as Hopf algebras, 
where $\tilde{\toba}(V)$ is the distinguished pre-Nichols algebra
defined in \cite[Section 3]{An2}. \vspace{.1cm}

\item [\rm (b)] The map $p: U_{\bf q}^{M}(\lieb^{+}) \twoheadrightarrow u(\cD_{M},\mu)$ defined
by $p(K_{m_{i}}) = g_{m_{i}}$ and 
$p(E_{i}) = x_{i}$ for all $i\in I$, is an epimorphism of Hopf algebras.  \vspace{.1cm}

\item [\rm (c)] The subalgebra $Z^{\geq}_{M}$ of 
$U_{\bf q}^{M}(\lieb^{+})$ generated
by the elements $K_{m}^{N}$, $E_{\alpha}^{N}$ 
with $m \in M$ and $\alpha\in Q^{+}$
is a central Hopf subalgebra
and $U_{\bf q}^{M}(\lieb^{+})$ is a free $Z^{\geq}_{M}$-module 
of finite rank.\vspace{.1cm}

\item [\rm (d)] The quotient Hopf algebra 
$\overline{U_{\bf q}^{M}(\lieb^{+})} = 	
U_{\bf q}^{M}(\lieb^{+})/\big(Z^{\geq}_{M}\big)^{+}U_{\bf q}^{M}(\lieb^{+})$
is isomorphic to the Hopf algebra
$\lieu^{M}(V)^{\geq 0}$ defined in \eqref{ha-u+} and 	one has the central 
exact sequence
$$
Z^{\geq}_{M}\hookrightarrow U_{\bf q}^{M}(\lieb^{+}) 
\twoheadrightarrow 
\lieu^{M}(V)^{\geq 0}.
$$
\item [\rm (e)] 
Let $p_{0}$ be the restriction 
of $p$ to the subalgebra $Z^{\geq}_{M}$. 
Then the image of $p_{0}$ is contained in $\C \subg$ 
and the following is a commutative diagram with exact rows
\begin{align}\label{eq:diag-U-u-teo-Cartan}
\begin{aligned}
\xymatrix{
Z^{\geq}_{M} \ar@{^{(}->}[r]^(0.4){\iota_{U}} \ar@{->>}[d]_{p_{0}} & 
U_{\bf q}^{M}(\lieb^{+})\ar@{->>}[r]^{\pi_{U}}\ar@{->>}[d]^{p}
& \lieu^{M}(V)^{\geq 0} \ar@{=}[d]^{}\\
\C \subg  \ar@{^{(}->}[r]^{\iota_{u}}& u(\cD_{M},\mu)\ar@{->>}[r]^{\pi_{u}}& 
\lieu^{M}(V)^{\geq 0}.}
\end{aligned}
\end{align}
\end{enumerate}
\end{prop}  
\pf
\noindent (a) Observe that the subalgebra of $U_{\bf q}^{M}(\lieb^{+})$
generated by $\{K_{m_{i}}\}_{i\in I}$  is isomorphic to the group algebra 
$\C \Z^{\theta}$, whereas the subalgebra generated by 
$\{E_{i}\}_{i\in I}$ is isomorphic to $\widetilde{\toba}(V)$.
Note that all vertices are of Cartan type because, by assumption, the braiding is of 
Cartan type. Thus, 
$U_{\bf q}^{M}(\lieb^{+}) \simeq \widetilde{\toba}(V)\# \C \Z^{\theta}$ 
as Hopf algebras.

\smallbreak
\noindent (b) It follows directly by using the presentation by generators and relations for the algebras
$U_{\bf q}^{M}(\lieb^{+})$ and $u(\cD_{M},\mu)$.

\smallbreak
\noindent (c) and (d) follow directly 
from Proposition 21, Theorem 23 and Remark 11 of \cite{An2}. 
In fact, comparing with the 
notation in \emph{loc.cit.} we have by \eqref{assumptionChi}
that $\chi(N_{\alpha}\alpha, \beta) = \chi_{\alpha}^{N}(K_{\beta}) = 1$
for all $\alpha\in Q^{+}$.

\smallbreak
\noindent (e) Note that, by \eqref{assumptionChi} 
and by the definition of $u(\cD_{M}, \mu)$, we have 
$\subg= \langle g^{N}_{m}\, |\, m\in M\rangle$ and 
$p_{0}(Z^{\geq}_{M}) = \C \subg$. 
Since $p_{0}$ is the restriction of $p$, the square on the left clearly 
commutes.   
The square on the right commutes by $(d)$ and the 
very definition of $\lieu^{M}(V)^{\geq 0}$.
\epf

 \begin{obs}\label{rmk:from-seq-U-to-O}
Consider the Hopf algebra isomorphism 
$\psi: U_{q}^{M}(\lieb^{+}) \to \cO_{q}(B^{+}_{M})$ given in
Remark \ref{rmk:zeta+}. Since 
$U^{M}_{\bf q}(\lieb^{+})  \simeq \big(\, U^{M}_{q} (\lieb^{+}) \big)_\sigma$
it follows that $U^{M}_{\bf q}(\lieb^{+})\simeq 
\big(\, \cO_{q}(B^{+}_{M})\big)_{\tilde{\sigma}}$, 
where $\tilde{\sigma}$ is the $2$-cocycle
 induced by $\sigma$ and $\psi$.
 The Hopf algebra $ \big(\, \cO_{q}(B^{+}_{M}) \big)_{\tilde{\sigma}}$ 
 is a multiparametric version $\cO_{\bf q}(B^{+}_{M}) $ 
 of the quantized algebra of function on the positive Borel subgroup 
 $B^{+}_{M}$. 
 Moreover, by \eqref{assumptionChi} 
 and by the definition of the $2$-cocycle $\sigma$,
 it turns out that the  
 central Hopf subalgebra $Z^{\geq}_{M}$ is not deformed under the 
 $2$-cocycle and remains central. In particular, 
 since by the proof of Theorem \ref{teo:quantum-subgrup-DJ}
 we have that $Z^{\geq}_{M} \simeq \cO(B^{+}_{M})$, it follows that
 $\cO(B^{+}_{M})_{\tilde{\sigma}} = \cO(B^{+}_{M}) $, and this 
 is a central Hopf subalgebra of  $\cO_{\bf q}(B^{+}_{M}) $.
 We denote again $\overline{\cO_{\bf q}(B^{+}_{M})}$ the Hopf algebra
 quotient given by 
 $\cO_{\bf q}(B^{+}_{M})/\big(\cO_{\bf q}(B^{+}_{M})^{+}\cO(B^{+}_{M})\big)$.
 \end{obs}

\subsubsection{The algebra  $u(\cD_{M},\mu)$ as a 
multiparameter quantum subgroup}
The next theorem is a direct consequence of Proposition \ref{prop:Uq-uD} 
and Remark \ref{rmk:from-seq-U-to-O} above.
It states that any pointed Hopf algebra $ u(\cD_{M},\mu)$ is a multiparameter
quantum subgroup.

\begin{teor}\label{teo:quantum-subgroup-Cartan}  
Let $\subg$ the subgroup of $\Gamma$ generated by the elements 
$g_{m_{i}}^{N}$, for all $i\in I$. 
There exists a commutative diagram of Hopf algebras with exact rows
\begin{align}\label{diag:theorem-liftings-Cartan}
\begin{aligned}
\xymatrix{
\cO(B^{+}_{M}) \ar@{^{(}->}[r] \ar@{->>}[d]_{\iota_{\mu}^{*}} &   
\cO_{\bf q}(B^{+}_{M}) \ar@{->>}[r] \ar@{->>}[d]^{p_{\mu}} & 
\overline{\cO_{\bf q}(B^{+}_{M})} 
\ar[d]^{\simeq}\\
\C^{\widehat{\subg}}\ar@{^{(}->}[r] & u(\cD_{M},\mu)  \ar@{->>}[r] & 
\lieu^{M}(V)^{\geq 0}} 
\end{aligned} 
\end{align}
\qed
\end{teor}

The proof of the next theorem follows {\it mutatis mutandis} from the proofs
of Theorem \ref{teo:qsubgroups-are-liftings} and Corollary 
\ref{cor:all-liftings-qsubgroups}, as they
rely on general structural properties.

\begin{teor}\label{teo:qsubgroups-are-liftings-Cartan} 
The following assertions hold:
\begin{enumerate}
\item [\rm (a)]	
Let $\subg$
be a finite abelian group 
and $k: \widehat{\subg}\hookrightarrow B^{+}_{M}$ 
be a monomorphism of algebraic groups.
Then there exist a finite abelian group $\Gamma$ and 
a diagonal braided vector space $V$  
of Cartan type such that the Hopf algebra 
$\cO_{\bf q}(B^{+}_{M})/\big(\iota(\Ker k^{*})\big)$ 
is a lifting of $\toba(V)\#\C \Gamma $. 
Moreover $\Gamma$ contains $M$, 
it is isomorphic to 
$\big(\Z/\ell_{1}N\Z\big)\times\cdots \times  \big(\Z/\ell_{\theta}N\Z\big)$
for some $\ell_{1},\ldots, \ell_{\theta} \in \N$, 
and the Cartan matrix associated with $V$ corresponds to that of $B^{+}_{M}$.

\vspace{.1cm}
		
\item [\rm (b)]	
Let $\Gamma = 
\big(\Z/\ell_{1}N\Z\big)\times\cdots \times  \big(\Z/\ell_{\theta}N\Z\big)$ be a 
finite abelian group containing $M$
with generators $g_{m_{1}},\cdots, g_{m_{\theta}}$ and 
$\ord g_{m_{i}} = \ell_{i}N$,
$\mu= (\mu_{\alpha})_{\alpha\in Q^{+}}$ 
a family of root vector parameters, $\cD_{M}$ 
the finite Cartan datum for $\Gamma$ as defined in 
$\S$\ref{subsec:data-Cartan-type}  and $u(\cD_{M},\mu)$ 
the corresponding pointed Hopf algebra. 
Then there are a finite abelian group $\subg$ and an algebraic group monomorphism $k: \widehat{\subg} \hookrightarrow B^{+}_{M}$  
such that 
$$
u(\cD_{M},\mu) \simeq \cO_{\bf q}(B^{+}_{M})/(\Ker k^{*}),
$$
as Hopf algebras.	\qed	
\end{enumerate}

\end{teor}

\begin{obs}\label{rmk:computing-liftings-Cartan}
As in the case of braidings of Drinfeld-Jimbo type, the 
results above give a geometric interpretation of the liftings;  see Remark \ref{rmk:computing-liftings}.
The lifting of the power root vector relation \eqref{eq-HA-4}
is based on the isomorphism 
$\psi: U_{q}^{M}(\lieb^{+})\to \cO_{q}(B^{+}_{M})$, 
the monomorphism $k: \widehat{\subg} \hookrightarrow B^{+}_{M}$
and the fact that 
any braiding of Cartan type is related to a braiding of Drinfeld-Jimbo type
by a $2$-cocycle. This allows one to interpret the pointed Hopf algebras
$u(\cD_{M},\mu)$ as quantum subgroups of the 
multiparametric quantum groups  $\cO_{\bf q}(B^{+}_{M})$, with 
${\bf q} = (q_{ij})_{i,j\in I}$.
Explicitly,
\begin{align} \label{for-deformation}
x_{\alpha}^{N}&=p(E_{\alpha}^{N}) = p_{\mu}(\psi(E_{\alpha}^{N}))
=\varphi^{-1} k^{*}(\psi(E_{\alpha}^{N})),\quad \text{for all }  \alpha \in Q^{+}.
\end{align}
We recall that
$\varphi^{-1}k^{*}(\psi(E_{\alpha}^{N})) \in \C \subg$ and $\varphi^{-1}:\C^{\widehat{\subg}}\to \C \subg$ is the isomorphism considered in Theorem \ref{teo:quantum-subgrup-DJ}. 
	
\end{obs}

\section{Liftings for the cases $A_{\theta}$, $B_{\theta}$  and $D_{\theta}$
for Drinfeld-Jimbo type braidings}
\label{sec:PHa-qsubgroups}

Our goal in this section is
to recover the liftings of type $A_{\theta}$
computed in \cite{AS2}, of type $B_2$ computed in \cite{BDR}, and of type $B_3$ computed in \cite{BGM} for diagonal braidings of Drinfeld-Jimbo type, 
using our method.  
Furthermore, we calculate the liftings of type 
$B_{\theta}$ and $D_{\theta}$, for $\theta \geq 2$, giving in this way 
new explicit infinite families of examples. 

 \smallbreak
Let  $\mcar = {\big(a_{ij} \big)}_{i, j \in I} $  
 be a symmetrizable Cartan matrix of finite type,
 $ \lieg $  the finite-dimensional simple Lie
 algebra associated with  $ \mcar \, $, 
$ Q = \bigoplus_{i \in I} \Z \alpha_i $  the root lattice
and  $ P =\bigoplus_{i \in I} \Z\, \omega_i$
the weight lattice.  

\smallbreak

Recall that $q\in \C^{\times}$ is a primitive root of unity of odd order 
$N$. We assume that $\gcd(3,N)=1$ and $N>7$. Let 
$\Gamma = \Z^{\theta}/(n_{1}\Z\times \cdots \times n_{\theta}\Z)$ 
be the finite abelian group
with $\theta$ generators 
$g_{m_{1}},\ldots, g_{m_{\theta}}$ of order
$\ord g_{m_{i}} = n_{i}$ such that $N|n_{i}$ for all $i\in I$.
Consider also a lattice $Q\subseteq M\subseteq P$ of rank $\theta$
generated by elements $m_{1},\ldots, m_{\theta}$, and 
denote by 
$g:M\to \Gamma$ the canonical group epimorphism 
$m_{i}\mapsto g_{m_{i}}$, for all $i\in I$. \smallbreak

For $j\in I$, let $g_{j}=g(\alpha_{j})$ and
$\chi_{j}: \Gamma \to \C^{\times}$ be
the character defined by $\chi_{j}(g_{m_{i}}) = 
q^{(m_{i},\alpha_{j})}$
for all $i\in I$.
Then  
$\cD_{M} = \cD(\Gamma, (g_{i})_{i\in I}, (\chi_{i})_{i\in I}, (a_{ij})_{i,j\in I})$ is a
datum of finite Cartan type for $\Gamma$,
and ${\bf q}=(q^{d_{i}a_{ij}})_{i,j\in I}$ is a braiding of 
Drinfeld-Jimbo type.\smallbreak

In order to describe the algebra $u(\cD_{M},\mu)$ by 
using the commutative diagram \eqref{diag:theorem-liftings}, 
we need to explicitly provide the maps involved in the diagram below:
\begin{align}\label{squarem}
	\begin{aligned}
		\xymatrix{
			\cO(B_M^{+}) \ar@{^{(}->}[r]^{\iota} \ar@{->>}[d]_{\iota_{\mu}^{*}} &   \cO_q(B_M^{+}) \ar@{->>}[r]^{\pi} \ar@{->>}[d]^{p_{\mu}} & \overline{\cO_{q}(B_M^{+})} \ar[d]^{\phi}\\
			\C^{\widehat{\subg}}\ar@{^{(}->}[r]_{\bar{\iota}} & u(\cD_{M},\mu)  \ar@{->>}[r]_{\bar{\pi}} & u_q^{M}(\mathfrak{b}^{+}).
		} 
	\end{aligned} 
\end{align}

\subsection*{The maps $\bar{\iota}$ and $\bar{\pi}$}
The map $\varphi:\C \subg\to \C^{\widehat{\subg}}$
given by $\varphi(x)(f)=f(x)$ for all $x\in T_{M}$ and $f\in \widehat{\subg}$ is a Hopf algebra isomorphism. Besides, the coradical  of $ u(\cD_{M},\mu) $ is the group algebra $\C\Gamma$. Thus, the map $\bar{\iota}$ is obtained by the composition of $\varphi^{-1}$ and the natural inclusion 
of $\C \subg$ in $\C\Gamma$. Also, the Hopf algebra epimorphism $\bar{\pi}$ is exactly the map $\pi_u$ defined in Proposition \ref{prop:ces-u} (a).

\subsection*{The maps $\phi$ and $p_{\mu}$}
By Definition \ref{def:flkernels}, 
the small quantum group $u^{M}_q(\mathfrak{b}^{+})$ 
is the algebra generated by $h_{m},e_i$ 
with $m\in M$, $1\leq i\leq \theta$, satisfying the relations: for all $m\in M$,
\begin{align}
	\label{rel-uqb1}	h_{m}e_jh^{-1}_{m}&=
	q^{(m,\alpha_j)}e_j,& &\ad(e_i)^{1-a_{ij}}(e_j)=0,\,\,\,\,i\neq j,& \\
	\label{rel-uqb2}	h_{m}^{N}&=1,& &\qquad e_i^{N}=0.&
\end{align}

As in Proposition \ref{prop:U-u} (a), there 
exists a Hopf algebra epimorphism 
$p:U_{q}^{M}(\lieb^{+})\to u(\cD_{M},\mu)$ given by 
\[
p(E_i)=x_i,\qquad p(K_{m})=g_{m},\quad \text{ for all } m\in M.
\]
Taking $\psi:U_{q}^{M}(\lieb^{+})\to \cO_{q}(B^{+}_{M})$
to be the Hopf algebra isomorphism from 
Remark \ref{rmk:zeta+}, one gets the 
Hopf algebra epimorphism 
$p_{\mu}:\cO_q(B_M^{+})\to u(\cD_{M},\mu)$ 
by composing $p$ with $\psi^{-1}$. 
The exact definition of the latter map highly depends on a 
presentation of $\cO_q(B_M^{+})$ by generators and relations,
which is in general given by the FRT-construction.
This will be clear in the examples below.

\subsection*{The map $\iota_{\mu}^{*}$}
For the classical cases of Cartan type, 
we use the presentations $\cO_{q}(B^{+}_{M})$
given by the FRT-construction following \cite[Chapter 9]{KS}.
Let $\ttn\in \N$ be such that $\ttn=\theta + 1$ if $G$ is of type 
$A_{\theta}$, $\ttn=2\theta + 1$ if $G$ is of type $B_{\theta}$ and 
$\ttn=2\theta$ if $G$ is of type $D_{\theta}$.

Moreover, in case $A_{\theta}$ we take $M=P$, so the 
classical coordinate algebra $\cO(\SL_{\ttn})$ 
is presented by 
$$
\cO(\SL_{\ttn}) = \C[X_{ij}\, |\, 
\det X = 1,\, 1\leq i,j \leq \ttn], 
$$
where $\det X = \sum_{\sigma\in \mathbb{S}_{\ttn}}
(-1)^{\ell(\sigma)}X_{1\sigma(1)}\cdots X_{\theta+1\sigma(\ttn)}$.

\smallbreak
On the other hand, for the cases $B_{\theta}$ and $D_{\theta}$
we take $M=Q$, so the coordinate algebras $\cO(\operatorname{SO}_{\ttn})$ 
are presented by 
\begin{align}\label{alg-so-relations}
\cO(\operatorname{SO}_{\ttn})= 
\C\big[X_{ij}\, |\, 
\sum_{k=1}^{\ttn}X_{kj}X_{k'i'}=\delta_{ij}, \det X =1,1\leq i,j \leq \ttn\big].
\end{align}

We define the map $\iota_{\mu}^{*}$ by using a natural inclusion of
the group $ \widehat{\subg}$ into the diagonal split torus of
$ B^{+}_{M}$ and the conjugation by a 
fixed unipotent matrix. A particular choice of this unipotent matrix 
gives the specific presentation of the liftings we are looking for.

\smallbreak
Let  
$ \widehat{\subg}\hookrightarrow B^{+}_{M}$ be a 
group monomorphism such that the image of each 
$\gamma$ in $\widehat{T_{M}}$ is a diagonal matrix $P_{\gamma}$ 
for all $\gamma\in\widehat{\subg}$.
For a fixed unipotent matrix $Q_{r}$ in $B_M^{+}$ of the form  
\begin{align}\label{unipotent}
	Q_{r}=\begin{pmatrix}
		1 & r_{12} & r_{13}& \ldots & r_{1\, \ttn}\\
		0& 1 & r_{23}& \ldots & r_{2\, \ttn}\\
		\vdots &  & \vdots &\ddots &\vdots\\
		0 &0 & 0 & \ldots & r_{\ttn-1\,  \ttn}\\
		0 &0 & 0 & \ldots & 1\\
	\end{pmatrix}
\end{align}
we define the map 
$\iota_{r}:\widehat{\subg}\to B^{+}$ given by 
$\iota_{r}(\gamma)=Q_{r}^{-1} P_\gamma Q_{r}$ for all
 $\gamma\in \widehat{\subg}$. By its very definition, 
 $\iota_{r}$ is a group monomorphism. 
For $P_\gamma = \text{diag}(t_1, \cdots, t_{\ttn})$ one has the following recursive formula for the entries of $\iota_{r}(\gamma)$
\begin{equation}\label{eq:formula-i-mu}
	\left(Q_{r}^{-1}P_\gamma Q_{r}\right)_{i\, j}=
	\begin{cases}
		t_i  & \text{ if } j=i,\\
		r_{ij}t_i -\sum\limits_{k=i+1}^{j} r_{i k}\left(Q_{r}^{-1}P_\gamma Q_{r}\right)_{k\, j}  & \text{ if } i< j\le  \ttn. 
	\end{cases} 
\end{equation}
Thus, the Hopf algebra 
epimorphism $\iota_{r}^{*}: \cO(B_M^{+}) \to \C^{\widehat{\subg}}$ 
reads
\begin{align}
	\label{for-forever}
\iota_{r}^*(X_{ij})=
r_{ij}\iota_{r}^*(X_{ii})-r_{ij}\iota_{r}^*(X_{jj}) - 
\textstyle\sum\limits_{k=i+1}^{j-1} r_{i k}\iota_{r}^*(X_{kj}),
\end{align}
where $X_{ij}$, with $1\leq i\leq j\leq \ttn$, 
are the generators of $\cO(B^{+}_{M})$. 
Since the entries $r_{ij}$ of the matrix $Q_r$ will correspond to
the family of root vectors parameters $\mu$, we will write eventually 
$\iota_{\mu}^*$ instead of $\iota_{r}^*$.

\subsection{$A_{\theta}$-type}
In this subsection we consider $G= G_{sc}=
\operatorname{SL}_{\theta+1}(\C)$ and 
$\lieg= \liesl_{\theta+1}(\C)$, so $M=P$ is the weight lattice. 
In this case, we set $B^{+}:=B^{+}_{M}$ as the 
subgroup of upper triangular matrices of $\SL_{\theta+1}(\C)$
and $\lieb^{+}$ the corresponding Lie algebra. 

\smallbreak

From \S \ref{subsec:OqSLn} we know that 
the Hopf algebra $\cO_q(\SL_{\theta+1})$ is generated by 
the elements $z_{ij}$ satisfying the following relations:
\begin{align}
	\label{rel-sln-1}&z_{is}z_{js}=qz_{js}z_{is},& &z_{si}z_{sj}=qz_{sj}z_{si},& &i<j,&\\[.2em]
	\label{rel-sln-2}&z_{it}z_{js}=z_{js}z_{it},& &z_{is}z_{jt}-z_{jt}z_{is}=(q-q^{-1})z_{it}z_{js},& &i<j,\,\,s< t.&\\[.2em]
	\label{rel-sln-3}& & &
	\sum_{\sigma\in \mathbb{S}_{\theta+1}}
	(-q)^{\ell(\sigma)}z_{1\sigma(1)}\cdots z_{\theta+1\sigma(\theta+1)} = 1.   && & 
\end{align} 
The set $\{z^{N}_{ij}\,:\, \,\,1\leq i,j\leq \theta+1\}$ generates 
a central Hopf subalgebra of $\cO_q(\SL_{\theta + 1})$ 
isomorphic 
to the coordinate algebra 
$\cO(\SL_{\theta+1})$. 
The inclusion
\begin{align}\label{map:inclusion-Hopf}
&	\cO(\SL_{\theta+1}) \hookrightarrow \cO_{q}(\SL_{\theta+1}),\qquad X_{ij}\mapsto z_{ij}^{N},&
\end{align}
is a Hopf algebra map which is the dual map of the quantum Frobenius map;
for more details we refer to \cite{PW}. The quantum function algebra $\cO_q(B^{+})$ corresponding to the 
standard Borel subgroup $B^{+}$ of $\SL_{\theta+1}(\C)$
is the quotient
$\cO_q(\SL_{\theta+1})/\mathcal{I}$ of $\cO_q(\SL_{\theta+1})$ by the Hopf ideal $\mathcal{I}$ generated by 
the elements $\{z_{ij}\,:\,i> j\}$.  
In particular, $\cO_q(B^{+})$ is the Hopf algebra generated by (the image of) $z_{ij}$ with $i\leq j$, 
satisfying the relations \eqref{rel-sln-1}, \eqref{rel-sln-2} and 
$z_{11}z_{22}\cdots z_{\theta+1,\theta+1}=1$ 
instead of \eqref{rel-sln-3}.
The corresponding Hopf ideal of 
$\cO(\SL_{\theta+1})$ that gives the quotient 
$\cO(B^{+})$ is generated by 
the elements $\{X_{ij}\,:\,i> j\}$.
\smallbreak

The map $\iota: \cO(B^{+})\hookrightarrow \cO_{q}(B^{+})$
is defined by $X_{ij}\mapsto z_{ij}^{N}$ for all $i\leq j$. Also,  $\overline{\cO_{q}(B^{+})} = 
\cO_{q}(B^{+})/\big( \iota(\cO(B^{+}))^{+}\cO_{q}(B^{+})\big) $ 
is the algebra 
generated by the elements $z_{ij}$, with $i\leq j$, satisfying the 
relations \eqref{rel-sln-1}, \eqref{rel-sln-2} and the following ones  
\begin{align} \label{rel-quotient}
	&z_{11}\cdots z_{\theta+1,\theta+1} = 1,&&	
	z_{ij}^{N} = \delta_{ij},\quad \text{for all } i\leq j.&
\end{align}

The map $\pi:\cO_q(B_M^{+})\to \overline{\cO_{q}(B_M^{+})} $ 
is the canonical epimorphism.
Observe that the elements $z_{ii}$ are (invertible) 
group-like elements both in $\cO_{q}(B^{+})$ and
$\overline{\cO_{q}(B^{+})}$.

Restricting the isomorphism 
$\psi:U_{q}^{P}(\lieb^{+}) \to \cO_{q}(B^{+})$ to the commutative
subalgebra $Z^{\geq}:=Z^{\geq}_{P}$ gives the isomorphism
$$
\psi: Z^{\geq} \to \cO(B^{+}),
\qquad K_{\omega_{i}}^{\pm N}\mapsto (X_{1,1}\ldots X_{i,i})^{\pm 1},
\quad E_{j}^N \mapsto - (q-q^{-1})^{-N}X_{j,j+1}X^{-1}_{j+1,j+1}.
$$
\smallbreak
Let $\Pi=\{\alpha_{1}, \ldots, \alpha_\theta\}$ be the simple  roots of $A_\theta$ and  $\{\omega_1, \ldots, \omega_\theta\}$ the corresponding   fundamental weights. 
We follow \cite{Bou} and consider a realization of the root system in the 
canonical  orthonormal basis $\{\mathbf{e}_i\}$ of $\mathbb{R}^{\theta+1}$.
In this case, the positive roots are given by 
$\alpha_{ij}:=\mathbf{e}_i - \mathbf{e}_j$ for all $1\leq i<j\leq \theta+1$. 
The simple roots are $\alpha_{i}=\alpha_{i\,i+1}= \mathbf{e}_i - \mathbf{e}_{i+1}$
for all $i\in I$, and the fundamental weights are
\begin{equation}\label{eq:cij}
	\omega_i =  \mathbf{e}_1+\cdots + \mathbf{e}_i 
	-\tfrac{i}{\theta+1}\left(\mathbf{e}_1+\cdots+\mathbf{e}_{\theta+1}\right)
	=\sum^{\theta+1}_{j=1} c_{ij} \mathbf{e}_j,\quad i\in I.  
\end{equation}
The group $T:=T_{M}=T_{P}$ is the subgroup of $\Gamma$
generated by the elements $g_{\omega_{i}}^{N}$ for all $i\in I$.
For each $i\in I$, let $\ell_i:=n_i/N$ and let $\xi_i$ be a primitive $\ell_i$-th root of the unity.
Then we may describe the character group $\widehat{T}$ of $T$
as the finite abelian 
group generated by the characters $\gamma_{j}$ such that 
$\gamma_{j}(g_{\omega_{i}}^{N}) = \xi_{j}^{\sum_{l=1}^{i}c_{jl}} $ for all $i,j\in I$. Since $g_{\alpha_i}=g^{-1}_{\omega_{i-1}}g^2_{\omega_i}g^{-1}_{\omega_{i+1}}$, it follows that 
\begin{align}\label{eq-gamma}
	\gamma_j(g^N_{\alpha_i})=\xi_{j}^{c_{ji}-c_{j(i+1)}}=\xi_j^{\delta_{i,j}}
	\quad \forall\ i,j\in I.	
\end{align}

For each $i\in I$, consider the diagonal matrix 
$P_i:= \text{diag}(\xi_i^{c_{i1}},\ldots,\xi_i^{c_{i\theta+1}})$.
Notice that $P_i\in B^{+}$ because $\sum_{j=1}^{\theta+1}c_{ij}=0$. 
Now, consider the following group monomorphism 
\begin{align}\label{group-morphism}
&\widehat{T}\hookrightarrow B^+,& &\gamma\mapsto P_{\gamma}=P_1^{s_1}\cdots P_\theta^{s_\theta},\qquad\text{ for all } \,\gamma=\gamma_{1}^{s_1}\ldots\gamma_{\theta}^{s_{\theta}}\in \widehat{T}.&	
\end{align}

Given a family
 $\mu=(\mu_{\alpha})_{\alpha\in Q^{+}}=(\mu_{ij})_{1\leq i<j\leq \theta+1}$ of root vector parameters,
we take the matrix $Q_{r}$ of \eqref{unipotent} with the following entries \[r_{ij}=(-1)^{i-j-1}(1-q^{-2})^{N}\mu_{ij}.\]

\begin{obs}\label{rmk:iota-transpose}
The composition of $\psi: Z^{\geq} \to \cO(B^{+})$ and $\iota_{\mu}^{*}$
gives a Hopf algebra epimorphism 
$\iota_{\mu}^{*}\psi: Z^{\geq} \to \C^{\widehat{T}}$
which satisfies 
$\iota_{\mu}^{*}\psi(K_{\alpha_i}^{N})(\gamma) =
 t_{i}t_{i+1}^{-1}=\xi_1^{s_1 \delta_{1,i}} \cdots 
\xi_\theta^{s_\theta \delta_{\theta,i}}$
for all $\gamma=\gamma_{1}^{s_{1}}\cdots \gamma_{\theta}^{s_{\theta}} 
\in \widehat{T}$, see \eqref{eq:formula-i-mu}. 
Also, by \eqref{eq-gamma} we have $\varphi(g^N_{\alpha_i})(\gamma)=\xi_1^{s_1 \delta_{1,i}} \cdots 
	\xi_\theta^{s_\theta \delta_{\theta,i}}$. Hence,
$$
g_i^N=g_{\alpha_i}^N=\varphi^{-1}\iota_{\mu}^{*}\psi(K_{\alpha_i}^{N}), 
\quad \text{for all } i\in I.
$$
\end{obs}

Following \cite[pg. 47]{AS2}, for each $1\leq i<j\leq \theta+1$, we write
\begin{align*}
	E_{i,j}&=\begin{cases}
		E_i,& \text{if } j=i+1,\\
		E_{i,j-1}E_{j-1}-q^{-1}E_{j-1}E_{i,j-1}, & \text{if } j-i\geq 2.
	\end{cases}
\end{align*}
Also, for each $1\leq i<j\leq \theta+1$, we set 
$x_{(ij)}:=p_{\mu}(E_{i,j})$ and $ g_{(ij)}:=g_ig_{i+1}\ldots g_{j-1}$.

\bigbreak
In the next theorem we recover using our method 
the formula for the liftings of 
braided vector spaces of type $A_{\theta}$ in the Drinfeld-Jimbo type
from \cite{AS2}.

\begin{teor}\label{lifting-case-an} 
Let $\mu=(\mu_{\alpha})_{\alpha\in \Phi^{+}}= 
(\mu_{ij})_{1\leq i < j \leq \theta+1}$
be a family of root vector parameters. 
Following Remark \ref{rmk:computing-liftings}, 
relation \eqref{eq-HA-4} is given by 	
\[
x_{(ij)}^{N} = \mu_{ij}(1-g_{(i j)}^{N})\,\, 
+\sum\limits_{k=i+1}^{j}  (1-q^{-2})^N \mu_{i k} \,\,x_{(k j)}^{N}.
\]
\end{teor}

\begin{proof}
Let $\psi:U_q^{P}(\lieb^+)\to \cO_{q}(B^+)$ be the Hopf algebra isomorphism given in Remark \ref{rmk:zeta+}. By induction, we obtain that
\[
\psi(E_{i,j})=(-1)^{j-i}q^{i-j+1}(q-q^{-1})^{-1}z_{ij}z^{-1}_{jj} 
\qquad 1\leq i < j \leq \theta+1.
\]
Thus, $\psi(E_{i,j}^N)=\lambda_{i,j}z_{ij}^Nz^{-N}_{jj}$, where $\lambda_{i,j}:=(-1)^{j-i}(1-q^{-2})^{-N}\in \C$. Hence, for all 
$\gamma\in \widehat{T}$ we have
\begin{align*}
\iota_{\mu}^{*}(\psi(E_{i,j}^{N}))(\gamma) &= 
\lambda_{i,j}(Q_{\mu}^{-1}P_{\gamma}Q_{\mu})_{ij} 
(Q_{\mu}^{-1}P_{\gamma}Q_{\mu})^{-1}_{jj}\\
&\overset{\mathclap{\eqref{for-forever}}}{=} 
\lambda_{i,j}\Big(r_{ij}(t_i-t_j)-\sum_{k=i+1}^{j-1}r_{ik}(Q_{\mu}^{-1}P_{\gamma}Q_{\mu})_{kj}\Big)t^{-1}_{j}\\
&= \lambda_{i,j}r_{ij}(t_it^{-1}_{j}-1)-\sum_{k=i+1}^{j-1}\lambda_{i,j}r_{ik}(Q_{\mu}^{-1}P_{\gamma}Q_{\mu})_{kj} (Q_{\mu}^{-1}P_{\gamma}Q_{\mu})^{-1}_{jj}\\
		&= \lambda_{i-1,j}r_{ij}(1-t_it^{-1}_{j})-\sum_{k=i+1}^{j-1} (-1)^{k-i}r_{ik}\lambda_{k,j}(Q_{\mu}^{-1}P_{\gamma}Q_{\mu})_{kj} (Q_{\mu}^{-1}P_{\gamma}Q_{\mu})^{-1}_{jj}\\
		&= \lambda_{i-1,j}r_{ij}(1-t_it^{-1}_{j})-\sum_{k=i+1}^{j-1} (-1)^{k-i}r_{ik}\iota_{\mu}^{*}(\psi(E_{k,j}^{N}))(\gamma) \\
		&= \mu_{ij}(1-t_it_{i+1}^{-1}t_{i+1}t_{i+2}^{-1}\ldots t_{j-1}t^{-1}_{j})+(1-q^{-2})^{N}\sum_{k=i+1}^{j-1}\mu_{ik}\iota_{\mu}^{*}(\psi(E_{k,j}^{N}))(\gamma) \\
		&= \mu_{ij}(1-\iota_{\mu}^{*}\psi(K_{\alpha_i}^{N}\ldots K_{\alpha_{j-1}}^{N})(\gamma))+(1-q^{-2})^{N}\sum_{k=i+1}^{j-1}\mu_{ik}\iota_{\mu}^{*}(\psi(E_{k,j}^{N}))(\gamma) 
	\end{align*}
	Then, using Remark \ref{rmk:iota-transpose}, we conclude that 
	\begin{align*}
		x^N_{(ij)}=\varphi^{-1}\iota_{\mu}^{*}(\psi(E_{i,j}^{N}))&= \mu_{ij}(1-g_i^N\ldots g^N_{j-1})+(1-q^{-2})^{N}\sum_{k=i+1}^{j-1}\mu_{ik}x^N_{(kj)}\\
		&=\mu_{ij}(1-g_{(i j)}^{N})\,\, +\sum\limits_{k=i+1}^{j}  (1-q^{-2})^N \mu_{i k} \,\,x_{(k j)}^{N}
	\end{align*}
as desired.
\end{proof}

\subsection{Explicit Hopf inclusion for the quantum Borel in $\SO_q(\ttn)$}\label{subsec:case Bn}
We start this subsection recalling the definition of $\mathcal{O}_q(\SO_\ttn)$;
we follow \cite[9.3]{KS}.  Consider  
$G= G_{ad}= \operatorname{\SO}_{\ttn}(\C)$ and 
$\lieg= \lieso_{\ttn}(\C)$, so $M=Q$ is the root lattice. 
For each $1\leq k\leq \ttn$, let  $k':=\ttn+1-k$ and 
\[
\rho_i=\ttn/2-i\,\,\text{ if $i<i'$ },\qquad \rho_{i'}=-\rho_i\,\,\text{ if $i\leq i'$. } 
\]
Let $J=(c_{ij}) $ be the matrix $\ttn\times \ttn$ with entries 
$c_{ij}:=\delta_{i,j'}q^{-\rho_i}$. We pick the R-matrix given by $R=\sum_{a,b,i,j} R_{a b}^{ij} E_{i a}\otimes E_{j b}$ where the coefficient  $R_{a b}^{ij}$ is defined as follows
\[
R_{a b}^{ij} = q^{\delta_{i\, j}-\delta_{i\, j'}}  \delta_{i\, a} \delta_{j\, b}  + 
\mathsf{u}(i-a)(q-q^{-1})(\delta_{j\, a}\delta_{i\, b}-\epsilon c_{j\, i}c_{a \, b})
\] 
and $\mathsf{u}(x)=0$ if $ x<1$ and $\mathsf{u}(x)=1$ if $ x\geq 1$.

In the free associative algebra generated by 
$\{z_{ij}\,:\,1\leq i,j\leq \ttn\}$, for $1\leq a,b,i,j\leq \ttn$, 
consider the element 
\begin{align}
	\label{FRT}
	\mathrm{FRT} (i,b,j,a)=\sum_{k,l=1}^{\ttn}R_{a\, b}^{l\, k} z_{i\, k}z_{j\, l} -R_{k\, l}^{j\, i} z_{k a}z_{l\, b}.
\end{align} 
The quantum function algebra is defined by
\[
\mathcal{O}_q(\SO_\ttn)=\mathbb{C}\langle z_{ij}\rangle/\langle \mathrm{FRT} (i,b,j,a),\, \mathbf{z}J\mathbf{z}^tJ^{-1}-\id,\, J\mathbf{z}^tJ^{-1}\mathbf{z}-\id \rangle, \qquad \mathbf{z}=(z_{ij}). 
\]
The coalgebra structure and the antipode map are given respectively by:
\begin{align}\label{coalg-antipode}
&\Delta(z_{ij})=\sum_{k=1}^{j}z_{ik}\otimes z_{kj},& &\varepsilon(z_{i,j})=\delta_{i,j},& &S(z_{ij})=q^{\rho_j-\rho_i}z_{j'\,i'}.&	
\end{align}
Similarly as in case $A_{\theta}$, the quantum function algebra $\cO_q(B^{+})$ corresponding to the 
standard Borel subgroup $B^{+}$ of $\SO_\ttn(\mathbb{C})$
is the quotient
$\cO_q(\SO_{\ttn})/\mathcal{I}$ of $\cO_q(\SO_\ttn)$ by the Hopf ideal $\mathcal{I}$ generated by 
the elements $\{z_{ij}\,:\,i> j\}$.  

The function algebra 
$\cO(\operatorname{SO}_{\ell})$
is a commutative Hopf algebra with comultiplication given by 
$\Delta(X_{ij})=\sum_{k=1}^{\ttn}X_{ik}\otimes X_{kj}$.
As in $A_{\theta}$-case, the set $\{z^{N}_{ij}\,:\, \,\,1\leq i,j\leq \ttn\}$ generates 
a central Hopf subalgebra of $\cO_q(\operatorname{SO}_{\ttn} )$ isomorphic
to $\cO(\operatorname{SO}_{\ttn})$. However, for $\ttn$ odd, 
the inclusion map given in \eqref{map:inclusion-Hopf} is not a 
coalgebra map as we will see in the next result.

Given $i,j\in \N$ and $i\leq j$, we fix $\mathbb{I}_{i,j}=\{i,\ldots,j\}$ and  
$\tilde{\mathbb{I}}_{i,j}=\{i+1,\ldots,j-1\}$. Set
\[
\ttn_0=(\ttn+1)/2,\qquad\qquad c^k_{ij} =\begin{cases}
	\tfrac{2}{(1+q)^N},\ &\text{if } k=\ttn_0 \text{ and } k\in \tilde{\mathbb{I}}_{i,j},\\
	1,&\text{otherwise. }
\end{cases} 
\] 
Notice that if $\ttn$ is even, then  $c^k_{ij}=1$ for all $k=i,\ldots,j$. 
In particular, this happens in case $D_{\theta}$ since $\ttn=2\theta$.

\begin{teor}\label{teorema: coproducto potencia N}
	The comultiplication $\Delta: \mathcal{O}_q(B^+)\to  \mathcal{O}_q(B^+)\otimes  \mathcal{O}_q(B^+)$ satisfies 
\begin{align}
	\label{formula-coproduto}
	\Delta(z_{ij}^N)=\sum_{k=i}^j c^k_{ij} z_{ik}^N \otimes z_{kj}^N, \quad\quad \text{for all }\, 1\leq i\leq j\leq \ttn.
\end{align}
\end{teor}
 
 \begin{proof}
 Computing the relation $\mathrm{FRT} (i,t,i,s)$ given in \eqref{FRT} for 
 $i\le  t<s$, we obtain the following identities:
\begin{align}
\label{misma fila no hat}      z_{is} z_{it}&=q^{-1}z_{it}z_{is}, & &\text{if } i\le  t<s,\  s\neq t',\ i\neq \ttn_0,\\
\label{misma fila hat} z_{it'} z_{it}&=q^{-2}z_{it}z_{it'}-(q - q^{-1})\sum_{k=i}^{t-1} q^{k-t-1}z_{ik}z_{ik'},& &\text{if } i\le  t<t',\ i\neq \ttn_0,\\
\label{fila-mitad}    z_{\ttn_0s} z_{\ttn_0t}&=z_{\ttn_0s} z_{\ttn_0t} - (q - q^{-1})\sum_{k=\ttn_0+1}^{t} q^{\ttn_0-k+1/2}z_{k's}z_{kt}& &\text{if } \ttn_0\le  t<s.
\end{align}
Similarly, computing  $\mathrm{FRT} (s,j,t,j)$ for $t<s\le j$, it follows:
\begin{align}
\label{misma columna no hat}	z_{sj} z_{tj}&=q^{-1} z_{tj}z_{sj} & &\text{for }  t<s\le j,\  s\neq t',\ j\neq \ttn_0,\\
\label{misma columna hat} z_{t'j} z_{tj}&=q^{-2}z_{tj}z_{t'j}-(q - q^{-1})\sum_{r=j'}^{t-1} q^{r-t-1 }z_{rj}z_{r'j }& &\text{if }   t<t'\le j,\\
\label{columna-mitad} z_{\ttn_0j} z_{tj}&=q^{-1} z_{tj} z_{\ttn_0j},  & &\text{for }  t<\ttn_0 < j.
\end{align}
Moreover,  denoting $F_{i,k}= \mathrm{FRT} (i, k', i, k)$, we have that
	\begin{align}
	0=F_{i,i'}   + q^2 F_{i,i} +   (q^2-1) \sum_{i<k\le \ttn /2} q^{k - i} F_{i,k}.
	\end{align}
Applying the antipode map in the previous identity, we obtain the following relations:
	\begin{align}\label{misma fila hasta ell0}
	0&=\sum_{i\le l< \ttn_0} q^{l}z_{il}z_{il'} + \frac{q^{\ttn_0+1/2}}{1+q}z_{i\ttn_0}z_{i\ttn_0},& &\text{if }  i<i',\\\label{misma columna hasta ell0}
		0&=\sum_{j'\le r<\ttn_0} q^{r}z_{rj}z_{r'j } +\frac{q^{\ttn_0+1/2}}{1+q} z_{\ttn_0j}z_{\ttn_0j},& &\text{if }    j'<j,
	\end{align}
where the last term on the right-hand side of equations 
\eqref{misma fila hasta ell0} and \eqref{misma columna hasta ell0} 
is defined as zero if $\ttn$ is even.
		Thus, for any  $i\le t< s\le j$ such that $t\neq s'$ and $\ttn_0\notin\{i,j\}$, it follows from \eqref{misma fila no hat} and \eqref{misma columna no hat} that
\begin{align}\label{qconmutacion tensor}
	(z_{i,s} \otimes   z_{s,j})  (z_{i,t}\otimes z_{t,j})=q^{-2} (z_{i,t}\otimes z_{t,j}) (z_{i,s}\otimes z_{s,j})
\end{align}

To complete the proof, we will proceed by cases. 

\vspace{.2cm}

{\noindent \it{Case 1: $i<j<j'$}. } 
By \eqref{qconmutacion tensor}, all summands in the coproduct formula 
\eqref{coalg-antipode} are elements that $q^{-2}$-commute. Thus
\[
\Delta(z_{ij}^N)=\sum_{k=i}^j  z_{ik}^N\otimes z_{kj}^N.
\]

{\noindent \it{Case 2: $i=\ttn_0<j$}. } 
In this case
	\[
	\Delta (z_{\ttn_0 j})= \sum_{k=\ttn_0}^j z_{\ttn_0  k}\otimes z_{kj}.
	\]
For $\ttn_0\le t\le j$, set $\mathtt{b}_{t}:=z_{\ttn_0 t}\otimes z_{tj}$ and  
$\mathtt{b}_{t,j}:=\sum_{k=t}^j z_{\ttn_0 k}\otimes z_{kj}$. 
From \eqref{fila-mitad} and \eqref{misma columna no hat} follows 
\[
\mathtt{c}_t :=\mathtt{b}_{t}\mathtt{b}_{t+1,j}-q \mathtt{b}_{t+1,j}\mathtt{b}_{t}= -(q-q^{-1})q^{\ttn_0+3/2}   \sum_{s=t+1}^j\sum_{k=\ttn_0+1}^t q^{-k} z_{k's}z_{kt}\ot z_{sj}z_{tj}.  
\]
It can be checked  that
\begin{align*}
	\mathtt{b}_{t}\mathtt{c}_t &=q^2\mathtt{c}_{t}\mathtt{b}_t,&  \mathtt{c}_t\mathtt{b}_{t+1,j}&=q^ 2 \mathtt{b}_{t+1,j}\mathtt{c}_t
\end{align*}
Hence, by  \cite[Corollary 2.3]{BGM}, for all $\ttn_0\le t\le j$ we have
\[
(\mathtt{b}_{t,j})^N=(\mathtt{b}_{t}+\mathtt{b}_{t+1,j})^N=\mathtt{b}_{t}^N+\mathtt{b}_{t+1,j}^N.
\]
Applying the previous relation recursively for $t=\ttn_0,\ldots j-1$ we get
\begin{align*}	
	\Delta (z^N_{\ttn_0j}) &= (\mathtt{b}_{\ttn_0}+\cdots +\mathtt{b}_{j})^N= (\mathtt{b}_{\ttn_0}+\mathtt{b}_{\ttn_0+1, j})^N\\
	&=\mathtt{b}_{\ttn_0}^N+\mathtt{b}_{\ttn_0+1, j}^N=\mathtt{b}_{\ttn_0}^N+(\mathtt{b}_{\ttn_0+1}+\mathtt{b}_{\ttn_0+2, j})^N\\
	&=\mathtt{b}_{\ttn_0}^N+\mathtt{b}_{\ttn_0+1}^N+\mathtt{b}_{\ttn_0+2, j}^N=\mathtt{b}_{\ttn_0}^N+\mathtt{b}_{\ttn_0+1}^N+(\mathtt{b}_{\ttn_0+2}+\mathtt{b}_{\ttn_0+3, j})^N\\
	&=\cdots =\mathtt{b}_{\ttn_0}^ N+\cdots +\mathtt{b}_{j}^N\\
	&=\sum_{k=\ttn_0}^j z_{\ttn_0\,  k}^N\otimes z_{kj}^N.
\end{align*}

{\noindent \it{Case 3: $i<i'\le j$}. } 
For this case, we set $\beta_1={1}/({1+q})$ and $\beta_2={q}/({1+q})$.  
Let $\mathtt{a}_0=0$ if $\ttn$ is even, and 
$\mathtt{a}_0=z_{i\, \ttn_0} \otimes z_{\ttn_0\, j}$ if $\ttn$ is odd. 
Consider 
\begin{align*}
\mathtt{a}_1 &= \beta_1 \mathtt{a}_0 + \sum_{i\le k< \ttn_0}z_{i k}\otimes z_{kj},&  \mathtt{a}_2 &= \beta_2 \mathtt{a}_0 + \sum_{\ttn_0<k\le i'} z_{i k}\otimes z_{kj},&
\mathtt{a}_3 &= \sum_{k=i'+1}^j z_{i k}\otimes z_{kj}.
\end{align*}
Then $\Delta(z_{ij}) = \mathtt{a}_1+\mathtt{a}_2+ \mathtt{a}_3$ and, 
by \eqref{qconmutacion tensor}, we have that  
$\mathtt{a}_s \mathtt{a}_3 = q^2\,  \mathtt{a}_3 \mathtt{a}_s$, for 
$s=1,2$. 

\vspace{.3cm}

{\noindent {\bf Claim.} 
$\mathtt{a}_1 \mathtt{a}_2 = q^2\,  \mathtt{a}_2 \mathtt{a}_1$.}

\vspace{.3cm}

\noindent This can be  checked by a long straightforward computation, 
we outline the main ideas in what follows. 
Fix  $\tilde{\mathtt{a}}_1 =\mathtt{a}_1 -\beta_{1}\mathtt{a}_0$ and  
$\tilde{\mathtt{a}}_2 =\mathtt{a}_2 -\beta_{2}\mathtt{a}_0$. 
Then
\begin{align*}
	q^{-2}\tilde{\mathtt{a}}_1 \tilde{\mathtt{a}}_2 - \,  \tilde{\mathtt{a}}_2 \tilde{\mathtt{a}}_1
	&= \sum_{i\le t < \ttn_0 < s\le j} (q^{-2} z_{i t} z_{i s} \otimes  z_{t j}z_{sj} -  z_{i s} z_{i t}\otimes z_{sj} z_{t j})\\
	&= \sum_{i\le t< \ttn_0}(q^{-2} z_{i t} z_{i t'} \otimes  z_{t j}z_{t'j} -  z_{i t'} z_{i t}\otimes z_{t'j} z_{t j})  +\\
	&\quad \sum_{\substack{i\le t \le \ttn_0\le s\le j\\
			t\neq s'}} (q^{-2} z_{i t} z_{i s} \otimes  z_{t j}z_{sj} -  z_{i s} z_{i t}\otimes z_{sj} z_{t j})\\
	&\overset{\eqref{qconmutacion tensor}}{=}\sum_{i\le t< \ttn_0}(q^{-2} z_{i t} z_{i t'} \otimes  z_{t j}z_{t'j} -  z_{i t'} z_{i t}\otimes z_{t'j} z_{t j} ).
\end{align*}
Recalling that $\lambda=q-q^{-1}$, we observe that
\begin{align*}
&\sum_{i\le t< \ttn_0} (q^{-2} z_{i t} z_{i t'} \otimes  z_{t j}z_{t'j} -  z_{i t'} z_{i t}\otimes z_{t'j} z_{t j}) = \\[.2em]
&\overset{\eqref{misma fila hat}}{=}\sum_{i\le t< \ttn_0}	(q^{-2} z_{i t} z_{i t'} \otimes  z_{t j}z_{t'j} -  q^{-2}z_{it}z_{it'}\otimes  z_{t'j} z_{t j})
+\lambda \sum_{i\le t< \ttn_0}\sum_{i\le s< t}	   q^{s-t-1}z_{is}z_{is'}\otimes  z_{t'j} z_{tj}\\
&=\sum_{i\le t< \ttn_0}	  z_{it}z_{it'}\otimes  (q^{-2} z_{t j} z_{t'j} -  q^{-2}z_{t'j} z_{t j})
+\lambda\sum_{i\le s< \ttn_0}	  \left( \sum_{s+1\le t< \ttn_0} q^{s-t- 1} \right) z_{is}z_{is'}\otimes  z_{t'j} z_{tj}\\
&=\sum_{i\le t< \ttn_0}	 z_{it}z_{it'}\otimes  \left( q^{-2} z_{t j} z_{t'j} - q^{-2}z_{t'j} z_{t j}
+\lambda\sum_{t+1\le a< \ttn_0} q^{t-a-1} z_{a'j} z_{a j}  \right).
\end{align*}
Using \eqref{misma columna hat}  for the terms $z_{t'j} z_{t j}$ and $z_{a'j} z_{a j}$, the  parenthesis in the previous line can be simplified to:
\begin{align*}
q^{-2} z_{tj}z_{t'j} -q^{-2}z_{t'j} z_{t j} +\lambda \sum_{t+1\le a< \ttn_0} q^{t-a-1} z_{a'j} z_{a j}=\lambda q^{t-\ttn+\tfrac{1}{2}((-1)^\ttn -3)} \sum_{j'\le r< \ttn_0}  q^{r} z_{r j}z_{r'j }.
\end{align*}
Hence
\begin{align}\label{parenteis-cuenta-larga}
q^{-2}\tilde{\mathtt{a}}_1 \tilde{\mathtt{a}}_2 - \,  \tilde{\mathtt{a}}_2 \tilde{\mathtt{a}}_1&=  \lambda q^{-\ttn+\tfrac{1}{2}((-1)^\ttn -3)}  \sum_{i\le t< \ttn_0}	   q^{t}  \,  z_{it}z_{it'}\otimes \sum_{j'\le r< \ttn_0}  q^{r} z_{r j}z_{r'j }
\end{align}
If $\ttn$ is even, the claim follows from \eqref{misma columna hasta ell0}. 
On the other hand, if $\ttn$ is odd,  from \eqref{misma fila hasta ell0} 
and \eqref{misma columna hasta ell0}, we conclude that 
\begin{align*}
q^{-2}\tilde{\mathtt{a}}_1 \tilde{\mathtt{a}}_2 - \,  \tilde{\mathtt{a}}_2 \tilde{\mathtt{a}}_1= 
 \tfrac{(q - q^{-1})}{(1+q)^2}          z_{i\ttn_0}^2\otimes  z_{\ttn_0j}^2.
\end{align*}
Using again   \eqref{qconmutacion tensor} it follows that
\[
q^{-2} \mathtt{a}_0\tilde{\mathtt{a}}_2 =\tilde{\mathtt{a}}_2 \mathtt{a}_0,\qquad\qquad q^{-2}\tilde{\mathtt{a}}_1 \mathtt{a}_0 =\mathtt{a}_0\tilde{\mathtt{a}}_1. 
\]
Then
\begin{align*}
q^{-2} \mathtt{a}_1 \mathtt{a}_2- \mathtt{a}_2\mathtt{a}_1&=q^{-2} (\beta_1\mathtt{a}_0+\tilde{\mathtt{a}}_1)(\beta_2\mathtt{a}_0+\tilde{\mathtt{a}}_2)- (\beta_2\mathtt{a}_0+\tilde{\mathtt{a}}_2)(\beta_1\mathtt{a}_0+\tilde{\mathtt{a}}_1)\\
&=(q^{-2}-1) \beta_1\beta_2\mathtt{a}_0^2 +q^{-2}\tilde{\mathtt{a}}_1 \tilde{\mathtt{a}}_2 - \,  \tilde{\mathtt{a}}_2 \tilde{\mathtt{a}}_1\\
&= \left((q^{-2}-1) \beta_1\beta_2 + \tfrac{(q - q^{-1})}{(1+q)^2}    \right)     z_{i\ttn_0}^2\otimes  z_{\ttn_0j}^2 =0,
\end{align*}
which completes the proof of the claim. 
Thus $\Delta(z_{ij}^N) = \mathtt{a}_1^N+\mathtt{a}_2^N+ \mathtt{a}_3^N$.
Moreover, observe that for each $s$, all the summands in $\mathtt{a}_s$ 
are elements that $q^2$-commute. Consequently
\[
\Delta(z_{ij}^N) = \mathtt{a}_1^N+\mathtt{a}_2^N+ \mathtt{a}_3^N = \sum_{\substack{k=i\\ k\neq \ttn_0}}^j z_{i k}^N\otimes z_{kj}^N + \tfrac{2}{(1+q)^N} z_{i\, \ttn_0}^N \otimes z_{\ttn_0\, j}^N,
\]
which implies that the result holds in case 3. 
The remaining cases are obtained  from the previous ones applying the 
antipode to the coproduct. Precisely, we have
\[
\Delta(z_{j'i'}^N) =\Delta(S(z_{ij}^N)) =\sum_{k=i}^j c^k_{ij} z_{j' {k'}}^N \otimes z_{{k'} i'}^N=\sum_{k=j'}^{i'} c^{{k'}}_{ij} z_{j' k}^N \otimes z_{k i'}^N=\sum_{k=j'}^{i'} c^{k}_{j'i'} z_{j' k}^N \otimes z_{k i'}^N, 
\]
and the result is proved.

\end{proof}

The next theorem gives an explicit formula for the dual quantum Frobenius
map for the positive quantum Borel algebras in type $B_{\theta}$ and 
$D_{\theta}$. In fact, in the $B_{\theta}$-case one may assume
that $q^{N/2}=1$, by taking $q^{2}$ instead of $q$. Hence, for the next $q^{N/2}=1$.

\begin{teor}\label{inclusion-case-bn}
Let $i,j\in \N$ with $i\leq j$.	The map $\iota: \cO(B^+)\to  \cO_q(B^{+})$ given by
		\[\iota(X_{ij})=		\begin{cases}
		\hfil 	\frac{(1+q)^N}{2}z^N_{ij},& \text{if }\, \,\ttn_0\in \tilde{\mathbb{I}}_{i,j},\\[.3em]
		z^N_{ij},& \text{otherwise},
			\end{cases}\] 
is a Hopf algebra monomorphism.
\end{teor}
\begin{proof}
Let $\gamma_{ij}=\tfrac{(1+q)^N}{2}$ for $\ttn_0 \in \tilde{\mathbb{I}}_{i,j}$ and $\gamma_{ij}=1$ otherwise. By Theorem \ref{teorema: coproducto potencia N}, we have
\[\Delta (\iota(X_{ij}))=\sum_{k=i}^j c^{k}_{ij}  \gamma_{ij} z_{ik}^N \otimes z_{kj}^N=\sum_{k=i}^j \frac{\gamma_{ij}}{\gamma_{ik}\gamma_{kj}}c^{k}_{ij} \iota(X_{ik}) \otimes\iota(X_{kj}).\] 	
For type $D_\theta$, all the coefficients $\gamma_{ij}, \gamma_{ik}, \gamma_{kj}, c^{k}_{ij}$ that appear in the previous sum are equal to $1$ and it follows that $\iota$ is a coalgebra map. Hence, we consider the case type $B_\theta$. If $ \ttn_0\notin \tilde{\mathbb{I}}_{i,j}$ then for each $i\le k\le j$ we have $c^{k}_{ij}=\gamma_{ij}=\gamma_{ik}=\gamma_{kj}=1$. On the other hand, suppose that $\ttn_0\in \tilde{\mathbb{I}}_{i,j}$ and $i\le k\le j$. If $k\neq \ttn_0$ then $c^{k}_{ij}=1$ and $\gamma_{ij}=\gamma_{ik}\gamma_{kj}$.  If $k=\ttn_0$ one has
$c^{\ttn_0}_{ij}\gamma_{ij}= q^{-N/2}=1$ and $\gamma_{i,\ttn_0}=1=\gamma_{\ttn_0, j}$. 
In any case we have
\begin{equation}\label{escalar-coprod}
c^{k}_{ij}\frac{\gamma_{ij}}{\gamma_{ik}\gamma_{kj}}=1, \qquad i\leq k\leq j,
\end{equation}
which proves that $\iota$ is a coalgebra map. 	Now we want to prove that $\iota$ is an algebra morphism.
Since $z_{ii}z_{i'i'}=1$ and $z_{\ttn_0 \ttn_0}=1$ then it follows the determinant relation 
\[\det \iota(X) = \prod_{k=1}^\ttn \iota(X_{kk})= \prod_{k=1}^{\ttn_0-1} \iota(X_{kk})\iota(X_{k' k'})=\prod_{k=1}^{\ttn_0-1} z_{ii}^Nz_{i'i'}^N =1\]

by \eqref{alg-so-relations}, it remains to show that 
$\sum_{k=i}^{j}\iota (X_{k'i'})\iota(X_{kj})=\delta_{ij}$ 

Since
\begin{align*}
	(S\otimes I)\circ \Delta (z_{ij}^N) &
	=\sum_{k=i}^j c^{k}_{ij} S(z_{ik}^N) \otimes z_{kj}^N
	=\sum_{k=i}^j c^{k}_{ij} (q^{\rho_j- \rho_j})^N\  z_{k', i'}^N\otimes z_{kj}^N
	=\sum_{k=i}^j c^{k}_{ij}   z_{k', i'}^N\otimes z_{kj}^N,
\end{align*}
and $\epsilon( z_{ij})=\delta_{ij}$, it follows from \eqref{escalar-coprod} that
\[
\delta_{ij}=\sum_{k=i}^j c_{ij}^{k}   z_{k', i'}^N z_{kj}^N=\sum_{k=i}^jc^{k}_{ij}\frac{\gamma_{ij}}{\gamma_{ik}\gamma_{kj}}    \iota (X_{k'i'})\iota(X_{kj})=\sum_{k=i}^j \ \iota (X_{k'i'})\iota(X_{kj}).
\]

\end{proof}

\subsection{$B_{\theta}$-type} 

For $1\le j\le 2\theta+1$, we consider $j':= 2\theta+2-j$. Let $\Pi=\{\alpha_{1}, \ldots, \alpha_\theta\}$ be the simple  roots of $B_\theta$. The group $\subg$ is the subgroup of $\Gamma$
generated by the elements $g_{\alpha_{i}}^{N}$ for all $i\in I$, and $\widehat{\subg}$ is the character group of $\subg$.
For each $i\in I$, let $\ell_i:=n_i/N$ and let $\xi_i$ be a primitive $\ell_i$-th root of the unity.
Then we may describe $\widehat{\subg}$ as the finite abelian 
group generated by the characters $\gamma_{j}$ such that 
\begin{align}\label{eq-gamma-bn}
	\gamma_j(g^N_{\alpha_i})=\xi_j^{(2-\delta_{i,\theta})\tfrac{a_{ij}}{2}} ,
	\quad \text{for all}\ i,j\in I.	
\end{align}
For each $i\in I$, consider the diagonal matrix 
$P_i\in B^{+}$ given by
\begin{align}\label{diag-matrices-bn}
	(P_i)_{jj}=\begin{cases}
		\xi_i,&  \text{if }\, \, j\in\{i+1,i'\} \text{ and } j \neq \theta+1,\\
		\xi^{-1}_i,& \text{if }\, \, j\in \{i, i'-1\}\text{ and } j \neq \theta+1,\\
		1,& \text{otherwise}.\\
	\end{cases}
\end{align}
As in \eqref{group-morphism}, but now with the diagonal matrices 
$P_i$ given by \eqref{diag-matrices-bn}, the map
\begin{align*}
	&\widehat{T}\hookrightarrow B^+,& &\gamma\mapsto P_{\gamma}=P_1^{s_1}\cdots P_\theta^{s_\theta},\qquad\text{ for all } \,\gamma=\gamma_{1}^{s_1}\ldots\gamma_{\theta}^{s_{\theta}}\in \widehat{T_M},&	
\end{align*}
is a group monomorphism.
Also,  the isomorphism 
$\psi:U_{q}^{Q}(\lieb^{+}) \to \cO_{q}(B^{+})$ (see \eqref{psi-bn} below) restricted to the commutative
subalgebra $Z^{\geq}:=Z^{\geq}_{Q}$ gives the isomorphism
$$
\psi: Z^{\geq} \to \cO(B^{+}),
\qquad K_{\alpha_{i}}^{\pm N}\mapsto (X_{i'-1,i'-1}X_{i,i})^{\pm},
\quad E_{i}^N \mapsto - (q-q^{-1})^{-N}(X_{i'-1,i'}X^{-1}_{i',i'}).
$$

\begin{obs}\label{rmk:iota-transpose-bn-case}
	The composition of $\psi: Z^{\geq} \to \cO(B^{+})$ and $\iota_{\mu}^{*}$
	gives a Hopf algebra epimorphism 
	$\iota_{\mu}^{*}\psi: Z^{\geq} \to \C^{\widehat{\subg}}$
	which satisfies  
		\[\iota_{\mu}^{*}\psi(K_{\alpha_i}^{N})(\gamma_j) = \xi_j^{\tfrac{s_ia_{ij}}{2}(2-\delta_{i,\theta})}.\] 
Also, by \eqref{eq-gamma-bn}, we have 
\[\varphi(g^N_{\alpha_i})(\gamma)=\prod_{j=1}^\theta \xi_j^{\tfrac{s_ia_{ij}}{2}(2-\delta_{i,\theta})}= t_{i'} t_{i'-1}^{-1} =\iota_{\mu}^{*}\psi(K_{\alpha_i}^{N})(\gamma),\qquad \gamma=\gamma_{1}^{s_{1}}\cdots \gamma_{\theta}^{s_{\theta}} \in
\widehat{\subg}. \]
Consequently $
g_{\alpha_i}^N=\varphi^{-1}\iota_{\mu}^{*}\psi(K_{\alpha_i}^{N})$,  for all $i\in I.$
\end{obs}

The  set $\Phi^+=\{\alpha_{i j}\,:\, 1\leq i\le \theta,\,\,    i< j\le i'-1 \}$ of positive roots is given in terms of the simple roots $\Pi=\{\alpha_{i}\}_{i=1}^{\theta}$ by:
\begin{align*}
	\alpha_{ij}&:=\alpha_{i}+\cdots+ \alpha_{j-1}, && \text{ if }\, i<j\le\theta+1,\\
	\alpha_{ij'}&:=\alpha_{i}+\cdots+ \alpha_{j-1}+2\alpha_{j}+\cdots+ 2\alpha_{\theta}, && \text{ if }\, \theta+1<j'\leq i'-1. 
\end{align*}
Given $i\in \mathbb{I}_{1,\theta}$ and $j\in \mathbb{I}_{i+1,i'-1}$, the root vectors $E_{\alpha_{ij}}:=E_{ij}$ are given by the recursive formula
\begin{align*}
	E_{ij}&=\begin{cases}
		E_{\alpha_{i}},& \text{if }\, j=i+1,\\
		E_{j-1}E_{i\,  j-1}-q^{-1} E_{i\, j-1}E_{j-1}, & \text{if }\, i+1<j\le \theta+1.\\ 
		E_{j'}E_{i\,  j-1}-q^{-1+\delta_{j,\theta+2}} E_{i\, j-1}E_{j'}, & \text{if }\, \theta+1<j\le i'-1.\\ 
	\end{cases}
\end{align*}
Assuming the previous notation, for each $1\leq i\le\theta$ and $i< j\leq i'-1$, set $g_i:=g_{\alpha_i}$ and consider
\begin{align*}
 x_{(ij)}&:=p_{\mu}(E_{ij}),& g_{(ij)}&:=g_ig_{i+1}\ldots g_{j-1},   && \text{ if }\, i<j\le\theta+1,  \\
&& g_{(ij')}&:=g_ig_{i+1}\ldots g_{j-1} g_{j}^2\ldots g_{\theta}^2, &&\text{ if }\, \theta+1<j'\le i'-1.\\
\end{align*}

 In this case, an explicit Hopf isomorphism 
 $\psi: U_{q^{1/2}}^{Q}(\lieb^{+})\to \cO_{q}(B_Q^{+})$  is given by:
\begin{align}\label{psi-bn}
	\psi(K_{\alpha_{i}}^{\pm 1})= (z_{i'-1\,  i'-1}\, z_{ii})^{\pm 1}, \qquad \quad \psi(E_{\alpha_i})= q^{2-\delta_{i,\theta}}(q-q^{-1})^{-1}z_{i'-1\, i'}\, z_{ii},
\end{align}
 and, by induction, we obtain that
 \begin{align}\label{iso-psi-bn}
 	\psi(E_{ij})=	\lambda_{ij}z_{j'i'}z_{ii}, 
 \end{align}
 where $\lambda_{ij}$ are the following scalars: 
 \begin{align*}
 	\lambda_{ij}=\begin{cases}
 		q^{j-i+1-\delta_{j,\theta+1}}(q-q^{-1})^{-1}, &\quad 1\leq i < j \leq \theta+1,\\[.2em]
 		(-1)^{\theta+1-j}q^{j-i-1/2}(q-q^{-1})^{-1}, &\quad  \theta+1<j\leq  i'-1.
 	\end{cases}
 \end{align*}

The next result gives the liftings of 
braided vector spaces of type $B_{\theta}$ in the Drinfeld-Jimbo case.
In particular, we recover using our method 
the formula for the liftings in type $B_{2}$ in \cite{BDR}
and in type $B_{3}$ in \cite{BGM}. Remember from \S \ref{subsec:DJ} that $q_1=q^{d_{1}}=q^2$.
 
\begin{teor}\label{lifting-case-bn} Let $\mu=(\mu_{\alpha})_{\alpha\in \Phi^{+}}=
	(\mu_{ij})_{(i,j)\in \mathbb{I}_{1,\theta}\times \mathbb{I}_{i+1,i'-1}}$
	be a family of root vector parameters. 
	For $1\le i<j\le \theta+1$ let 

\begin{equation}\label{escalares-rij-bn}\begin{aligned}
		r_{j'i'}&:=
		(q_1^2-1)^{N} \mu_{ij},\\
		r_{ji'}&:=(-1)^{\theta+1-j}\tfrac{1}{2}(q_1^2-1)^{N}(q_1-1)^{N}\mu_{ij'},  \text{ if $j\neq \theta+1$}\qquad   
\end{aligned}
\end{equation}
and extend this definition to all indices $r_{ij}$ so that the matrix $Q_r$ 
given in \eqref{unipotent} belongs to $\operatorname{SO}_{2\theta+1}$. 
Then, for all $1\le i<j\le \theta+1$, the
relations \eqref{eq-HA-4} are given by
\begin{align}\label{formula root vetor facil bn}
x_{(ij)}^{N}&=	\mu_{ij} \big(g^N_{(ij)}-1\big) -\sum_{s=i+1}^{j-1} r_{j's'} x^N_{(is)}.
\end{align}
If $j\neq \theta+1$ we have 
\begin{equation}\label{formula root vetor dificil bn}
\begin{aligned}
x_{(ij')}^{N}&= \mu_{ij'}\big(g^N_{(ij')}-1\big)-2(-1)^{\theta+1-j}(1+q_1)^{-N} \sum_{s=i+1}^{\theta}  r_{js'} x^N_{(is)}\\
&-2(-1)^{\theta+1-j}{(1+q_1)^{-N}}r_{j\,\theta+1} x^N_{(i\,\theta+1)}-\sum_{s=\theta+2}^{j'-1} (-1)^{s-j}  r_{js'} x^N_{(is)}, 
\end{aligned}
\end{equation}
\end{teor}

\begin{proof}
Let  $1\leq i\le \theta$ and $i< j$. We start by fixing the scalars
\begin{align}\label{scalar-zeta}
	\begin{aligned}
	\zeta_{ij}=\begin{cases}
		\hfil 	\frac{(1+q_1)^N}{2},& \text{if }\, \,\theta+1\in \tilde{\mathbb{I}}_{i,j},\\[.3em]
		1,& \text{otherwise}.
	\end{cases}
\end{aligned}
\end{align}
From \eqref{iso-psi-bn} follows that $\psi(E_{ij}^N)=\lambda^N_{ij}z_{j'i'}^Nz^{N}_{ii}$.  Note that, for all $\gamma\in \widehat{\subg}$, we have \vspace{.1cm}
	\begin{align*}
	\iota_{\mu}^{*}\big(\psi(E_{ij}^{N})\big)(\gamma) &= \lambda^N_{ij}	\iota_{\mu}^{*}(z^N_{j'i'}z^N_{ii})(\gamma)\\[.7em]
	    &=\lambda^N_{ij}\zeta^{-1}_{j'i'}	\iota_{\mu}^{*}\big(\zeta_{j'i'}z^N_{j'i'}z^N_{ii}\big)(\gamma)\\[.7em]
		&=\lambda^N_{ij}\zeta^{-1}_{j'i'}	\iota_{\mu}^{*}\big(X_{j'i'}X_{ii}\big)(\gamma) \qquad\quad \text{(by Theorem \ref{inclusion-case-bn})}\\
		&=\lambda^N_{ij}\zeta^{-1}_{j'i'}	\Big(  r_{j'i'}(\iota^*_{\mu}(X_{j'j'}X_{ii})-1) -\sum_{k=j'+1}^{i'-1}r_{j'k}\iota^*_{\mu}\big(X_{ki'}X_{ii}\big)\Big)(\gamma)\\
		&=\lambda^N_{ij}\zeta^{-1}_{j'i'}r_{j'i'}\big(g^N_{(ij)}-1\big) -\sum_{s=i+1}^{j-1}  \lambda^N_{ij}\zeta^{-1}_{j'i'} r_{j's'}\iota^*_{\mu}\big(X_{s'i'}X_{ii}\big)    (\gamma)\\
		&=\lambda^N_{ij}\zeta^{-1}_{j'i'}r_{j'i'}\big(g^N_{(ij)}-1\big) -\sum_{s=i+1}^{j-1}  \lambda^N_{ij}\lambda^{-N}_{is}\zeta^{-1}_{j'i'}\zeta_{s'i'} r_{j's'} x^N_{(is)}.
	\end{align*}
Since $x_{(ij)}^N=p_{\mu}(E^N_{ij})=\iota_{\mu}^{*}(\psi(E_{ij}^{N}))(\gamma)$, the result is proved.
\end{proof}

\begin{exemplo}[Case $B_2$]
	We would like to compare the relations obtained in Theorem \ref{lifting-case-bn} for the $B_2$ case  to the relations given in \cite[Thm 2.6]{BDR}. In this case, the set of positive roots is $\Phi^+=\{\alpha_{12},\alpha_{23},\alpha_{13},\alpha_{12'}\}$, where the simple roots are $\alpha_{12},\alpha_{23}$ and the longest root is $\alpha_{12'}$. Given the family of root vector parameters $\{\mu_{12},\mu_{23},\mu_{13},\mu_{12'}\}$, the explicit scalars in \eqref{escalares-rij-bn} are given by
	\begin{align*}
	r_{45}&=(q_1^2-1)^N \mu_{12},&
	r_{35}&=(q_1^2-1)^N \mu_{13},\\
	r_{34}&=(q_1^2-1)^N \mu_{23,}&
	r_{25}&=-\tfrac{1}{2}(q_1^2-1)^N(q_1-1)^N \mu_{12'}.&
	\end{align*}
	Then, from \eqref{formula root vetor facil bn}, it follows
	\begin{align*}
	x_{(12)}^N&= \mu_{12}(g_{(12)}^N-1), && x_{(13)}^N=\mu_{13}(g_{(13)}^N-1)-(q_1^2-1)^N \mu_{23}\, x_{(12)}^N.\\
	x_{(23)}^N&=\mu_{23}(g_{(23)}^N-1),
	\end{align*}
	On the other hand, compare these formulas with the relations for $y_i$ and $v$ in \cite[Thm 2.6]{BDR}. We point out that there is a typo in the relation of $v$ given in \cite[Thm 2.6]{BDR}. In fact, the correct one is 
	$v:=z^n+\mu_2(q^2-1)^nx_n^n-\lambda(g_1^ng_2^n-1)$. 
	By \eqref{formula root vetor dificil bn},
	\begin{align*}
	x_{(12')}^N&=\mu_{12'}(g_{(12')}^N-1)+2(1+q_1)^{-N} r_{24}\, x_{(12)}^N+2  (1+q_1)^{-N}r_{23}\, x_{(13)}^N.
	\end{align*}
	Since $
	Q_{r}=\begin{pmatrix}
	1 & r_{12} & r_{13}&  r_{14}& r_{15}\\
	0& 1 & r_{23}&  r_{24}& r_{25}\\
	0& 0& 1 & r_{34}&  r_{35}\\
	0& 0& 0& 1 & r_{45}&\\
	0& 0& 0& 0& 1\\
	\end{pmatrix}\in SO_{5}
	$ it follows from \eqref{alg-so-relations} that 
	\begin{align*}
	r_{23}&=-r_{34}=-(q_1^2-1)^N \mu_{23},&
	r_{24}&=-\tfrac{1}{2}r_{34}^2=-\tfrac{1}{2}(q_1^2-1)^{2N} \mu_{23}^2.&
	\end{align*}
	Thus
	\begin{align}\label{x12N-b2}
	x_{(12')}^N&=\mu_{12'}(g_{(12')}^N-1)- (q_1^2-1)^{N} (q_1-1)^{N}\mu_{23}^2\, x_{(12)}^N-2  (q_1-1)^N \mu_{23}\, x_{(13)}^N
	\end{align}
	This last relation coincides with the relation $w=0$ of \cite[Thm 2.6]{BDR}.
\end{exemplo}

\begin{exemplo}[Case $B_3$]
	Now we want to compare our results, with the ones obtained in \cite{BGM} for the $B_3$ case. First we have the scalars
	
	\begin{align*}
	r_{67}&=(q_1^2-1)^N \mu_{12},&
	r_{46}&=(q_1^2-1)^N \mu_{24},&
	r_{36}&=-\tfrac{1}{2}(q_1^2-1)^N(q_1-1)^N \mu_{23'},\\
	r_{56}&=(q_1^2-1)^N \mu_{23},&
	r_{57}&=(q_1^2-1)^N \mu_{13},&
	r_{37}&=-\tfrac{1}{2}(q_1^2-1)^N(q_1-1)^N \mu_{13'},\\
	r_{45}&=(q_1^2-1)^N \mu_{34},&
	r_{47}&=(q_1^2-1)^N \mu_{14},&
	r_{27}&=-\tfrac{1}{2}(q_1^2-1)^N(q_1-1)^N \mu_{12'}.
	\end{align*}
	In a similar way  as in the $B_2$ case, computing \eqref{formula root vetor facil bn} we obtain:
	\begin{align*}
	x_{(12)}^N&= \mu_{12}(g_{(12)}^N-1),&   x_{(24)}^N&=\mu_{14}(g_{(14)}^N-1)-(q_1^2-1)^N \mu_{34}\, x_{(23)}^N,\\
	x_{(23)}^N&= \mu_{23}(g_{(23)}^N-1),&    x_{(13)}^N&=\mu_{13}(g_{(13)}^N-1)-(q_1^2-1)^N \mu_{23}\, x_{(12)}^N.\\
	x_{(34)}^N&=\mu_{34}(g_{(34)}^N-1),& &
	\end{align*}
	Also, $x_{(23')}^N$ is computed  using \eqref{formula root vetor dificil bn} in a completely analogous way as in \eqref{x12N-b2}, obtaining
	\begin{align*}
	x_{(23')}^N&=\mu_{23'}(g_{(23')}^N-1)- (q_1^2-1)^{N} (q_1-1)^{N}\mu_{34}^2\, x_{(23)}^N-2  (q_1-1)^N \mu_{34}\, x_{(24)}^N.
	\end{align*}
	It remains to compute $x_{(13')}^N$ and $x_{(12')}^N$. Again, by \eqref{formula root vetor dificil bn},
	\begin{align*}
	x_{(13')}^N&=\mu_{13'}(g_{(13')}^N-1)+ 2(1+q_1)^{-N}\Big(r_{36} x_{(12)}^N+r_{35}x_{(13)}^N\Big)
	+ 2(1+q_1)^{-N} r_{34}x_{(14)}^N,\\
	x_{(12')}^N&=\mu_{12'}(g_{(12')}^N-1)- 2(1+q_1)^{-N}\Big(r_{26} x_{(12)}^N+r_{25}x_{(13)}^N\Big)
	- 2(1+q_1)^{-N} r_{24}x_{(14)}^N+r_{23}x_{(15)}^N.
	\end{align*}
	Since $Q_r=(r_{ij})\in SO_{7}$ we have
	\begin{align*}
	r_{35}&=-\tfrac{1}{2}r_{45}^2,&
	r_{34}&=-r_{45},&
	r_{26}&=-\tfrac{1}{2}r_{46}^2-r_{36}r_{56},\\
	r_{25}&=-r_{36}-r_{45}r_{46}+\tfrac{1}{2}r_{45}^2r_{56},&
	r_{24}&=-r_{46}+r_{45}r_{56}.&
	r_{23}&=-r_{56}.
	\end{align*}
	Using these identities in the relations of $x_{(13')}^N$ and $x_{(12')}^N$ we recover the scalars obtained in \cite[Lemma 3.7.]{BGM} and \cite[Lemma 3.12.]{BGM}.
\end{exemplo}

\subsection{$D_{\theta}$-type}
For $1\le j\le 2\theta$, let $j':= 2\theta+1-j$. 
Let $\Pi=\{\alpha_{1}, \ldots, \alpha_\theta\}$ be the simple  roots of $D_\theta$.
The  set of positive roots $\Phi^+=\{\alpha_{i j}\,:\, 1\leq i < \theta,\,\,    i< j\le i'-1 \}$ 
is given in terms of the simple roots $\Pi=\{\alpha_{i}\}_{i=1}^{\theta}$ by
\begin{align*}
\alpha_{ij}&:=\alpha_{i}+\cdots+ \alpha_{j-1}, && \text{ if }\, i<j\le\theta,\\
\alpha_{i\theta+1}&:=\alpha_{i}+\cdots+ \alpha_{\theta-2}+\alpha_{\theta}, && \text{ if }\, j=\theta+1,\\
\alpha_{ij'}&:=\alpha_{i}+\cdots+ \alpha_{j-1}+2\alpha_{j}+\cdots+ 2\alpha_{\theta-2}+ \alpha_{\theta-1}+ \alpha_{\theta}, && \text{ if }\, i <j \leq \theta-1. 
\end{align*}
Given $i\in \mathbb{I}_{1,\theta-1}$ and $j\in \mathbb{I}_{i+1,i'-1}$, 
the root vectors $E_{\alpha_{ij}}:=E_{ij}$ are given by the recursive formula
\begin{align*}
E_{ij}&=\begin{cases}
E_{\alpha_{i}},& \text{if }\, j=i+1,\\
E_{\alpha_{\theta}},& \text{if }\, (i,j)=(\theta-1,\theta+1),\\
E_{j-1}E_{i\,  j-1}-q^{-1} E_{i\, j-1}E_{j-1}, & \text{if }\, i+1<j \le \theta,\\ 
E_{\theta}E_{i\,  \theta-1}-q^{-1} E_{i\, \theta-1}E_{\theta}, & \text{if }\, i+2<j=\theta+1,\\ 
E_{j'}E_{i\,  j-1}-q^{-1} E_{i\, j-1}E_{j'}, & \text{if }\, \theta+1<j < i'.\\ 
\end{cases}
\end{align*}

With the previous notation, set  $g_i=g_{\alpha_i}$ and consider
\begin{align*}
x_{(ij)}&:=p_{\mu}(E_{ij}), &&i\in \mathbb{I}_{1,\theta-1},\quad j\in \mathbb{I}_{i+1,i'-1}\\
g_{(ij)}&:=g_ig_{i+1}\ldots g_{j-1},   &&\text{if }\, i<j\le\theta,  \\
g_{(i\theta+1)}&:=g_ig_{i+1}\ldots g_{\theta-2}g_{\theta},   &&\text{if }\, i<\theta,  \\
g_{(ij')}&:=g_ig_{i+1}\ldots g_{j-1} g_{j}^2\ldots g_{\theta-2}^2g_{\theta-1}g_{\theta}, &&\text{if }\, i<j\le \theta-1.
\end{align*}

In this case, an explicit Hopf homomorphism 
$\psi: U_{q}^{Q}(\lieb^{+})\to \cO_{q}(B_Q^{+})$  is given by:
\begin{align}\label{psi-dn}
\psi(K_{\alpha_{i}}^{\pm 1})= \begin{cases}
(z_{i'-1\,  i'-1}\, z_{ii})^{\pm 1},& \text{if } i<\theta,\\
(z_{i'-2\,  i'-2}\, z_{ii})^{\pm 1},& \text{if } i=\theta,
\end{cases}
\end{align}
\begin{align}\label{psi-parte2-dn}
	\psi(E_{\alpha_i})= \begin{cases}
		q^{2}(q-q^{-1})^{-1}z_{i'-1\, i'}\, z_{ii},& \text{if } i<\theta,\\
		q^{2}(q-q^{-1})^{-1}z_{i'-2\, i'}\, z_{ii},& \text{if } i=\theta.
	\end{cases}
\end{align}
By induction on $(j-i)$, we obtain that
\begin{align}\label{iso-psi-dn}
\psi(E_{ij})=	\lambda_{ij}z_{j'i'}z_{ii}, 
\end{align}
where $\lambda_{ij}$ are the following scalars: 
\begin{align*}
\lambda_{ij}=\begin{cases}
q^{j-i+1-\delta_{j,\theta+1}}(q-q^{-1})^{-1}, &\quad 1\leq i < j \leq \theta+1,\\[.2em]
(-1)^{\theta+1-j}q^{j-i}(q-q^{-1})^{-1}, &\quad  \theta+1<j\leq  i'-1.
\end{cases}
\end{align*}

The next result describes all the liftings of type $D_{\theta}$ for diagonal
braidings of Dinfeld-Jimbo type.  
The proof follows \emph{mutatis mutandis} from the proof of 
Theorem \ref{lifting-case-bn} with the only exception that 
$\zeta_{ij}$ given in \eqref{scalar-zeta} are equal to $1$, for all $i,j$.

\begin{teor}\label{lifting-case-dn} 
Let $\mu=(\mu_{\alpha})_{\alpha\in \Phi^{+}}=
(\mu_{ij})_{(i,j)\in \mathbb{I}_{1,\theta-1}\times \mathbb{I}_{i+1,i'-1}}$
be a family of root vector parameters. For $1\le i<j\le \theta+1$ let 
\begin{equation}\label{escalares-rij-dn}\begin{aligned}
		r_{j'i'}&:=
		(q^2-1)^{N} \mu_{ij},\\
		r_{ji'}&:=(-1)^{\theta+j}(q^2-1)^{N}\mu_{ij'},  \text{ if $j\neq \theta+1$}\qquad   
\end{aligned}
\end{equation}
and extend this definition to all indices $r_{ij}$ so that the matrix $Q_r$ 
given in \eqref{unipotent} belongs to 
$\operatorname{SO}_{2\theta}(\C)$. 
Then for all $1\le i<j\le \theta+1$  the relations \eqref{eq-HA-4} are given by
\begin{align}\label{formula root vetor facil dn}
x_{(ij)}^{N}&=	\mu_{ij} \big(g^N_{(ij)}-1\big) -\sum_{s=i+1}^{j-1} r_{j's'} x^N_{(is)},
\end{align}
If $j\neq \theta+1$ we have 
\begin{equation}\label{formula root vetor dificil dn}
\begin{aligned}
x_{(ij')}^{N}&= 
\mu_{ij'}\big(g^N_{(ij')}-1\big)-(-1)^{\theta+j} \sum_{s=i+1}^{\theta+1}  r_{js'} x^N_{(is)}-\sum_{s=\theta+2}^{j'-1} (-1)^{s-j}  r_{js'} x^N_{(is)}, 
\end{aligned}
\end{equation} 
\qed
\end{teor}

\begin{exemplo}[Case $D_5$]
As in the $B_\theta$ case, we would like to illustrate  how to compute the coefficients $r_{ij}$ for an specific root vector using Theorem \ref{lifting-case-dn}. Consider $G$ of type $D_5$ and the root $\alpha_{13'}=\alpha_1+\alpha_2+2\alpha_3+\alpha_4+\alpha_5$. Then, using \eqref{formula root vetor dificil dn}, we have
\begin{align}\label{defor-d5-root}
x_{(13')}^N&= \mu_{13'}\big(g^N_{(13')}-1\big)- \sum_{s=2}^{6}  r_{3s'} x^N_{(1s)}- r_{34} x^N_{(17)}.
\end{align}
Set  $\eta=(q^2-1)^{N}$. Given a family of vector parameters 
$(\mu_{\alpha})_{\alpha\in \Phi^{+}}$, by \eqref{escalares-rij-dn}, we have 
\begin{align*}
r_{78}&=\eta \mu_{34},&
r_{67}&=\eta\mu_{45},&
r_{68}&=\eta\mu_{35},&
r_{57}&=\eta\mu_{45'},\\
r_{58}&=\eta\mu_{35'}&
r_{48}&=-\eta\mu_{34'},&
r_{39}&=\eta\mu_{23'}.
\end{align*}
Using the fact that $Q_r=(r_{ij})\in \operatorname{SO}_{10}$, it follows that
\begin{align*}
r_{35}&=\eta^2\mu_{34}\mu_{45} -\eta\mu_{35},&
r_{34}&=-\eta \mu_{34},\\
r_{36}&=\eta^2\mu_{34}\mu_{45'} -\eta\mu_{35'},&
r_{39}&=\eta \mu_{23'},\\
r_{38}&=\eta^2\mu_{34'}\mu_{34}-\eta^2\mu_{35'}\mu_{35} \\
r_{37}&=\eta\mu_{34'}-\eta^2\mu_{35}\mu_{45'}-\eta^2\mu_{35'} \mu_{45}+\eta^3\mu_{34}\mu_{45}\mu_{45'}.
\end{align*}
Hence, we have the explicitly description for the relation  \eqref{defor-d5-root}. 
In a similar way, the coefficients $r_{ij}$ involved in the formulas \eqref{formula root vetor facil dn} and \eqref{formula root vetor dificil dn} can be computed
for the others positive roots. 
\end{exemplo}

\end{document}